\documentclass[11pt]{amsart}

\usepackage{amsmath}
\usepackage{amssymb}
\usepackage{graphicx}
\usepackage{xcolor}

\newtheorem{theorem}{Theorem}[section]

\newtheorem{lemma}[theorem]{Lemma}
\newtheorem{proposition}[theorem]{Proposition}
\theoremstyle{definition}

\newtheorem{remark}[theorem]{Remark}

\newcommand{\Pro}{\mathbf{P}}

\newcommand{\E}{\mathbb{E}}
\newcommand{\R}{\mathbb{R}}

\newcommand{\al}{\alpha}
\newcommand{\eps}{\varepsilon}

\newcommand{\bL}{\mathbb{L}}
\newcommand{\bT}{\mathbb{T}}

\newcommand{\rma}{\mathrm{a}}

\newcommand{\mE}{\mathbf{E}}

\newcommand{\cE}{\mathcal{E}}

\newcommand{\cO}{\mathcal{O}}
\newcommand{\cL}{\mathcal{L}}

\newcommand{\cG}{\mathcal{G}}
\newcommand{\cV}{\mathcal{V}}
\newcommand{\cB}{\mathcal{B}}

\newcommand{\cM}{\mathcal{M}}

\newcommand{\cH}{\mathcal{H}}
\newcommand{\cA}{\mathcal{A}}
\newcommand{\cC}{\mathcal{C}}

\newcommand{\cI}{\mathcal{I}}

\newcommand{\ocA}{\overline{\mathcal{A}}}

\newcommand{\cS}{\mathcal{S}}

\newcommand{\trace}{\mbox{Tr}}

\newcommand{\aeff}{\rma^{{\rm eff}}}
\newcommand{\aeffk}{\rma^{k,{\rm eff}}}
\newcommand{\urma}{\underline{\rma}}
\newcommand{\urmak}{\urma^{k,{\rm eff}}}

\makeatletter
\newsavebox{\@brx}
\newcommand{\llangle}[1][]{\savebox{\@brx}{\(\m@th{#1\langle}\)}%
  \mathopen{\copy\@brx\mkern2mu\kern-0.9\wd\@brx\usebox{\@brx}}}
\newcommand{\rrangle}[1][]{\savebox{\@brx}{\(\m@th{#1\rangle}\)}%
  \mathclose{\copy\@brx\mkern2mu\kern-0.9\wd\@brx\usebox{\@brx}}}
\makeatother

\newcommand{\blla}{ \llangle[\Big]}
\newcommand{\brra}{ \rrangle[\Big]}
\newcommand{\lla}{ \llangle[]}
\newcommand{\rra}{ \rrangle[]}

\numberwithin{equation}{section}
\title[Asymptotic decomposition of solutions]{%Random parabolic operators with oscillating coefficients: exact asymptotic decomposition of the solutions}
Asymptotic decomposition of solutions to parabolic equations with a random microstructure}

\author[M. Kleptsyna]{Marina Kleptsyna}

\address[M. Kleptsyna]{Le Mans Universit\'e,
Laboratoire Manceau de Math\'ematiques,
Avenue Olivier Messiaen,
72085 Le Mans, Cedex 9, France.}
\email{{\tt marina.kleptsyna@univ-lemans.fr}}

\author[A. Piatnitski]{Andrey Piatnitski}

\address[A. Piatnitski]{The Arctic University of Norway, campus Narvik,  P.O.Box 385,
8505 Narvik, Norway\ \ \  and \ \ \   Institute for Information Transmission Problems of RAS, 19, Bolshoy Karetny per., Moscow 127051,
Russia}
\email{{\tt apiatnitski@gmail.com}}

\author[A. Popier]{Alexandre Popier}

\address[A. Popier]{Le Mans Universit\'e, Laboratoire Manceau de Math\'ematiques, Avenue Olivier Messiaen, 72085 Le Mans,
Cedex 9, France}
\email{{\tt alexandre.popier@univ-lemans.fr}}



\begin{document}

\begin{abstract}
We consider a Cauchy problem for a divergence form second order parabolic operator with rapidly oscillating coefficients that are periodic in spatial variables and random stationary ergodic in time. As was proved in \cite{JKO_1} and \cite{KP_1995} in this case the homogenized operator is deterministic.

We obtain the leading terms of the asymptotic expansion of the solution, these terms being deterministic functions, and show that a properly renormalized
difference between the solution and the said leading terms converges to a solution of some SPDE.

%We obtain an asymptotic development of the solution, representing it as a sum of powers of $\eps$-scaled deterministic functions, such that the renormalized difference converges to the solution of some SPDE.
\end{abstract}

\maketitle

%%%%%%%%%%%%%%%%%%%%%%%%%%%%%%%%%%%%%%%%%%%%%%%%%%%%%%%%%%%%%%%%%%%%%%%%%%%%%%%%

\tableofcontents

\section{Introduction}\label{s_intro}
This work is devoted to obtaining an exact asymptotic development (as $\eps\to 0$) of
solutions to the following Cauchy problem
\begin{equation}\label{ori_cauch}
\left\{ \begin{array}{l}
\displaystyle
\frac\partial{\partial t}u^\eps=\mathrm{div}\Big[a\Big(\frac x\eps,\xi_{\frac t{\eps^\alpha}}\Big)\nabla u^\eps\Big]= \cA^\eps u^\eps \qquad
\hbox{in }\mathbb R^d\times(0,T]\\[4mm]
u^\eps(x,0)=\imath(x).
\end{array}\right.
\end{equation}
Here $\eps$ is a small positive parameter that tends to zero, $\alpha>0$, $\alpha\neq 2$,
and $a(z,s)$ is a positive definite matrix whose entries are periodic in the $z$ variable and random stationary ergodic in $s$.

It is known (see \cite{JKO_1, KP_1995}) that this problem admits homogenization and that the homogenized operator
is deterministic and has constant coefficients.  The homogenized Cauchy problem takes the form
\begin{equation}\label{eq:u_0}
\left\{ \begin{array}{c}
\displaystyle
\frac\partial{\partial t}u^0=\mathrm{div}(\aeff \nabla u^0)\\[3mm]
u^0(x,0)=\imath(x).
\end{array}\right.
\end{equation}
The formula for the effective matrix $\aeff$ is given in \eqref{eq:def_aeff} in Section 2 (see also \cite{KP_1995}).

%The goal of this paper is to study the limit behaviour of the difference $u^\eps-u^0$, as $\eps\to0$.

In the existing literature there is a number of works that deal with homogenization of random parabolic problems.
The results obtained in \cite{Ko78} and \cite{PV80} for random divergence form elliptic operators also apply
to the parabolic case. In the presence of large lower order terms the limit dynamics might remain random
and show diffusive or even more complicated behaviour.  The papers \cite{CKP_2001}, \cite{PP}, \cite{KP_3}
focus on the case of time dependent parabolic operators  with periodic in spatial variables and random
in time coefficients. The fully random case has been studied in \cite{PP_1}, \cite{Ba_1}, \cite{Ba_2},
\cite{HPP}.

One of the important aspects of homogenization theory is  estimating the rate of convergence.
For random operators the first estimates have been obtained in \cite{Yu80}. Further important progress in this direction
was achieved in the recent works \cite{GO12}, \cite{GM12}.

Problem (\ref{ori_cauch}) in the case of diffusive scaling $\alpha=2$ was studied in our previous work \cite{KPP_2015}.
It was shown that, under proper mixing conditions, the difference $u^\eps-u^0$ is of order $\eps$, and that
the normalized difference $\eps^{-1}(u^\eps-u^0)$ after subtracting an appropriate corrector, converges in law
to a solution of some limit SPDE.

However, for positive $\alpha\neq 2$, the situation becomes much more intriguing. The random solution $u^{\eps}$ admits an asymptotic decomposition as $\eps\to 0$, that is a sum of terms each of which scales as a power of~$\eps$. Our main result, Theorem~\ref{thm:main_result} below, provides such a description; we will start by its brief description.

Heuristically speaking, this theorem can be thought of as follows. First, as $\eps\to 0$, the \emph{random} solution $u^{\eps}$ converges to the \emph{deterministic} limit $u^0$. Considering the difference $u^{\eps}-u^0$ and dividing it by an appropriate power of $\eps$, one can pass to the limit; if the limit is deterministic, we iterate this procedure
%of development into series by powers of $\eps$ continues,
until at some stage we reach a \emph{random} limit. Returning to $u^\eps$, we obtain its expansion being a sum of terms of increasing order of $\eps$, with all but the last terms being deterministic, and the random term coming with the scaling factor $\eps^{\alpha/2}$.

A first remark here is that the powers of $\varepsilon$ appearing in the expansion are not all integer,
%only in the form of $\eps^k$ with integer~$k$,
but also of the form $\eps^{\delta k}$, where $\delta=|\alpha-2|$.

An important observation  is that for $\alpha>2$ the final power of $\eps$, the one associated to the random limit, is greater than $1$. And that looks very surprising (even next to impossible) due to the following handwaving argument. The solution to the Cauchy problem at some $\eps$ is naturally connected to the diffusion process on a compact at the $\eps^2$-rescaled time $t/\eps^2$. Now, if we were considering behaviour of the averages of the type
\[
\int_0^t g(s, x_{s/\eps^2}) ds,
\]
where $x_s$ is a sufficiently well-mixing ergodic process and $g$ is a function, we would have convergence to the integral of the space average $\int_0^t \bar{g}(s)\, ds$ (where $\bar{g}(s)$ is the expectation of $g(s,\cdot)$ with respect to the stationary distribution of $x_{\cdot}$) with the Central Limit Theorem-governed speed $\sqrt{\eps^2}=\eps^1$. In our problem, the natural rescaling is
$(x_{t/\eps^2}, \xi_{t/\eps^\alpha})$, so one would naturally expect that the randomness occurs at the scaling $\eps^{1\wedge \frac{\alpha}{2}}$. However, it is not the case: for $\alpha>2$ the randomness occurs not at the power $\eps^1$, but still at the power~$\eps^{\alpha/2}$.

\subsection{Organization  of the paper}

The paper is organized as follows. In Section \ref{s_1} we introduce the studied problem and provide all the assumptions. Then we formulate the main result of the paper (Theorem \ref{thm:main_result}) that reads differently depending on whether $\al < 2$, or  $2<\al < 4$, or $\al \geq 4$. We also define the numerous correctors and auxiliary problems required to state the main result.

In Section \ref{sect:formal_exp} we give the formal expansion $\cE^\eps$ of $u^\eps$. Formally we define the function $\cE^\eps$ such that
$$
R^\eps(x,t) =\eps^{-\alpha/2}\left[ u^\eps(x,t)-\cE^\eps (x,t) \right]
$$
converges in law in a suitable functional space to some non trivial and random limit $q^0$. The main result of this section is given by Propositions \ref{prop:formal_exp_al<2} and \ref{prop:dev_al}.
%To sum up, this formal expansion of $u^\eps$ gives the sequences of constants $\aeffk$ and $\urmak$ and of smooth functions $v^k$ and $u^k$.
Constructing the formal expansion of $u^\eps$ gives rise to the  sequences of deterministic constants $\aeffk$ and $\urmak$ and smooth functions $v^k$ and $u^k$ that characterize the leading part of the expansion.
%Note that the functions $w^k$ in the definition of $v^k$ will be left as \textit{free parameters} in this first part.
The normalized difference  $R^\eps$ contains asymptotically large parameters %(as $\eps$ tends to zero),
both in its dynamics and, for $\al > 2$, in its initial condition \eqref{eq:init_cond_R_eps}.
%Therefore $R^\eps$ is split into five terms $R^\eps = r^\eps +\check r^\eps +  \hat r^\eps + \widetilde r^\eps + \rho^\eps $ such that:
The function $R^\eps$ can be represented as the sum of the following five terms: $R^\eps = r^\eps +\check r^\eps +  \hat r^\eps + \widetilde r^\eps + \rho^\eps $, where the limit behaviour of each of  these terms depends on whether $\alpha<2$, or $2<\alpha<4$, or $\alpha\geq 4$. We will show that
\begin{itemize}
\item The term  $r^\eps$ contains a  martingale with a large factor $\eps^{1-\al}$ if $\al < 2$ and $\eps^{-1}$ if $\al > 2$. This term converges to zero for $\al<2$ (Proposition \ref{prop:limit_ident_al<2}) and to $q^0$ for $\al > 2$ (Proposition \ref{prop:limit_ident_al>2}). 
\item $\check r^\eps$ appears only for $\al < 2$ and converges to $q^0$ (see Proposition~\ref{prop:weak_conv_al<2}).
\item $\hat r^\eps$ exists only when $\al > 2$ and converges in a weak topology to zero.
\item $ \widetilde r^\eps$ converges in a strong topology to zero.
\item The last term $\rho^\eps$ is required if  $\al > 2$.  This term compensates asymptotically growing initial condition of $R^\eps$. We prove that it also converges to zero.
\end{itemize}
Let us emphasize that in this section the dimension $d$ plays no role and some terms in $\cE^\eps$ may be negligible depending on the value of $\al$.

Section \ref{sect:conv_sing_mart_part} focuses on the proof of the convergence of $r^\eps$. This term contains, at least for $\al >1$, a martingale with an asymptotically growing parameter. After proper choice of a number of free parameters we  show that the contribution of this term  weakly converges to zero if $\al < 2$ or to the limit $q^0$ if $\al > 2$.
%Roughly speaking, we need $w^k$ to obtain a uniform bound in $H^1(\R)$ of the indefinite integral of $R^\eps$.
Here we widely use the fact that $d=1$.

In Section \ref{sect:rest_init_cond} the trouble comes from the initial condition for $R^\eps$ when $\al > 2$. Here we construct an asymptotic expansion of the corresponding terms (see Eq. \eqref{eq:expansion_varrho}) and study their properties (Lemmata \ref{lmm:behaviour_beta_0}, \ref{lmm:behaviour_beta_ell} and \ref{lmm:behaviour_m_ell}). We prove that under our particular choice of the initial condition for the terms of the expansion
%, it is possible to define the constants $\cI_k$ such that the initial condition of $R^\eps$ does not contribute in the limit equation, that is
$\rho^\eps$ converges to zero in a strong sense (Proposition \ref{prop:rest_init_cond}) and does not contribute in the limit equation. In this section the dimension $d$ could be any positive integer.

To summarize, the conclusion of Theorem \ref{thm:main_result} follows from
\begin{itemize}
\item For $\al < 2$: Propositions \ref{prop:formal_exp_al<2}, \ref{prop:weak_conv_al<2} and \ref{prop:limit_ident_al<2}.
\item For $\al > 2$: Propositions \ref{prop:dev_al}, \ref{prop:limit_ident_al>2} and \ref{prop:rest_init_cond}.
\end{itemize}
%Finally, in the Appendix we postpone some straightforward but cumbersome computations.

 \section{Problem setup and main result}\label{s_1}
%------------------

In this section we provide all the assumptions for Problem \eqref{ori_cauch}, introduce some notations and formulate the main results.
\subsection{Assumptions }
Concerning the coefficients of Equation \eqref{ori_cauch}, we assume that:
\begin{enumerate}%[label=\textbf{(a\arabic*)}]
\item \label{a1} The initial condition $\imath$ belongs to space\footnote{In fact, this condition can be essentially relaxed (see Remark \ref{r_regu_ini}).} $C_0^\infty(\mathbb R)$.
\item \label{a2} Function $a$ is periodic in $z$ and smooth in both variables $z$ and $y$. Moreover, for each $N>0$ there exists $C_N>0$ such that
$$
\|a\|_{C^N(\mathbb T\times\mathbb R^n)}\leq C_N.
$$
Here and in what follows we identify periodic functions with functions on the torus $\mathbb T$.
\item \label{a3} Coefficient $a=a(z,y)$ satisfies the uniform ellipticity condition: there exists $\lambda > 0$ such that for any $z \in \mathbb T$ and any $y \in \mathbb R^n$:
$$
\lambda^{-1} \leq a(z,y)\leq \lambda;
$$
\end{enumerate}
The random noise $\xi=(\xi_s,\ s\geq 0)$ is a diffusion process in $\mathbb R^n$ with a generator
$$
\mathcal{L}= \frac{1}{2} \trace [q(y)D^2]+ b(y).\nabla
$$
($\nabla$ stands for the gradient, $D^2$ for the Hessian matrix).
Moreover we suppose that matrix-function $q$ and vector-function $b$ possess the following properties:
\begin{enumerate}%[label=\textbf{(a\arabic*)}]
\setcounter{enumi}{3}
\item \label{a4} The matrix $q=q(y)$  satisfies the uniform ellipticity condition: there exists $\lambda>0$ such that
$$
\lambda^{-1}|\zeta|^2\leq q(y)\zeta\cdot\zeta\leq \lambda|\zeta|^2, \quad y,\,\zeta\in\mathbb R^n.
$$
Moreover there exists a matrix $\sigma=\sigma(y)$ such that $q(y) = \sigma^*(y)\sigma(y)$.
\item \label{a5} The matrix function $\sigma$ and vector function $b$ are smooth, that is for each $N>0$ there exists $C_N>0$ such that
$$
\|\sigma\|_{C^N(\mathbb R^n)}\leq C_N,\qquad \|b\|_{C^N(\mathbb R^n)}\leq C_N.
$$
\item \label{a6} The following inequality holds for some $R>0$ and $C_0>0$ and $p>-1$:
 $$
 b(y)\cdot y\leq -C_0|y|^p \quad \hbox{for all } y\in\{y\in\mathbb R^n\,:\,|y|\geq R\}.
 $$
\end{enumerate}
We say that Condition {\bf (A)} holds if \ref{a1} to \ref{a6} are satisfied.

Let us recall that according to \cite{pard:vere:01, pard:vere:03} under conditions \ref{a4} and \ref{a6} a diffusion process $\xi_\cdot$ with generator $\mathcal{L}$ has an invariant measure in $\mathbb R^n$ that has a smooth density $p=p(y)$. For any $N>0$ it holds
$$(1+|y|)^N p(y)\leq C_N$$
with some constant $C_N$. The function $p$ is the unique up to a multiplicative constant bounded solution  of
the equation $\mathcal{L}^* p=0$; here $\mathcal{L}^*$ denotes the formally adjoint operator. We assume that the process $\xi$ is stationary and distributed with the density $p$. In the rest of the paper
\begin{itemize}
\item $\overline  f$ denotes the mean w.r.t. the invariant measure $p$;
\item $\langle f \rangle$ is the mean on the torus $\mathbb T$.
\end{itemize}
$a^\eps$ and $\mathfrak{a}^\eps$ denote the matrices:
$$a^\eps=a\left( \frac{x}{\eps},\xi_{\frac{t}{\eps^\alpha}}\right),\quad \mathfrak{a}^\eps = a\left( z,\xi_{\frac{t}{\eps^{\alpha -2}}}\right).$$

From \cite{KP_1995} under Condition {\bf (A)}, we know that $u^\eps$ converges in probability in the space
$$V_T=L^2_w(0,T; H^1(\R))\cap C(0,T;L^2_w(\R))$$
to $u^0$, the solution of \eqref{eq:u_0}
\begin{equation*}
u^0_t = \mbox{div }( \aeff \nabla u^0 ),\quad u^0(x,0) = \imath(x),
\end{equation*}
where the effective matrix $\aeff$ is defined by:
\begin{equation}\label{eq:def_aeff}
\aeff = \overline{ \langle a + a \nabla_z \chi^{0} \rangle}.%  = \langle \overline{a + \left( a \nabla \chi^{1} + \nabla (a \chi^{1}) \right)}\rangle  ,
\end{equation}
The symbol $w$ in the definition of $V_T$ means that the corresponding space is endowed with its weak topology.

The corrector $\chi^0$ is defined in different ways depending on whether $\alpha < 2$ (Equation \eqref{eq:aux_0_dif_al<2}), or $\alpha > 2$ (Equation \eqref{eq:def_chi_0_al>2}), thus the function $u^0$ is not the same for $\al > 2$ and $\al < 2$. More precisely, for $\al< 2$, the function $\chi^0=\chi^0(z,y)$ is a periodic solution of the equation
\begin{equation}\label{eq:aux_0_dif_al<2}
\mathrm{div}_z\big(a(z,y)\nabla_z \chi^0(z,y)\big)=-\mathrm{div}_z a(z,y);
\end{equation}
here $y\in\mathbb R^n$ is a parameter. We choose an additive constant in such a way that
\begin{equation}\label{eq:norm_cond_chi0}
\int_{\mathbb T}\chi^0(z,y)\,dz=0.
\end{equation}
When $\al > 2$, the corrector $\chi^0$ is the solution of
\begin{equation} \label{eq:def_chi_0_al>2}
\bar \cA \chi^0 = \mbox{div } \left[ \bar a \nabla \chi^0\right] = -  \bar a_x
\end{equation}
where $\bar a$ is the mean value of $a$ w.r.t. $y$:
$$\bar a (z) = \int_{\R^n} a(z,y) p(y) dy.$$
It is known that matrix $\aeff$ is positive definite in both cases (see, for instance, \cite{CKP_2001, KP_1995}).

\subsection{Main result}

In the rest of the paper we denote
\begin{itemize}
\item $\delta = |\alpha-2| >0$,
\item $J_0=\lfloor\frac\alpha{2\delta}\rfloor+1$, where $\lfloor\cdot\rfloor$ stands for the integer part,
\item $J_1 = \lfloor\frac\alpha{2}\rfloor$.
\end{itemize}
Let us remark that: $\min(\delta +1 , J_1+1 ,\delta J_0 ) >\alpha /2.$ For technical reasons, we also use $N_0= 2J_0+2$.

For any $\al \neq 2$, we construct a sequence of constants $\aeffk$, $k\geq 1$, and a sequence of functions $u^j$, $j\geq 1$,  as solutions of problems
\begin{equation}\label{eq:def_u_k}
\begin{array}{c}
\displaystyle
\frac\partial{\partial t}u^j = \mathrm{div}(\aeff\nabla u^j)+\sum\limits_{k=1}^{j} \aeffk\frac{\partial^2}{\partial x^2} u^{j-k} + w^j
\end{array}
\end{equation}
with initial condition $u^j(x,0) = 0$. The definition of the sequence $\aeffk$ depends on the sign of $\al - 2$ (see Eq. \eqref{eq:def_aeffk_al<2} and \eqref{eq:def_aeffk_al>2}). The smooth functions $w^j$ are defined recursively and depend also on the sign of $\al - 2$ (see Eq. \eqref{eq:def_w_k_al<2} and \eqref{eq:def_w_k_al>2} ). They are used to control the large martingale terms (see Section \ref{sect:conv_sing_mart_part} and Eq. \eqref{eq:mart_problem}).

For $\al > 2$ (that is $J_1\geq 1$), to obtain the desired convergence we need a second sequence of auxiliary functions with a different scaling. We construct two other sequences of constants $(\urmak)_{k\geq 1}$ and $(\cI_k)_{k\geq 1}$ and introduce $v^0 = u^0$ and then
\begin{equation} \label{eq:def_v_k_al>2}
\frac\partial{\partial t} v^j = \aeff \frac{\partial^2}{\partial x^2}v^j + S^j,\quad v^j(x,0)=\cI_j \partial_x^j u^0(x,0),
\end{equation}
with
\begin{equation*}% \label{eq:S_k}
S^j(x,t) = \sum_{k=1}^{j} \urmak(\partial^{k+2}_x v^{j-k})\qquad\hbox{for } j\geq 1\ .
\end{equation*}
%The sequence of constants $(\cI_k)_{k\geq 1}$ are such that $\cI_0=1$, $\cI_1=0$, and defined by \eqref{def:I_k_al>2} for $k\geq 0$. In the expansion of $u^\eps$, we need to take into account the initial value of the remainder. For $\al < 2$, this additional term is negligible. But for $\al > 2$, it contains negative powers of $\eps$ and thus it should be controlled. This is the role of this sequence $\cI_k$.
Finally the correctors $\chi^j$ are defined by \eqref{eq:def_chi_k_al>2}. These additional correctors $v^j$ and $\chi^j$ are used in particular to control the initial value of the remainder (see Section \ref{sect:rest_init_cond}).

Our main result is the following.
\begin{theorem}\label{thm:main_result}
Under Condition {\bf (A)}, there exists a non-negative constant $\Lambda$ (defined by \eqref{eq:def_Lambda_al<2} for $\al < 2$ and \eqref{eq:def_Lambda_al>2} for $\al > 2$), such that the normalized functions
\begin{eqnarray*}
&& q^\eps= \eps^{-\al/2} \Bigg\{ u^\eps (x,t) - u^0(x,t) - \sum_{k=1}^{J_0} \eps^{k\delta} u^k(x,t)   \\
& &\qquad  \qquad \left. - \sum_{k=1}^{J_1} \eps^{k} \left[ v^k(x,t)  +  \sum_{\ell=1}^k \chi^{\ell-1}\left( \frac{x}{\eps}\right) \partial^{\ell}_x v^{k-\ell} \left(x,t\right) \right]   \right\}
\end{eqnarray*}
converge in law, as $\eps\to0$, in $L_w^2(\mathbb R\times(0,T))$ to the unique solution of the following SPDE
 \begin{equation}\label{eq:eff_spde}
\begin{array}{c}
\displaystyle
dq^0=\mathrm{div}(\aeff\nabla q^0)\,dt+(\Lambda^{1/2}) \left( \frac{\partial^2}{\partial x^2}u^0\right) \,dW_t
\\[3mm]
q^0(x,0)=0;
\end{array}
\end{equation}
driven by a standard one-dimensional Brownian motion $W$.
\end{theorem}

Note that the values of $J_0=\lfloor\frac\alpha{2\delta}\rfloor+1$ and of $\alpha$ are related as follows:
\begin{itemize}
\item $\al < 2$ and
$$2 - \frac{2}{2J_0-1}\leq\alpha < 2- \frac{2}{2J_0+1};$$
\item $2 < \al \leq 4$, $J_0\geq 2$ and
$$2+ \frac{2}{2J_0-1}< \alpha \leq 2+\frac{2}{2J_0-3};$$
\item $\al > 4$ and $J_0=1$.
\end{itemize}
In other words, $J_0$ becomes large as $\alpha$ approaches 2.
Let us specify more precisely what happens for $q^\eps$ in the four cases: $\al < 2$, $\al < 4$, $\al = 4$ and $\al > 4$.
\begin{itemize}
\item $\alpha < 2$: $J_1=0$ and $q^\eps$ can be written as follows:
$$q^\eps=  \eps^{-\al/2} \left\{ u^\eps (x,t) - u^0(x,t) - \sum_{k=1}^{J_0} \eps^{k\delta} u^k(x,t) \right\}.$$
Here the sequence $v^k$ is not involved.
\item $2 < \alpha < 4$: $J_1=1$ and
\begin{eqnarray*}
&& q^\eps= \eps^{-\al/2} \Bigg\{ u^\eps (x,t) - u^0(x,t) - \sum_{k=1}^{J_0} \eps^{k\delta} u^k(x,t)    \\
& &\qquad  \qquad  \qquad -  \eps \left[ v^1(x,t)  + \chi^0\left( \frac{x}{\eps}\right) \partial_x u^{0} \left(x,t\right) \right]  \Bigg\}
\end{eqnarray*}
\item $\alpha = 4$: $J_0 = 1 $ and $J_1=2$. Thereby $q^\eps$ becomes
\begin{align*}
q^\eps&=\eps^{-2} \Bigg\{ u^\eps (x,t) - u^0(x,t) - \eps  \left[ v^1(x,t)  +   \chi^0\left( \frac{x}{\eps}\right) \partial_x u^{0} \left(x,t\right) \right] \\
& \qquad \quad - \eps^2 \left[ u^1(x,t) +  v^2(x,t) + \chi^0\left( \frac{x}{\eps}\right)\partial^{2}_x u^{0}(x,t)
+\chi^1\left( \frac{x}{\eps}\right) \partial_x v^{1} \left(x,t\right) \right] \Bigg\} .
\end{align*}
$\alpha =4$ is a kind of critical value, since here $u^1$ and $v^2$ coexist.
\item $\alpha > 4$: $J_0 = 1$ and for any $m\geq 2$
$$ J_1=m \Leftrightarrow 2m \leq \alpha \leq 2(m+1).$$
Hence
\begin{eqnarray*}
&& q^\eps = \eps^{-\al/2} \Bigg\{ u^\eps (x,t) - u^0(x,t) -\eps^{\delta} u^1(x,t) \\
& &\qquad  \qquad \left. - \sum_{k=1}^{J_1} \eps^{k} \left[ v^k(x,t)  +  \sum_{\ell=1}^k \chi^\ell\left( \frac{x}{\eps}\right) \partial^{\ell}_x v^{k-\ell} \left(x,t\right) \right] \right\} .
\end{eqnarray*}

\end{itemize}

\begin{remark}[When $\al=2$]
In \cite{KPP_2015}, we prove that
$$
q^\eps(x,t):= \frac{u^\eps(x,t)-u^0(x,t)}{\eps}- \chi \big(\frac{x}{\eps},\frac{t}{\eps^2}\big)\cdot\nabla u^0(x,t)
$$
converges to the SPDE
$$d q^0=\left[\mathrm{div}\Big(a^{\rm eff}\nabla q^0\Big)+\mu \frac{\partial^3}{\partial x^3}u^0\right]dt+\Lambda^{1/2}\frac{\partial^2}{\partial x^2}u^0\,dW_t.$$
\end{remark}

\begin{remark}[When $J_0=1$] \label{rem:role_u^1}
For $\al > 4$ or $\al < 4/3$, we have $\delta > \al/2$ and $J_0=1$. Thus we may remove $u^1$ in the quantity $q^\eps$: $\eps^{\delta-\al/2} u^1$ tends to zero for the strong topology and thus does not contribute directly to the limit $q^0$ of $q^\eps$. Nevertheless we emphasize that $u^1$ and $w^1$ are used to obtain the weak convergence of $q^\eps$.
\end{remark}

\begin{remark}[When $\al < 1$]\label{rem:al<1_strong_convergence}
In this range for $\al$, we have a stronger convergence and the result can be extended to any dimension $d$, that is $a$ is periodic in $z$ with the period $[0,1]^d$ and we identify periodic functions with functions defined on the torus $\mathbb T^d$.
\end{remark}

\begin{remark}[Regularity of $\imath$]\label{r_regu_ini}
The regularity assumption on $\imath$ given in condition {\rm \ref{a1}} can be weakened. Namely, the statement of Theorem \ref{thm:main_result} holds if $\imath$ is $\max(J_0+1,J_1)$ times continuously differentiable and the corresponding partial derivatives decay at infinity sufficiently fast.
\end{remark}

\begin{remark}[Dimension $d$]\label{r_dim_pb}
All results in Sections \ref{ssect:aux_pb} and \ref{sect:formal_exp} are valid if the dimension of the problem is any integer $d \geq 1$, that is if $z \in \mathbb{T}^d$. However the trick used in Section \ref{sect:conv_sing_mart_part} works only in dimension 1.
%n dimension 1, we emphasize that in this section the dimension can be  $d$. Nevertheless we present the result with $z \in \mathbb{T}$ (instead of $\mathbb T^d$), whereas the second variable $y$ belongs to $\R^n$.
\end{remark}

%we distinguish three cases:
%\begin{itemize}
%\item For $\alpha \leq 1$, $w^j=0$ for any $j$.
%\item For $\alpha \in (1,2)$, $w^1=0$ and we define $w^j$ recursively by:
%\end{itemize}
%\begin{equation}\label{eq:def_w_k_al<2}
%\forall k \geq 0, \quad w^{k+2}(x,t) = -  \sum_{m=0}^{k} \cC_{k,m} \partial^2_{xx}  w^m(x,t) - \sum_{m=2}^{k+1} w^{m}(x,t).
%\end{equation}
%\begin{itemize}
%\item For $\alpha>2$, we define $w^j$ recursively by:
%\end{itemize}
%\begin{equation}\label{eq:def_w_k_al>2}
%\forall k \geq 0, \quad w^{k+1}(x,t) = -  \sum_{m=0}^{k} \cC_{k,m} \partial^2_{xx} v^m(x,t) - \sum_{m=1}^{k} w^{m}(x,t).
%\end{equation}

\subsection{Auxiliary problems}\label{ssect:aux_pb}
%--------------------

In this section, we define several functions required in Theorem \ref{thm:main_result} and study their main properties.

If {\bf (A)} holds, since $\aeff$ is positive, the problem \eqref{eq:u_0} is well posed, uniquely defined, smooth and satisfies the estimates
\begin{equation}\label{eq:est_uzero}
\Big|(1+|x|)^ N\frac{\partial^{\bf k}u^0(x,t)}{\partial t^{k_0}\partial x^{k_1}}\Big|\leq C_{N,{\bf k}}
\end{equation}
for all $N>0$ and all multi index ${\bf k}=(k_0,k_1)$, $k_i\geq 0$.

Now we distinguish the two cases $\al < 2$ and $\al > 2$.

\subsubsection*{Correctors and constants for $\al < 2$}
%-----------------

We begin by considering Problem \eqref{eq:aux_0_dif_al<2}. This equation has a unique up to an additive constant vector periodic solution. By the classical elliptic estimates, under our standing assumptions  for any $N>0$ there exists $C_N$ such that
\begin{equation}\label{eq:est_chi_al<2}
\|\chi^0\|_{C^N(\mathbb T\times\mathbb R^n)}\leq C_N.
\end{equation}
Indeed, multiplying equation \eqref{eq:aux_0_dif_al<2} by  $\chi^0$, using the Schwarz and Poincar\'e inequalities and considering \eqref{eq:norm_cond_chi0}, the estimate follows from \cite{GT}.

Higher order correctors are defined as periodic solutions of the equations
\begin{equation}\label{eq:aux_j_dif_al<2}
\mathrm{div}_z\big(a(z,y)\nabla_z \chi^j(z,y)\big)=-\mathcal{L}_y\chi^{j-1}(z,y), \quad j=1,\,2,\ldots,J_0.
\end{equation}
Notice that $\int_{\mathbb T}\chi^{j-1}(z,y)\,dz=0$ for all  $j=1,\,2,\ldots,J^0$, thus the compatibility condition is satisfied and the equations are solvable. By the similar arguments, the solutions $\chi^j$ defined by \eqref{eq:aux_j_dif_al<2} satisfy the same estimate as $\chi^0$.

We introduce the real numbers for $k\geq 1$:
\begin{equation} \label{eq:def_aeffk_al<2}
\aeffk=\int_{\mathbb R^n}\int_{\mathbb T} \big[a(z,y)\nabla_z\chi^k(z,y)+\nabla_z\big(a(z,y)\chi^k(z,y)\big)\big]p(y)\,dzdy.
\end{equation}
Arguing as for $u^0$ we conclude that  solutions $u^j$ of problems \eqref{eq:def_u_k} and functions $w^j$ defined by  \eqref{eq:def_w_k_al<2} are smooth functions that satisfy also  estimate \eqref{eq:est_uzero}.

Now we define:
\begin{eqnarray*} \nonumber
\widehat a^0(z,y)& =& a(z,y)+a(z,y)\nabla_z\chi^0(z,y)+\nabla_z\big(a(z,y)\chi^0(z,y)\big), \\ \nonumber
\langle a\rangle^0(y)& =&\int_{\mathbb T}\big(\widehat a^0(z,y)-\aeff\big)dz,
\end{eqnarray*}
and we consider the equation
\begin{equation}\label{eqdefQ0}
\mathcal{L}Q^0(y)=\langle {a}\rangle^0(y).
\end{equation}
According to \cite[Theorems 1 and 2]{pard:vere:01}, this equation has a unique up to an additive constant solution of at most polynomial growth. The constant $\Lambda$ is defined by:
\begin{equation}\label{eq:def_Lambda_al<2}
\Lambda=\int_{\mathbb R^n}\Big[\frac{\partial}{\partial y_{r_1}}(Q^{0})(y)\Big]q^{r_1r_2}(y)
\Big[\frac{\partial}{\partial y_{r_2}}(Q^{0})(y)\Big] p(y)\,dy.
\end{equation}
Note that the matrix $\Lambda$ is non-negative. Consequently its square root $\Lambda^{1/2}$ is well defined.

Finally we define $w^j$ recursively by $w^1=0$ and
\begin{equation}\label{eq:def_w_k_al<2}
\forall k \geq 0, \quad w^{k+2}(x,t) = -  \sum_{m=0}^{k} \cC_{k,m} u^m_{xx}(x,t) - \sum_{m=1}^{k} w^{m+1}(x,t).
\end{equation}
The triangular array of constants $(\cC_{k,m})_{0\leq m \leq k}$ is defined by \eqref{eq:cC_k_m}.

\subsubsection*{Correctors and constants for $\al > 2$}
%-----------------

Let start here with the definition of the functions $w^j$. Here we have
\begin{equation}\label{eq:def_w_k_al>2}
\forall k \geq 0, \quad w^{k+1}(x,t) = -  \sum_{m=0}^{k} \cC_{k,m} u^m_{xx}(x,t) - \sum_{m=1}^{k} w^{m}(x,t),
\end{equation}
where the constants $(\cC_{k,m})_{0\leq m \leq k}$ are again defined by \eqref{eq:cC_k_m}.
%\item For $\alpha \leq 1$, $w^j=0$ for any $j$.
%\item For $\alpha \in (0,2)$, we define $w^j$ recursively by $w^1=0$ and
%\begin{equation}\label{eq:def_w_k_al<2}
%\forall k \geq 0, \quad w^{k+2}(x,t) = -  \sum_{m=0}^{k} \cC_{k,m} u^m_{xx}(x,t) - \sum_{m=1}^{k} w^{m+1}(x,t).
%\end{equation}
Note that these constants are not the same if $\al > 2$ or if $\al <2$ since the correctors used in \eqref{eq:cC_k_m} are different. Somehow the function $w^k$ for $\al>2$ is equal to the function $w^{k+1}$ for $\al <2$; there is a shift between them.

Here the first auxiliary problem \eqref{eq:def_chi_0_al>2} $\bar \cA \chi^0 = -\bar a_z$ reads
\begin{equation*}
\mathrm{div}\big(\bar a(z)\nabla \chi^0(z)\big)=-\mathrm{div}\, \bar a(z),\quad z\in\mathbb T;
\end{equation*}
where
$$\bar a(z) = \int_{\R^n} a(z,y) p(y) dy.$$
It has a unique up to an additive constant periodic solution. This constant is chosen in such a way that \eqref{eq:norm_cond_chi0} holds, namely $ \int_{\mathbb T}\chi^0(z)\,dz=0$.
By the classical elliptic estimates (see \cite{GT}), under {\bf (A)}, we have
\begin{equation}\label{eq:est_chi_al>2}
\|\chi^0\|_{L^\infty(\mathbb T)} + \|\chi^0\|_{C^k(\mathbb T)}\leq C.
\end{equation}

Then we can define recursively
\begin{eqnarray*} %\label{eq:def_f_1}
f^0(z,y) & = & a-\aeff + \left( a \chi^{0}_{z} + (a \chi^{0})_z \right), \\ % \label{eq:def_chi_2}
\bar \cA \chi^{1}(z) &= & \overline{f^0},\quad \\ %\label{eq:def_f_2}
f^1(z,y) &= &\chi^{0} (a-\aeff)  +  \left( a \chi^{1}_{z} + (a \chi^{1})_z \right).
\end{eqnarray*}
Note that here we use $\langle \overline{f^{0}} \rangle = 0$ to obtain the periodic solution $\chi^1$. Now we define the constant $\urma^{1,\rm eff}$ by:
\begin{equation*}% \label{eq:def_c_1}
\urma^{1,\rm eff}= \langle \overline{f^{1}} \rangle
\end{equation*}
and the corrector $\chi^2$ by:
$$\ocA \chi^{2}(z) = \overline{f^{1}} - \langle\overline{ f^{1}} \rangle.$$
Let us also define by induction the following quantities for $k\geq 2$
\begin{align} \nonumber%\label{eq:def_f_k}
f^k(z,y) &= &\chi^{k-1} (a-\aeff)  +  \left( a \chi^{k}_{z} + (a \chi^{k})_z \right), \\  \nonumber%\label{eq:def_u_a_k}
\urmak & =& \langle \overline{f^{k}} \rangle, \\  \label{eq:def_chi_k_al>2}
\ocA \chi^{k+1}(z) & = & \overline{f^{k}} - \langle\overline{ f^{k}} \rangle  + \left( \sum_{j=1}^{k-1} \urma^{k-j,{\rm eff}} \chi^{j-1}  \right).
\end{align}
All functions $\chi^k$, $k\geq 1$, are solution of an equation of the form $\ocA v = F$, where $F$ is a periodic  zero mean  bounded function. Hence all functions $\chi^k$ are well defined on $\mathbb{T}$ and satisfy \eqref{eq:norm_cond_chi0} and \eqref{eq:est_chi_al>2}.

\vspace{0.5cm}
Finally from $\chi^0$, we can define $\kappa^0$ as the solution of:
\begin{equation}\label{eq:def_kappa_0_al>2}
\cL \kappa^0(z,y)= (\mathrm{div}_z a(z,y) - \mathrm{div}_z \bar a(z)) +\mathrm{div}_z ((a(z,y)  - \bar a(z))\nabla_z \chi^0(z) ),
\end{equation}
%To be precised, $\kappa^0$ satisfies for any $y \in \R^n$
%\begin{equation*}%\label{eq:def_kappa_1}
%\cL \kappa^0(z,y)= (\mathrm{div}_z a(z,y) - \mathrm{div}_z \bar a(z)) +\mathrm{div}_z ((a(z,y)  - \bar a(z))\nabla_z \chi^0(z) ),
%\end{equation*}
$z \in \mathbb T$ being a parameter. Since the right-hand side has a zero mean value (w.r.t. $y$) and is bounded, according to \cite{pard:vere:01}, this equation has a unique up to an additive constant (w.r.t. $y$) solution of at most polynomial growth:
\begin{equation}\label{eq:est_kappa_1}
|\kappa^0(z,y)| \leq C(1+|y|^p), \quad \forall (z,y) \in\mathbb T \times \R^d.
\end{equation}
Moreover we can impose that
\begin{equation}\label{eq:norm_kappa_1}
 \int_{\mathbb T}\kappa^0(z,y)\,dz=0.
\end{equation}
Finally the right-hand side of \eqref{eq:def_kappa_0_al>2} being a smooth function w.r.t. $z$ with bounded derivatives, again according the representation of \cite{pard:vere:01}, $z \mapsto \kappa^0(z,y)$ is smooth. Indeed the operators $\cL$ and derivative w.r.t. $z$  commute. Then by induction, we introduce a sequence $\kappa^k$, $k \geq 0$ defined by:
\begin{equation*}%\label{eq:def_kappa_k}
%\cL \kappa^1(z,y) &= &(a_z-\bar a_z) + (\cA-\bar \cA)\chi^1,\\
 \cL \kappa^{k+1}(z,y) =  (\cA-\ocA) \chi^{k+1} + (f^{k} -  \overline{f^{k}}).
\end{equation*}
for $k=1,2,\ldots, J_1+2$. These higher order correctors $\kappa^k$ exist and satisfy Estimates \eqref{eq:est_kappa_1} and \eqref{eq:norm_kappa_1}.

\vspace{0.5cm}
We also use the following functions or constants:
\begin{align} \label{eq:def_tau_1}
\ocA \tau^0(z) & =- \overline{(\cA \kappa^0)},\\ \nonumber%\label{eq:def_g_1}
g^0(z,y) & =  a(\kappa^0+\tau^0)_z + (a(\kappa^0 +\tau^0))_z, \\ \label{eq:def_d_1}
\rma^{1,\rm eff} & = \langle \overline{g^0} \rangle = \overline{\langle a(\kappa^0+\tau^0)_z \rangle},
\end{align}
%and $v^1$ is defined by \eqref{eq:def_v_k} (recall that $v^0=u^0$):
%$$v_t^1 = \aeff v^1_{xx} + d_{1} u^0_{xx} + w^1, \quad v^1(x,0)=0.$$
%The function $w^1$ is defined by \eqref{eq:def_w_k} with $k=0$:
%$$w^1(x,t) = \cC_{0,0} u^0_{xx}(x,t).$$
%The constant $\cC_{0,0}$ is defined below by \eqref{eq:def_cC_ell_m}
Now for any $k\geq 1$
\begin{equation}\label{eq:def_tau_k}
\ocA \tau^k(z) = -\overline{\cA \gamma^{k-1}},
 \end{equation}
and
\begin{align} \label{eq:def_gamma_1}
\cL \gamma^0(z,y)& = (\cA - \ocA) \tau^0 + (\cA \kappa^0 - \overline{\cA \kappa^0}), \\ \label{eq:def_gamma_k}
\cL \gamma^k(z,y) & = (\cA - \ocA)\tau^k + (\cA \gamma^{k-1} - \overline{\cA \gamma^{k-1}}).
\end{align}
For $k\geq 2$ we put
\begin{equation} \label{eq:def_aeffk_al>2}
\aeffk = \overline{\langle  a (\tau^{k-1}(z) + \gamma^{k-2}(z,y) )_z\rangle}.
\end{equation}
By the same arguments, the set of correctors $\tau^k$ and $\gamma^k$ defined respectively by \eqref{eq:def_tau_1}, \eqref{eq:def_tau_k}, \eqref{eq:def_gamma_1} and \eqref{eq:def_gamma_k} verify again Estimates \eqref{eq:norm_cond_chi0}, \eqref{eq:est_chi_al>2}, \eqref{eq:est_kappa_1} and \eqref{eq:norm_kappa_1}. Roughly speaking, they are smooth in both variables, bounded in $z$ and show polynomial growth in $y$.

\vspace{0.5cm}
To establish Lemmata \ref{lmm:first_abs_cont_part_al_geq_4} and \ref{lem:second_abs_term}, we introduce two sequences of functions
\begin{equation*} %\label{eq:def_g_h_k}
g^k(z,y) = a (\tau^k + \gamma^{k-1})_z,\quad h^k(z,y) = (a (\tau^k + \gamma^{k-1}))_z
\end{equation*}
and $\eta^k(z)$ and $\zeta^k(z,y)$, $k\geq 1$, solutions of
\begin{align} \label{eq:def_eta_1}
\ocA \eta^1 & =   - (\overline{\cA \kappa^1})  - (\overline{g^0} - \langle \overline{g^0} \rangle), \\ \label{eq:def_zeta_1}
\cL \zeta^1 &=  (\cA - \ocA) \eta^1 + (\cA \kappa^1 - \overline{\cA \kappa^1})+ (g^0-\overline{g^0}),
\end{align}
and for $k\geq 2$
\begin{align} \label{eq:def_eta_k}
\ocA \eta^k & =   - (\overline{\cA \zeta^{k-1}})  - \overline{h^{k-1}} - (\overline{g^{k-1}} - \langle \overline{g^{k-1}} \rangle), \\  \label{eq:def_zeta_k}
\cL \zeta^k &=  (\cA - \ocA) \eta^k + (\cA \zeta^{k-1} - \overline{\cA \zeta^{k-1}}) \\ \nonumber
& +  (h^{k-1}-\overline{h^{k-1}}) + (g^{k-1}-\overline{g^{k-1}}).
\end{align}
Again the same arguments show that $\eta^k$ and $\zeta^k$ defined by \eqref{eq:def_eta_1}, \eqref{eq:def_zeta_1}, \eqref{eq:def_eta_k} and \eqref{eq:def_zeta_k}, exist and are smooth.

Finally the constant $\Lambda$ is defined by
\begin{equation} \label{eq:def_Lambda_al>2}
\Lambda = \overline{ \langle \chi^0 \Upsilon^0 \rangle^2} > 0,
\end{equation}
where $\Upsilon^0$ is given by:
\begin{equation*}% \label{eq:def_upsilon_1}
\Upsilon^0 (z,y) = -\kappa^0_y (z,y)q(y).
\end{equation*}
Here $\kappa^0_y$ stands for the gradient of $\kappa^0$ w.r.t. $y$ (this notation is used in the sequel of the paper).

\section{Formal expansion for the solution}\label{sect:formal_exp}
%-----------------

In both cases $\al >2$ and $\al<2$, we begin by constructiong a formal expansion of $u^\eps$, where the functions $w^k$ in the definition \eqref{eq:def_u_k} of $u^k$ will be left as \textit{free parameters} in this first part. This development leads to a remainder $R^\eps$
that satisfies a problem with  large parameters (as $\eps$ tends to zero) in the equation and in the initial condition. Moreover this development gives the main part of $q^\eps$.

We denote by $B$ the $n$-dimensional standard Brownian motion driving the process $\xi$.

\subsection{The case $\al < 2$}
%---------------------

We consider the following expression:
$$\cE^\eps(x,t) =  \sum\limits_{k=0}^{J_0}\eps^{k\delta}\Big(u^k(x,t)+\sum\limits_{j=0}^{J_0-k}\eps^{(j\delta+1)}\chi^j\Big(\frac x\eps,\xi_\frac t{\eps^\alpha}\Big)\nabla u^k(x,t)\Big)$$
where $u^k(x,t)$  and $\chi^j(z,y)$ are respectively defined in  \eqref{eq:def_u_k} and \eqref{eq:aux_j_dif_al<2}. We show that $u^\eps - \cE^\eps$ is equal to $\eps^{\al/2} R^\eps$ and we split the remainder $R^\eps$ as follows: $R^\eps=r^\eps+\check r^\eps+\widetilde r^\eps + \rho^\eps$, such that $\widetilde r^\eps $ and $\rho^\eps$ strongly converge to zero and $\check r^\eps$ converges to $q^0$ (Proposition \ref{prop:weak_conv_al<2}). The first term $r^\eps$ contains large martingale terms and its weak convergence to zero is proved in Section \ref{sect:conv_sing_mart_part}. Note that for $\al<2$, there is no term $\widehat r^\eps$.

For $k\geq 1$ we define
\begin{align} \label{eq:def_widehat_a_k_al<2}
\widehat{a}^k(z,y) & = {a}(z,y)\nabla_z\chi^k(z,y)+\nabla_z\big({a}(z,y)\chi^k(z,y)\big), \\ \label{eq:def_mean_widehat_a_k_al<2}
\langle {a}\rangle^k(y)& = \int_{\mathbb T}\big(\widehat { a}^k(z,y)-a^{k,{\rm eff}}\big)dz.
\end{align}
\begin{proposition} \label{prop:formal_exp_al<2}
If we define:
$$R^\eps(x,t) =\eps^{-\alpha/2}\left[ u^\eps(x,t)-\cE^\eps (x,t) \right], $$
then $R^\eps$ is the sum $R^\eps=r^\eps+\check r^\eps+\widetilde r^\eps + \rho^\eps$ where:
\begin{itemize}
\item The dynamics of $r^\eps$ contains the terms with large parameters:
\begin{eqnarray} \label{eq:dyn_r_eps_al<2}
&&dr^\eps = (\cA^\eps r^\eps) dt-\eps^{-\al/2} \sum\limits_{k=1}^{J_0} \eps^{k\delta} w^k(x,t) \ dt  \\ \nonumber
&&\quad -\eps^{1-\alpha} \sum\limits_{k=0}^{J_0}\sum\limits_{j=0}^{J_0-k}\eps^{(k+j)\delta}\sigma(\xi_\frac t{\eps^\alpha})\nabla_y\chi^j\Big(\frac x\eps,\xi_\frac t{\eps^\alpha}\Big)\nabla u^k(x,t)\,dB_t
\end{eqnarray}
with $r^\eps(x,0)=0$.
\item The dynamics of $\check r^\eps$ is given by:
\begin{equation} \label{eq:dyn_check_r_eps_al<2}
\partial_t \check r^\eps- \cA^\eps \check r^\eps = \eps^{-\alpha/2}\sum\limits_{j=0}^{J_0}\sum\limits_{k=0}^{J_0-j}\eps^{(k+j)\delta}\Big[\widehat {\rm a}^k\Big(\frac x\eps, \xi_\frac t{\eps^\alpha}\Big)-a^{k,{\rm eff}}\Big] \frac{\partial^2 u^j}{\partial x^2}
\end{equation}
with $\check r^\eps (x,0)=0$.
\item The last terms $\widetilde r^\eps$ and $\rho^\eps$ satisfy:
\begin{equation} \label{eq:estim_L_2_al<2}
\mathbf{E}\| \widetilde r^\eps\|^2_{L^2(\mathbb R\times(0,T))} + \mathbf{E}\| \rho^\eps\|^2_{L^2(\mathbb R\times(0,T))}\leq C\eps^\delta.
\end{equation}
\end{itemize}
\end{proposition}
\begin{proof}
We substitute $R^\eps$ for $u^\eps$ in \eqref{ori_cauch} using It\^o's formula:
\begin{align*}
& d R^\eps-\mathrm{div} \big[a\big(\frac x\eps,\xi_{\frac t{\eps^\alpha}}\big)\nabla R^\eps\big]dt \\ \nonumber
& =-\eps^{-\frac\alpha2}\sum\limits_{k=0}^{J_0}\eps^{k\delta}\Big[\partial_tu^k
+\sum\limits_{j=0}^{J_0-k}\eps^{(j\delta+1-\alpha)}\big(\cL_y \chi^j\big)\big(\frac x\eps,\xi_{\frac t{\eps^\alpha}}\big)\nabla u^k \\ \nonumber
&\qquad  \qquad \qquad \qquad +\sum\limits_{j=0}^{J_0-k}\eps^{(j\delta+1)}\chi^j\big(\frac x\eps,\xi_{\frac t{\eps^\alpha}}\big)\partial_t\nabla u^k
\Big] dt \\ \nonumber
&-\sum\limits_{k=0}^{J_0}\sum\limits_{j=0}^{J_0-k}\eps^{(1-\alpha+(k+j)\delta)}\sigma(\xi_\frac t{\eps^\alpha})\nabla_y\chi^j\Big(\frac x\eps,\xi_\frac t{\eps^\alpha}\Big)\nabla u^k(x,t)\,dB_t\\ \nonumber
&+\eps^{-\frac\alpha2}\sum\limits_{k=0}^{J_0}\eps^{k\delta-1}\Big[(\mathrm{div} a)\big(\frac x\eps,\xi_{\frac t{\eps^\alpha}}\big)
+\sum\limits_{j=0}^{J_0-k}\eps^{j\delta}\big(\mathrm{div}(a\nabla\chi^j)\big)\big(\frac x\eps,\xi_{\frac t{\eps^\alpha}}\big)\Big]
\nabla u^k\\ \nonumber
& +\eps^{-\frac\alpha2}\sum\limits_{k=0}^{J_0}\sum\limits_{j=0}^{J_0-k}\eps^{(k+j)\delta}\ \widehat a^{j,im}\big(\frac x\eps,\xi_{\frac t{\eps^\alpha}}\big)
\frac{\partial^2}{\partial x_i\partial x_m}u^k dt\\ \nonumber
&+\eps^{-\frac\alpha2}\sum\limits_{k=0}^{J_0}\sum\limits_{j=0}^{J_0-k}\eps^{(k+j)\delta+1}\ (a^{im}\chi^{j,l})\big(\frac x\eps,\xi_{\frac t{\eps^\alpha}}\big)
\frac{\partial^3}{\partial x_i\partial x_m\partial x_l}u^k dt.
\end{align*}
Due to \eqref{eq:aux_0_dif_al<2} and \eqref{eq:aux_j_dif_al<2}
\begin{align*}
& -\sum\limits_{k=0}^{J_0}\eps^{k\delta} \sum\limits_{j=0}^{J_0-k}\eps^{(j\delta+1-\alpha)}\big(\cL_y\chi^j\big)\big(\frac x\eps,\xi_{\frac t{\eps^\alpha}}\big)\nabla u^k\\
&\quad +\sum\limits_{k=0}^{J_0}\eps^{k\delta-1}\Big[(\mathrm{div} a)\big(\frac x\eps,\xi_{\frac t{\eps^\alpha}}\big)
+\sum\limits_{j=0}^{J_0-k}\eps^{j\delta}\big(\mathrm{div}(a\nabla\chi^j)\big)\big(\frac x\eps,\xi_{\frac t{\eps^\alpha}}\big)\Big]
\nabla u^k \\
&=-\eps^{(J_0+1)\delta-1}\sum\limits_{k=0}^{J_0}\big(\cL_y\chi^{J_0-k}\big)\big(\frac x\eps,\xi_{\frac t{\eps^\alpha}}\big)\nabla u^k.
\end{align*}
Considering equations \eqref{eq:def_u_k} and the definitions of $a^{k,{\rm eff}}$ and $\widehat{\rm a}^k(z,y)$,  we obtain
\begin{eqnarray}\label{eq:prob_V_dif_2_al<2}
&& d R^\eps\displaystyle(x,t)-\mathrm{div} \Big[{\rm a}\Big(\frac x\eps,\xi_\frac t{\eps^\alpha}\Big)\nabla R^\eps\Big]\,dt\\ \nonumber
&&\quad = \Big(\eps^{-\alpha/2}\sum\limits_{j=0}^{J_0}\sum\limits_{k=0}^{J_0-j}\eps^{(k+j)\delta}\Big[\widehat {\rm a}^k\Big(\frac x\eps, \xi_\frac t{\eps^\alpha}\Big)-a^{k,{\rm eff}}\Big]^{im} \frac{\partial^2 u^j}{\partial x_i\partial x_m}\Big)\,dt\\ \nonumber
&&\quad - \sum\limits_{k=0}^{J_0}\sum\limits_{j=0}^{J_0-k}\eps^{(1-\alpha+(k+j)\delta)}\sigma(\xi_\frac t{\eps^\alpha}) \nabla_y\chi^j\Big(\frac x\eps,\xi_\frac t{\eps^\alpha}\Big)\nabla u^k(x,t)\,dB_t\\ \nonumber
&&\quad -  \sum\limits_{k=1}^{J_0} \eps^{k\delta-\alpha/2} w^k(x,t) \ dt \\ \nonumber
&&\quad +\eps^{1-\alpha/2} \sum\limits_{j=0}^{J_0}\eps^{j\delta}\mathfrak b^j\Big(\frac x\eps,\xi_\frac t{\eps^\alpha}\Big)\mathfrak F^j(x,t) \,dt,
\end{eqnarray}
with $\rma^{0,{\rm eff}} = \aeff$ and with periodic in $z$ smooth functions $\mathfrak b^j=\mathfrak b^j(z,y)$ of at most polynomial growth in $y$, and $\mathfrak F^j$ satisfying \eqref{eq:est_uzero}, that is
$$\left| (1+|x|)^{N} D^{\bf k}\mathfrak F^j \right| \leq C_{{\bf k},N}.$$
The initial condition for $R^\eps$ is given by:
$$R^\eps(x,0)=\eps^{1-\al/2} \sum\limits_{k=0}^{J_0}\sum\limits_{j=0}^{J_0-k}\eps^{j\delta}\chi^j\Big(\frac x\eps,\xi_0\Big)\nabla u^k(x,0).$$

By the linearity of \eqref{eq:prob_V_dif_2_al<2}, we represent $R^\eps$ as the sum $R^\eps=r^\eps+\check r^\eps+\widetilde r^\eps + \rho^\eps$ where $r^\eps$ and $\check r^\eps$ are given by \eqref{eq:dyn_r_eps_al<2} and \eqref{eq:dyn_check_r_eps_al<2}, and $\widetilde r^\eps$ contains all negligible terms:
\begin{equation*}
\partial_t \widetilde r^\eps- \cA^\eps \widetilde r =\eps^{1-\alpha/2} \sum\limits_{j=0}^{J_0}\eps^{j\delta}\mathfrak b^j\Big(\frac x\eps,\xi_\frac t{\eps^\alpha}\Big)\mathfrak F^j(x,t) = \cB^\eps(x,t)
\end{equation*}
together with $\widetilde r^\eps(x,0)=0$. For the right-hand side here we have the following estimate:
\begin{eqnarray*}
{\bf E}\|\mathcal{B}^\eps\|^2_{L^2(\mathbb R \times(0,T))} & \leq & C\eps^{1-\alpha/2}\int_0^T\int_{\mathbb R}\int_{\mathbb R^n}
(1+|y|)^{N_1}(1+|x|)^{-2n}p(y)\,dydxdt \\
& \leq & C \eps^{1-\alpha/2}.
\end{eqnarray*}
Similarly, ${\bf E}\|\widetilde r^\eps(\cdot,0)\|^2_{L^2(\mathbb R)}\leq C\eps^{1-\alpha/2}$. Therefore, $ \widetilde r^\eps$ satisfies \eqref{eq:estim_L_2_al<2} and thus does not contribute in the limit. Finally $\rho^\eps$ satisfies the evolution equation
$$\partial_t \rho^\eps =\cA^\eps \rho^\eps$$
with the initial condition $\rho^\eps(x,0)=R^\eps(x,0)$. Since $\al < 2$, we deduce that ${\bf E}\|\rho^\eps(\cdot,0)\|^2_{L^2(\mathbb R)}\leq C\eps^{1-\alpha/2}$. Thereby this term does not contribute in the limit as well.
\end{proof}

The second term $\check r^\eps$ gives the limit in Theorem \ref{thm:main_result}.
\begin{proposition} \label{prop:weak_conv_al<2}
The solution $\check r^\eps$ of Problem \eqref{eq:dyn_check_r_eps_al<2} converges in law, as $\eps$ goes to $0$, in $L^2(\mathbb R\times(0,T))$
equipped with strong topology, to the solution of \eqref{eq:eff_spde}.
\end{proposition}
\begin{proof}
Recall that from the very definition of $\widehat a_k$ in \eqref{eq:def_widehat_a_k_al<2} and $\langle{\rm a}\rangle^k$  in \eqref{eq:def_mean_widehat_a_k_al<2} the problem
$$
\Delta_z\zeta^{k}(z,y)=   (\widehat a^k(z,y)-\langle a\rangle^k(y))
$$
has a periodic solution which is unique up to an additive constant. Letting $\Theta^{k}(z,y)=\nabla\zeta^{k}(z,y)$, we obtain a vector functions $\Theta^{k}$ such that
$$
\mathrm{div}\,\Theta^{k}(z,y)=(\widehat a^k(z,y)-\langle a\rangle^k(y)), \qquad \|\Theta^{k}\|^N_{C^\ell(\mathbb T\times\R^n)}
\leq C_{N,\ell}.
$$
It is then  straightforward to check that for the functions
\begin{eqnarray*}
H^\eps(x,t) & = & \eps^{-\alpha/2}\sum\limits_{j=0}^{J_0}\sum\limits_{k=0}^{J_0-j}\eps^{(k+j)\delta}\Big[\widehat {\rm a}^k\Big(\frac x\eps, \xi_\frac t{\eps^\alpha}\Big)-\langle{\rm a}\rangle^k\big(\xi_\frac t{\eps^\alpha}\big)\Big]
\frac{\partial^2 u^j}{\partial x^2} \\
& = & \eps^{1-\frac\alpha2}\sum\limits_{j=0}^{J_0}\sum\limits_{k=0}^{J_0-j}\eps^{(k+j)\delta}\Big\{
\mathrm{div}\Big[\Theta^{k}\Big(\frac x\eps,\frac t{\eps^\alpha}\Big)\frac{\partial^2}{\partial x^2}u^j(x,t)\Big] \\
&& \qquad-\Theta^{k}\Big(\frac x\eps,\frac t{\eps^\alpha}\Big)\nabla\Big(\frac{\partial^2}{\partial x^2}u^j(x,t)\Big)\Big\}
\end{eqnarray*}
the following estimate is fulfilled:
\begin{equation*} %\label{eq:estimH_eps_al<2}
{\bf E}\|H^{\eps}\|^2_{L^2(0,T; H^{-1}(\mathbb R))}\leq C\eps^{2-\alpha}.
\end{equation*}
We split $\check r^\eps =\check r^{\eps,1}+\check r^{\eps,2}$, where
\begin{itemize}
\item $\check r^{\eps,1}$ solves \eqref{eq:dyn_check_r_eps_al<2} with $H^\eps$ on the right-hand side; it admits the estimate:
$$\mathbf{E}\|\check r^{\eps,1}\|^2_{L^2(0,T;H^1(\mathbb R))}\leq C\eps^\delta.$$
\item $\check r^{\eps,2}$ solves also \eqref{eq:dyn_check_r_eps_al<2}, but with
\begin{equation}\label{def_check r2_ap}
\eps^{-\alpha/2}\sum\limits_{j=0}^{J_0}\sum\limits_{k=0}^{J_0-j}\eps^{(k+j)\delta}\Big[\langle{\rm a}\rangle^k\big(\xi_\frac t{\eps^\alpha}\big)-a^{k,{\rm eff}}\Big] \frac{\partial^2 u^j}{\partial x^2}
\end{equation}
on the right-hand side.
\end{itemize}
According to Assumption \ref{a6} and to \cite[Theorem 3]{pard:vere:01} (see also \cite[Lemma VIII.3.102 and Theorem  VIII.3.97]{JaShi}), the processes
$$
A^k(t)=\int_0^t (\langle a\rangle^k(\xi_s)- a^{k,{\rm eff}} )ds
$$
satisfy the functional central limit theorem (invariance principle), that is the process
$$
A^{\eps,k}(t)=\eps^{\frac\alpha2}\int_0^{\eps^{-\alpha}t} (\langle a\rangle^k(\xi_s)- a^{k,{\rm eff}} )ds
$$
converges in law in $C([0,T]; \R)$ to a one-dimensional Brownian motion with variance coefficient
$$
(\Lambda_k)=\int_{\mathbb R^n}\Big[\frac{\partial}{\partial y_{r_1}}(Q^{k})(y)\Big]q^{r_1r_2}(y)
\Big[\frac{\partial}{\partial y_{r_2}}(Q^{k})(y)\Big] \rho(y)\,dy.
$$
with $Q^0$ defined in \eqref{eqdefQ0} and $Q^k$ given by
\begin{equation*}%\label{eqdefQj}
\mathcal{L}Q^k(y)=\langle {\rm a}\rangle^k(y), \qquad k=1,\ldots.
\end{equation*}
In the expression of $\Lambda_k$ the summation over the indices $r_1$ and $r_2$ is assumed.

Denote by $\check r^{\eps,0}$ the solution of the following problem
\begin{equation*}%\label{eq:prob_r0}
\partial_t \check r^{\eps,0} - \cA^\eps \check r^{\eps,0} = \eps^{-\alpha/2}\Big[\langle{\rm a}\rangle^0\big(\xi_\frac t{\eps^\alpha}\big)-\aeff \Big] \frac{\partial^2 u^0}{\partial x^2}.
\end{equation*}
Obviously if $\check r^{\eps,0} $ converges, then $\check r^{\eps,2}$ solving the equation with the right-hand side defined in \eqref{def_check r2_ap} also converges to the same limit. We consider an auxiliary problem
\begin{equation*}%\label{prob_V_auxi}
\left\{\begin{array}{l}
\displaystyle
\partial_t r_{\rm aux}^{\eps}-\mathrm{div}\displaystyle \Big[\aeff \nabla r_{\rm aux}^{\eps}\Big]
=\eps^{-\alpha/2}\Big[\langle{\rm a}\rangle^0\big(\xi_\frac t{\eps^\alpha}\big)-\aeff \Big]
\frac{\partial^2u^0}{\partial x^2}\\[4mm]
r_{\rm aux}^{\eps}(x,0)=0,
\end{array}\right.
\end{equation*}
and notice that this problem admits an explicit solution
$$
r_{\rm aux}^{\eps}=\eps^{\alpha/2}A^0 \Big(\frac{t}{\eps^\alpha}\Big)\frac{\partial^2u^0}{\partial x^2}
= A^{\eps,0}(t) \frac{\partial^2u^0}{\partial x^2}.
$$
Since $u^0$ satisfies estimates \eqref{eq:est_uzero}, the convergence of $A^{\eps,0}$ implies that $r_{\rm aux}^{\eps}$   converges in law in $C((0,T);L^2(\mathbb R))$ to the solution of problem \eqref{eq:eff_spde}.

Next, we represent $ \check r^{\eps,0}$ as $ \check r^{\eps,0}(x,t)=\mathcal{Z}^{\eps}(x,t)+r_{\rm aux}^{\eps}(x,t)$. Then $\mathcal{Z}^{\eps}$
solves the problem
\begin{equation*}%\label{prob_Zeps}
\left\{\begin{array}{l}
\displaystyle
\partial_t \mathcal{Z}^{\eps}- \cA^\eps \mathcal{Z}^{\eps}
= \mathrm{div}\Big(\Big[a\Big(\frac x\eps,\frac t{\eps^\alpha}\Big)-\aeff \Big]\nabla r_{\rm aux}^{\eps}(x,t)\Big)\\[4mm] \mathcal{Z}^{\eps}(x,0)=0.
\end{array}\right.
\end{equation*}
The conclusion of the proposition can be deduced from the next result.
\end{proof}
\begin{lemma} \label{lmm:aux_conv_al<2}
The quantity $\mathcal{Z}^{\eps}$ goes to zero in probability in $L^2((0,T)\times\mathbb R)$,
as $\eps$ tends to $0$.
\end{lemma}
\begin{proof}
The arguments are the same as in the proof of \cite[Lemma 5.1]{KPP_2015}. %\red{For completeness, they are reproduced in the appendix. }
\end{proof}

\vspace{0.5cm}

To finish the proof of Theorem \ref{thm:main_result}, we need to control $r^\eps$, solution of problem \eqref{eq:dyn_r_eps_al<2}.

\subsection{The case $\al > 2$}
%------------------

The main idea is basically the same as that in the previous case. We construct an expansion $ \cE^\eps$ such that $u^\eps -  \cE^\eps = \eps^{\al/2} R^\eps$ and study the remainder to obtain Theorem \ref{thm:main_result}. However,  the asymptotic of $u^\eps$ is more complicated in this case,  and $\cE^\eps$ consists of two parts $\cE^\eps_1$ and $\cE^\eps_2$ (see Lemmata \ref{lmm:first_abs_cont_part_al_geq_4} and \ref{lem:second_abs_term} below). The second difference is the breakdown of $R^\eps$ in four parts: $R^\eps = r^\eps + \hat r^\eps  +  \widetilde r^\eps+ \rho^\eps$, such that
\begin{itemize}
\item $r^\eps$ again contains large parameters and now converges to $q^0$ (See Section \ref{sect:conv_sing_mart_part});
\item $\hat r^\eps$ weakly converges to zero (a direct consequence of the convergence of $r^\eps$);
\item $\widetilde r^\eps$ is negligible (as for $\al < 2$);
\item $\rho^\eps$ deals with the initial value of the discrepancy $R^\eps$ and is not a priori negligible for $\al > 2$. Its convergence to zero is studied in Section  \ref{sect:rest_init_cond}.
\end{itemize}
For $\al > 2$, the term $\check r^\eps$ does not exist.

%
%Nevertheless \red{the computations are quite similar and thus are postponed to the appendix}.
In this section the following notations are used for $k\geq 1$:
\begin{align} \label{eq:def_phi_k}
\phi^k\left( \frac{x}{\eps},x,t\right) & =  \sum_{n=1}^{k} \chi^{n-1}\left( \frac{x}{\eps}\right) \partial^{n}_x v^{k-n} \left(x,t\right), \\  \label{eq:def_Phi_k}
\Phi^k\left( \frac{x}{\eps},\xi_{t/\eps^\al},x,t\right) & = \sum_{n=1}^{k} \kappa^{n-1}\left( \frac{x}{\eps},\xi_{t/\eps^\al}\right) \partial^{n}_x v^{k-n} \left(x,t\right) ,
\end{align}
and
\begin{align} \label{eq:def_theta_k}
& \theta^k\left( \frac{x}{\eps},x,t\right) =\widehat \theta^k\left( \frac{x}{\eps},x,t\right) + \chi^0\left( \frac{x}{\eps}\right) u^k_x(x,t) \\ \nonumber
& \quad =  \left[ \sum_{n=0}^{k-1} \tau^{k-1-n} \left( \frac{x}{\eps}\right) u^n_x(x,t) \right] + \chi^0\left( \frac{x}{\eps}\right) u^k_x(x,t),\\ \label{eq:def_Theta_k}
& \Theta^k\left( \frac{x}{\eps},\xi_{t/\eps^\al},x,t\right) = \widehat \Theta^k\left( \frac{x}{\eps},\xi_{t/\eps^\al},x,t\right) + \kappa^0 \left( \frac{x}{\eps},\xi_{t/\eps^\al}\right) u^k_x (x,t) \\ \nonumber
& \quad =  \sum_{n=0}^{k-1} \gamma^{k-n-1} \left( \frac{x}{\eps},\xi_{t/\eps^\al}\right) u^n_x (x,t)+ \kappa^0\left( \frac{x}{\eps},\xi_{t/\eps^\al}\right) u^k_x (x,t).
\end{align}
Remember that $v^0 = u^0$. Finally
\begin{align} \label{eq:def_psi_k}
&\quad\psi^k\left( \frac{x}{\eps},x,t\right)= \widehat \psi^k \left( \frac{x}{\eps},x,t\right) + \chi^1 \left( \frac{x}{\eps}\right) u^k_{xx}(x,t) \\  \nonumber
&\quad =\sum_{n=0}^{k-1} \eta^{k-n} \left( \frac{x}{\eps}\right)u^n_{xx}(x,t)  + \tau^{k-1} \left( \frac{x}{\eps}\right) v^1_{x}(x,t) + \chi^1 \left( \frac{x}{\eps}\right) u^k_{xx}(x,t),\\  \label{eq:def_Psi_k}
&\quad\Psi^k\left( \frac{x}{\eps},\xi_{t/\eps^\al},x,t\right) = \widehat \Psi^k\left( \frac{x}{\eps},\xi_{t/\eps^\al},x,t\right) + \kappa^1\left( \frac{x}{\eps},\xi_{t/\eps^\al}\right) u^k_{xx}(x,t) \\  \nonumber
&\quad = \sum_{n=0}^{k-1} \zeta^{k-n} \left( \frac{x}{\eps},\xi_{t/\eps^\al}\right) u^n_{xx} + \gamma^k \left( \frac{x}{\eps},\xi_{t/\eps^\al}\right)v^1_{x}(x,t)\\ \nonumber
& \qquad \qquad + \kappa^1 \left( \frac{x}{\eps},\xi_{t/\eps^\al}\right) u^k_{xx}(x,t).
\end{align}
The functions $\phi^k$, $\theta^k$, $\psi^k$ are bounded and smooth functions, whereas the random functions $\Phi^k$, $\Theta^k$ and $\Psi^k$ are bounded and smooth w.r.t. $x$ and of at most polynomial growth w.r.t. $\xi_{t/\eps^\al}$.

As in \cite{KPP_2015}, Eq. (21), we consider a first principal part of the asymptotic of $u^\eps$ of the form:
%\begin{eqnarray}\label{eq:operator_E_1}
%\cE_1^\eps(x,t) & = &  u^0(x,t) +  \sum_{k=1}^{J_1+1} \eps^{k} v^k(x,t) +\sum_{k=1}^{J_1+2} \eps^k \phi^k\left( \frac{x}{\eps},x,t\right) \\ \nonumber
%&  + & \sum_{k=1}^{J_1    +2} \eps^{k+\delta}\Phi^k\left( \frac{x}{\eps},\xi_{t/\eps^\al},x,t\right) \\ \nonumber
%& + & \eps^{\delta} u^1(x,t) +\eps^{\delta +1} \theta^1\left( \frac{x}{\eps},x,t\right) + \eps^{2\delta+1} \Theta^1 \left( \frac{x}{\eps},\xi_{t/\eps^\al},x,t\right) \\ \nonumber
%&+& \eps^{\delta+2} \psi^1\left( \frac{x}{\eps},x,t\right) + \eps^{2\delta+2} \Psi^1\left( \frac{x}{\eps},\xi_{t/\eps^\al},x,t\right).
%\end{eqnarray}

\begin{multline}\label{eq:operator_E_1}
\cE_1^\eps(x,t)  =   u^0(x,t) +  \sum_{k=1}^{J_1+1} \eps^{k} v^k(x,t)\\
+\sum_{k=1}^{J_1+2} \eps^k \phi^k\left( \frac{x}{\eps},x,t\right)
  +  \sum_{k=1}^{J_1    +2} \eps^{k+\delta}\Phi^k\left( \frac{x}{\eps},\xi_{t/\eps^\al},x,t\right) \\
 +  \eps^{\delta} u^1(x,t) +\eps^{\delta +1} \theta^1\left( \frac{x}{\eps},x,t\right)\\
 + \eps^{2\delta+1} \Theta^1 \left( \frac{x}{\eps},\xi_{t/\eps^\al},x,t\right)
+ \eps^{\delta+2} \psi^1\left( \frac{x}{\eps},x,t\right) + \eps^{2\delta+2} \Psi^1\left( \frac{x}{\eps},\xi_{t/\eps^\al},x,t\right).
\end{multline}
with $\phi^k$ and $\Phi^k$ defined by \eqref{eq:def_phi_k} and \eqref{eq:def_Phi_k}. The functions $v^k$ are given as the solution of Eq. \eqref{eq:def_v_k_al>2}.
\begin{lemma} \label{lmm:first_abs_cont_part_al_geq_4}
The decomposition of the quantity $\cS^\eps_1 = (\partial_t - \cA^\eps) (\cE^\eps_1)$ is given by:
\begin{multline}  \label{eq:Ito_S_1}
\cS^\eps_1  =   \left[ \sum_{k=1}^{J_1+2} \eps^{k+\delta-\al/2} \Phi^k_y + \eps^{2\delta+1-\al/2} \Theta^1_y +  \eps^{2\delta+2-\al/2} \Psi^1_y \right] \sigma(\xi_{t/\eps^\al}) dB_t \\
 +  \left[ \eps^{J_1+1} r^{a,1,\eps} + \eps^{\delta+1} r^{a,2,\eps} \right] dt +\eps^\delta w^1 dt \\
 -   \left[ \eps^{2\delta -1} \cA^\eps \Theta^1 + \eps^{2\delta} ( a^\eps \Theta^1_{xz} + (a^\eps \Theta^1_x)_z) + \eps^{2\delta} \cA^\eps \Psi^1 \right] dt
\end{multline}
where the two remainders $r^{a,1,\eps}=r^{a,1,\eps}(x,t)$ and $r^{a,2,\eps}=r^{a,2,\eps}(x,t)$ only contain non-negative powers of $\eps$, and are bounded and smooth functions.
%\begin{eqnarray*}
%r^{a,1,\eps}(x,t)&=& v^{J_1+1}_t +   \phi^{J_1+1}_t +\eps \phi^{J_1+2}_t \\
%& - & a^\eps v^{J_1+1}_{xx}- \left( a^\eps \phi^{J_1+2}_{xz} + (a^\eps \phi^{J_1+2}_x)_z \right) -a^\eps \phi^{J_1+1}_{xx} - a^\eps \phi^{J_1+2}_{xx}, \\
%r^{a,2,\eps}(x,t) &= &- \eps^{J_1-1}\cA^\eps \Phi^{J_1+2} - \sum_{k=2}^{J_1+2} \eps^{k-2} \left( a^\eps \Phi^k_{xz} + (a^\eps \Phi^k_x)_z \right) \\
%& + & \sum_{k=1}^{J_1+2} \eps^{k-1} (\Phi_t^k - a^\eps \Phi^k_{xx} ) \\
%& + &   \theta^1_t +  \eps^{\delta } \Theta^1_t + \eps^{\delta + 1} \Psi^1_t +\eps \psi^1_t  -  a^\eps \theta^1_{xx}- \eps^{\delta } a^\eps \Theta^1_{xx} \\ \nonumber
%&-&  \left[  \eps^{\delta} ( a^\eps \Psi^1_{xz} + (a^\eps \Psi^1_x)_z) +  \eps^{\delta +1} a^\eps \Psi^1_{xx}\right] \\ \nonumber
%&  -&   \left[  ( a^\eps \psi^1_{xz} + (a^\eps \psi^1_x)_z) +  \eps^{\delta +1} a^\eps \psi^1_{xx}\right].
%\end{eqnarray*}
\end{lemma}
\begin{proof}
To prove this claim, we simply apply the It\^o formula. Using the definitions of the functions introduced in Section \ref{ssect:aux_pb}, and after some straightforward but cumbersome computations, we obtain the desired equality (also see the beginning of the proof of Proposition \ref{prop:formal_exp_al<2}). Boundedness and smoothness are consequences of the properties justified in Section \ref{ssect:aux_pb}.
\end{proof}

\vspace{0.3cm}
Since we need to control the last term in \eqref{eq:Ito_S_1}, we consider a second expansion:
\begin{multline}\label{eq:operator_E_2}
\cE^\eps_2(x,t)  =  \sum_{k=2}^{N_0} \eps^{k\delta} u^k(x,t) +\sum_{k=2}^{N_0} \eps^{k\delta +1} \theta^k\left( \frac{x}{\eps},x,t\right) \\
+  \sum_{k=2}^{N_0} \eps^{(k+1)\delta+1} \Theta^k \left( \frac{x}{\eps},\xi_{t/\eps^\al},x,t\right) \\
+ \sum_{k=2}^{N_0} \eps^{k\delta+2} \psi^k\left( \frac{x}{\eps},x,t\right) +  \sum_{k=2}^{N_0} \eps^{(k+1)\delta+2} \Psi^k\left( \frac{x}{\eps},\xi_{t/\eps^\al},x,t\right)
\end{multline}
where $N_0= 2J_0+2$. Adding the last term in \eqref{eq:Ito_S_1} we define
$$\cS^\eps_2 = (\partial_t - \cA^\eps) \cE^\eps_2 -  \left[ \eps^{2\delta -1} \cA^\eps \Theta^1 + \eps^{2\delta} ( a^\eps \Theta^1_{xz} + (a^\eps \Theta^1_x)_z) + \eps^{2\delta} \cA^\eps \Psi^1 \right] dt$$
In $\cS^\eps_2$ there is a martingale term of the from
\begin{eqnarray*}
&& \left[ \  \sum_{k=2}^{N_0} \eps^{(k+1)\delta+1-\al/2} \Theta^k_y +  \sum_{k=2}^{N_0} \eps^{(k+1)\delta+2-\al/2} \Psi^k_y\right]  \sigma(\xi_{t/\eps^\al}) dB_t \\
&&\quad = \frac{\eps^{\al/2}}{\eps} \sum_{k=2}^{N_0} \eps^{k\delta} \Theta^k_y \sigma(\xi_{t/\eps^\al})  dB_t+ \eps^{\al/2+\delta}   \sum_{k=2}^{N_0} \eps^{(k-1)\delta} \Psi^k_y \sigma(\xi_{t/\eps^\al})   dB_t.
\end{eqnarray*}

\begin{lemma} \label{lem:second_abs_term}
For a given sequence of smooth functions $w^k$, there exist a unique sequence of correctors $u^k$, $\theta^k$, $\Theta^k$, $\psi^k$ and $\Psi^k$ given respectively by \eqref{eq:def_u_k}, \eqref{eq:def_theta_k}, \eqref{eq:def_Theta_k}, \eqref{eq:def_psi_k}, \eqref{eq:def_Psi_k} and a unique sequence of constants $\aeffk$ such that
\begin{itemize}
\item The sequence $\aeffk$, given by \eqref{eq:def_d_1} and \eqref{eq:def_aeffk_al>2}, does not depend on the choice of $w^k$.
\item The absolutely continuous part of $\cS^\eps_2$ is given by
$$ \sum_{k=2}^{N_0} \eps^{k\delta} w^k(x,t) + \eps^{\delta+1} r^{a,3,\eps} + \eps^{(2J_0+3)\delta -1} r^{a,4,\eps}$$
%with $r^{a,3,\eps}$ given by \eqref{eq:def_r_3} and $r^{a,4,\eps}$ given by \eqref{eq:def_r_4}.
where the two remainders $r^{a,3,\eps}=r^{a,3,\eps}(x,t)$ and $r^{a,4,\eps}=r^{a,4,\eps}(x,t)$ only contain non-negative powers of $\eps$ and are smooth and bounded in $(x,t)$.
\end{itemize}
\end{lemma}
\begin{proof}
Again the proof is a long and awkward application of the It\^o formula. The definition of all correctors leads to the cancellation of a lot of terms.
\end{proof}

\vspace{0.3cm}
Using the definition of $J_0$, $J_1$ and $\delta$, we conclude that
%$$J_1+1 - \al/2 > 0, \quad \delta+1-\al/2 = \al/2 -1 > 0,$$
the absolutely continuous term
$$\eps^{-\al/2} \left[\eps^{J_1+1} r^{a,1,\eps} + \eps^{\delta+1} r^{a,2,\eps} + \eps^{\delta+1} r^{a,3,\eps} +  \eps^{(2J_0+3)\delta -1} r^{a,4,\eps}\right]$$
 tends to zero as $\eps\to0$  in probability\footnote{Even in $\bL^q(\Omega)$ for $q \geq 1$.} in $C(0,T;L^p(\R^n))$ for any $p\geq 1$.

%
%Again from Section \ref{ssect:aux_pb} and the definition of $J_0$, the absolutely continuous term
%$$\eps^{-\al/2} \left[ \right]$$
%will tend to zero as $\eps$ tends to zero in probability in $C(0,T;L^p(\R^n))$ for any $p\geq 1$.

Assume again that we represent $u^\eps$ as follows:
\begin{equation*}
u^\eps = \cE^\eps_1 + \cE^\eps_2 + \eps^{\al/2} R^\eps
\end{equation*}
where $\cE^\eps_1$ is given by \eqref{eq:operator_E_1} and $\cE^\eps_2$ by \eqref{eq:operator_E_2}. Recall that from the definition of $\delta$, $J_0$ and $J_1$, we have
$$\nu =  \min( J_1+1, J_0\delta , \delta + 1) > \al/2.$$
From the definition of $\phi^k$ given by \eqref{eq:def_phi_k}, we obtain the development:
\begin{align*}%\label{eq:dev_u_eps_bis}
u^\eps(x,t) & =   u^0(x,t) +\sum_{k=1}^{J_1} \eps^k \left[ v^k(x,t) + \sum_{\ell=1}^k \chi^{\ell-1} \left( \frac{x}{\eps}\right) \partial^{\ell}_x v^{k-\ell} \left(x,t\right)  \right]  \\ \nonumber
& +  \sum_{k=1}^{J_0} \eps^{k\delta} u^k(x,t)  + \eps^{\al/2}R^\eps(x,t) + \eps^{\al/2} \widetilde R^\eps(x,t),
\end{align*}
As was shown in Section \ref{ssect:aux_pb}, the residual $\widetilde R^{\eps}(x,t) $ converges to zero when $\eps$ goes to zero (at least) in probability in $C(0,T;L^p(\R^n))$ for any $p\geq 1$.

Let us state our result on $R^\eps$.
\begin{proposition} \label{prop:dev_al}
The discrepancy $R^\eps$ can be split in four parts:
\begin{equation*} %\label{eq:decomp_rest}
R^\eps = r^\eps + \hat r^\eps  +  \widetilde r^\eps+ \rho^\eps
\end{equation*}
such that
\begin{itemize}
\item the dynamics of $r^\eps$ contains the terms with large parameters:
\begin{eqnarray} \label{eq:SPDE_eps}
&& d r^\eps =  (\cA^\eps r^\eps) dt + \sum_{k=1}^{N_0} \eps^{k\delta-\al/2} w^k(x,t) dt - \frac{1}{\eps} \Bigg[ \kappa^1_y\left( \frac{x}{\eps},\xi_{t/\eps^\al}\right) u^0_x(x,t) \\ \nonumber
&&\hspace{4cm} \left. + \sum_{k=1}^{N_0} \eps^{k\delta} \Theta^k_y\left( \frac{x}{\eps},\xi_{t/\eps^\al},x,t\right)  \right]  \sigma(\xi_{t/\eps^\al}) dB_t, \\ \nonumber
&& r^\eps(x,0) = 0,
\end{eqnarray}
\item the dynamics of $\hat r^\eps$ contains all other terms:
\begin{eqnarray}  \label{eq:def_hat_rho_eps}
d \hat r^\eps & = & (\cA^\eps \hat r^\eps) dt - \left(\kappa^1_y u^1_{x} + \kappa^2_y u^0_{xx} \right) \sigma(\xi_{t/ \eps^\al}) dB_t \\ \nonumber
\hat r^\eps(x,0) & = & 0,
\end{eqnarray}
\item $\rho^\eps$ satisfies
\begin{eqnarray} \label{eq:SPDE_eps_1}
d \rho^\eps & = & (\cA^\eps \rho^\eps) dt,
\end{eqnarray}
and has the initial condition $\rho^\eps(x,0) =  R_0^\eps(x)$ with
\begin{eqnarray} \label{eq:init_cond_R_eps}
\rho^\eps(x,0) & = & R_0^\eps(x) \\ \nonumber
&= &-\sum_{k=1}^{J_1}  \eps^{k-\al/2} \left[ \cI_k  + \sum_{\ell=1}^k \cI_{k-\ell}  \chi^{\ell-1}\left( \frac{x}{\eps}\right) \right] \partial_x^k u^0(x,0) .
\end{eqnarray}
\item $\widetilde r^\eps$ contains all negligible terms and satisfies
$$\mE \| \widetilde r^\eps\|^2_{L^2(\R \times (0,T))} \leq C \eps^{ \nu}.$$
\end{itemize}
Moreover if $r^\eps$ has a limit, then $\hat r^\eps$ defined by \eqref{eq:def_hat_rho_eps} converges to zero.
\end{proposition}
\begin{proof}
Indeed, gathering Lemmas \ref{lmm:first_abs_cont_part_al_geq_4} and \ref{lem:second_abs_term}, the remainder $R^\eps$ satisfies:
\begin{eqnarray} \label{eq:reste}
d R^\eps & =  & (\cA^\eps R^\eps) dt - \cM^\eps \sigma(\xi_{t/\eps^\al}) dB_t - \sum_{k=1}^{N_0}\eps^{k\delta-\al/2} w^k (x,t)dt \\ \nonumber
& - &  (m^\eps q(\xi_{t/\eps^\al}) dW_t  + r^{a,\eps} dt)
\end{eqnarray}
with
\begin{itemize}
\item a martingale term $\cM^\eps$ with ``large parameters'':
$$\cM^\eps =  \left[  \frac{1}{\eps} \Phi^1_y + \frac{1}{\eps} \sum_{k=1}^{N_0} \eps^{k\delta} \Theta^k_y +  \Phi^2_y   \right] $$
\item a martingale term $m^\eps$ of order smaller than $\eps^{\al/2}$:
$$m^\eps  =  \left[\eps \sum_{k=3}^{J_1+2} \eps^{k-3} \Phi^k_y  + \eps^{\delta}   \sum_{k=1}^{N_0} \eps^{(k-1)\delta} \Psi^k_y\right],$$
\item and a negligible term $r^{a,\eps}$ of order $\cO(\eps^\nu)$ (that is, convergence to zero in strong topology).
\end{itemize}
With $\cI_0=1$, $\cI_1=0$, we have for any $k\geq 1$, $v^k(x,0) = \cI_k \partial_x^k u^0(x,0)$ and thus:
\begin{eqnarray*}
&&\sum_{k=1}^{J_1} \eps^k \left[ v^k(x,0) + \sum_{\ell=1}^k \chi^{\ell-1}\left( \frac{x}{\eps}\right) \partial^{\ell}_x v^{k-\ell} \left(x,0\right)  \right] \\
&& \quad = \sum_{k=1}^{J_1} \eps^k \left[ \cI_k \partial_x^k u^0(x,0) + \sum_{\ell=1}^k\cI_{k-\ell} \chi^{\ell-1}\left( \frac{x}{\eps}\right)  \partial^{k}_x u^{0} \left(x,0\right)  \right] \\
&& \quad = \sum_{k=1}^{J_1} \eps^k \left[ \cI_k  + \sum_{\ell=1}^k\cI_{k-\ell} \chi^{\ell-1} \left( \frac{x}{\eps}\right) \right] \partial_x^k u^0(x,0) = -\eps^{\al/2}R_0^\eps(x).
\end{eqnarray*}
 Since $u^0(x,0)=u^\eps(x,0)=\imath(x)$, we deduce that $R^\eps$ satisfies the initial condition:
\begin{eqnarray*}
&&R^\eps(x,0) = R_0^\eps(x) + r_0^\eps(x)\\ \nonumber
&&= -\sum_{k=1}^{J_1}  \eps^{k-\al/2} \left[ \cI_k  + \sum_{\ell=1}^k \cI_{k-\ell}  \chi^{\ell-1}\left( \frac{x}{\eps}\right) \right] \partial_x^k u^0(x,0)+ r_0^\eps(x) .
\end{eqnarray*}
where $r^\eps_0 = \cO(\eps^{\nu} )$. By the linearity of equation \eqref{eq:reste} we obtain the desired decomposition.
In particular $\widetilde r^\eps$ satisfies
\begin{eqnarray*}
d \widetilde r^\eps & = & (\cA^\eps \widetilde r^\eps) dt - (m^\eps q(\xi_{t/\eps^\al}) dW_t  + r^{a,\eps} dt), \quad \widetilde r^\eps(x,0)=r_0^\eps(x).
\end{eqnarray*}
Very classical arguments and standard parabolic estimates prove that $\widetilde r^\eps$ goes to zero when $\eps$ goes to zero: $\mE \| \widetilde r^\eps\|^2_{L^2(\R \times (0,T))} \leq C \eps^{ \nu}.$

For the last assertion, we can apply the result concerning $r^\eps$ to $(1/\eps) \hat r^\eps$. This last quantity will converge in the same sense as $r^\eps$.
\end{proof}

%In Section \ref{sect:conv_sing_mart_part} we detail the precise behaviour of $r^\eps$.
Note that the term $R^\eps_0$ contains negative powers of $\eps$. Hence this term could have a priori a non trivial contribution in the behaviour of $\rho^\eps$ and thus of $R^\eps$. Nevertheless in Section \ref{sect:rest_init_cond} we show that we can choose the constants $\cI_k$ in such a way that the remainder $\rho^\eps$ converges strongly in $L^2(\R \times (0,T))$ to zero.

\section{Limit behaviour of the remainder} \label{sect:conv_sing_mart_part}
%------------------

The goal of this section is to pass to the limit in the term  $r^\eps$.
Let us recall that in any case $r^\eps(x,0)=0$ and that the evolution of $r^\eps$ is given by \eqref{eq:dyn_r_eps_al<2} for $\al < 2$ and by \eqref{eq:SPDE_eps} for $\al > 2$. We can summarize these equations as follows:
\begin{eqnarray}  \label{eq:mart_problem}
dr^\eps &=& (\cA^\eps r^\eps) dt-\sum\limits_{k=1}^{K_0} \eps^{k\delta-\al/2} w^k(x,t) \, dt  \\ \nonumber
&-&\eps^{\varpi-1} \sum\limits_{k=0}^{K_0}\eps^{k\delta}\Upsilon^k\Big(\frac x\eps,\xi_\frac t{\eps^\alpha}\Big) u_x^{k}(x,t)\sigma(\xi_\frac t{\eps^\alpha}) \,dB_t
\end{eqnarray}
where
\begin{itemize}
\item $K_0$ is an integer such that: $(K_0+1)\delta \geq \max(2,\al/2)$,
\item $\Upsilon^k$ are defined on $\bT \times \R^n$ and smooth functions satisfying \eqref{eq:est_chi_al<2} and such that $\langle \Upsilon^k \rangle = 0$,
\item $\varpi = \max( 2-\alpha, 0)\geq 0$.
\end{itemize}
If $\varpi > 1$, i.e. $\al < 1$, we obtain a stronger convergence result (see Part \ref{ssect:conv_al<1}).

Let $\widetilde \Upsilon^k$ be a function such that $\partial_z \widetilde \Upsilon^k = \Upsilon^k$ with zero mean value w.r.t. $z$. And $\widetilde w^k_x = w^k$. Define $v^\eps$ as the solution of
\begin{eqnarray*}  \nonumber
dv^{\eps}  & = & a \left( \frac{x}{\eps},\xi_{t/\eps^\al}\right)v^{\eps}_{xx}  dt - \sum_{k=1}^{K_0} \eps^{k\delta-\al/2} \widetilde w^k(x,t) dt\\ %\label{eq:dynamic_theta}
& - & \eps^{\varpi} \sum_{k=0}^{K_0} \eps^{k\delta}    \widetilde \Upsilon^k\left( \frac{x}{\eps},\xi_{t/\eps^\al}\right) u^{k}_x(x,t) \sigma(\xi_\frac t{\eps^\alpha}) dB_t.  %\\
%& = &  a\left( \frac{x}{\eps},\xi_{t/\eps^\al}\right)v^{\eps}_{xx}  dt- \sum_{k=1}^{L} \eps^{k\delta-\al/2} \widetilde w^k(x,t) dt  \\ \nonumber
%&- &\eps^{\varpi}  \sum_{k=0}^{L} \eps^{k\delta} \widetilde \cG^k\left( \frac{x}{\eps},\xi_{t/\eps^\al},x,t\right)  dB_t
\end{eqnarray*}
In the rest of this section, $\widetilde \cG^k$, for $k=0,\ldots,K_0$, are defined by
\begin{eqnarray*} %\label{eq:def_tilde_G_k}
\widetilde \cG^k\left( z,y,x,t\right) & = & \widetilde \Upsilon^k \left( z,y\right) u^{k}_x(x,t).
\end{eqnarray*}
%\red{Here $\varpi = 2-\alpha =\delta>0$. But in the case $\alpha > 2$, then $\varpi=0$.}
Then $v^\eps_x = r^{\eps} + \check{v}^\eps$ where
\begin{eqnarray*}
&& d \check{v}^{\eps} = (\cA^\eps \check{v}^\eps)  dt -  \eps^{\varpi} \sum_{k=0}^{K_0} \eps^{k\delta}   \widetilde \Upsilon^k\left( \frac{x}{\eps},\xi_{t/\eps^\al}\right) u^{k}_{xx}(x,t)\sigma(\xi_\frac t{\eps^\alpha}) dB_t.
\end{eqnarray*}
Since $r^\eps(x,0)=0$, we assume that $v^\eps(x,0)=\check{v}^\eps(x,0)=0$. We are going to show that $\check{v}^\eps$ asymptotically vanishes and thus  the limit  behaviour of $r^{\eps}$ depends only on $v^\eps$.
\begin{lemma}
$\check{v}^\eps$ tends to 0 in $\bL^2(\R \times (0,T))$  in probability.
\end{lemma}
%The proof of this Lemma will be given later on in this section.

\begin{proof}
{\color{blue}
Clearly, it suffices to prove that the statement of Lemma holds for the function $\check{v}_m^\eps$ being a solution to the following
problem:
\begin{equation*}
d \check{v}_m^{\eps} = {(\cA^\eps \check{v}_m^\eps)}  dt -  \eps^{\varpi}    \widetilde \Upsilon^0\left( \frac{x}{\eps},\xi_{t/\eps^\al}\right) u^{0}_{xx}(x,t)\sigma(\xi_\frac t{\eps^\alpha}) dB_t,\quad \check{v}_m^{\eps}(x,0)=0.
\end{equation*}
If  $\alpha<2$, then $\varpi>0$ and the required statement is evident. So we only dwell on the case $\alpha>2$. In this case
$\varpi=0$.
Observe that the expectation of $\bL^2(\R \times (0,T))$ norm of $\check{v}_m^\eps$ is bounded.  Indeed,
denoting $\mathcal N^0_m= \eps^\varpi\widetilde \Upsilon^0\left( \frac{x}{\eps},\xi_{t/\eps^\al}\right) u^{0}_{xx}(x,t)\sigma(\xi_\frac t{\eps^\alpha})$, and applying Ito's formula to $\|\check{v}_m^\eps\|^2_{L^2(\mathbb R)}$ we obtain
$$
d\|\check{v}_m^\eps\|^2_{L^2(\mathbb R)}=-2\big({\textstyle a(\frac \cdot\eps,\xi_{t/\eps^\alpha})}\check{v}_{m,x}^\eps,
\check{v}_{m,x}^\eps\big)_{L^2(\mathbb R)}dt+2\big(\check{v}_{m}^\eps,
\mathcal{N}^0_m\big)_{L^2(\mathbb R)}dB_t
$$
$$
+\big(\mathcal{N}^0_m,
\mathcal{N}^0_m\big)_{L^2(\mathbb R)}dt
$$
Taking the expectation and using the BDG inequality yields
$$
\mathbf{E}\Big(\|\check{v}_m^\eps\|^2_{L^\infty(0,T;L^2(\mathbb R)})
+\int_0^T\|\check{v}_{m,x}^\eps\|^2_{L^2(\mathbb R)}dt\Big)\leq C.
$$
Next, for an arbitrary $\varphi\in C_0^\infty(\mathbb R)$ consider a function
$$
\boldsymbol{\mu}_t^\eps=\big(\check{v}_m^\eps, \varphi\big)+\sum\limits_{k=1}^{K_0}\eps^{1+k\delta}
\big({\textstyle \mathtt{Q}^{k-1}(\frac\cdot\eps,\xi_{t/\eps^\alpha})}\varphi_x,\check{v}_m^\eps\big)
+\sum\limits_{k=0}^{K_0}\eps^{1+k\delta}
\big({\textstyle \boldsymbol{\chi}^{k}(\frac\cdot\eps)}\varphi_x,\check{v}_m^\eps\big),
$$
where $\boldsymbol{\chi}^0$ coincides with  ${\chi}^0$ defined in
\eqref{eq:def_chi_0_al>2}, \
$\boldsymbol{\chi}^k$, $k\geq 1$, solves the equation
$$
(\overline a\boldsymbol{\chi}^k_z)_z=-\overline{A\mathtt{Q}^{k-1}},\qquad \langle\boldsymbol{\chi}^k\rangle=0,
$$
$\mathtt{Q}^0$ and $\mathtt{Q}^k$ are solutions of
$$
\mathcal{L}\mathtt{Q}^0=-a_z(z,y)+\overline{a}_z(z)+(A-\overline{A})\boldsymbol{\chi}^0,\qquad
\langle\mathtt{Q}^0\rangle=0,
$$
and
$$
\mathcal{L}\mathtt{Q}^k=-(A-\overline{A})\boldsymbol{\chi}^k-(A\mathtt{Q}^{k-1}-\overline{A\mathtt{Q}^{k-1}}),\qquad
\langle\mathtt{Q}^k\rangle=0,
$$
By the Ito formula, considering the definition of $\boldsymbol{\chi}^k$ and $\mathtt{Q}^k$, we obtain
$$
{\textstyle
d\boldsymbol{\mu}_t^\eps=\big[\big(a^\eps,\varphi_{xx}\check v_m^\eps\big)-\big([a^\eps\boldsymbol{\chi}^0\big(\frac\cdot\eps\big)]_z,\varphi_{xx}\check v_m^\eps\big)-
\big(a^\eps_z\boldsymbol{\chi}^0\big(\frac\cdot\eps\big),\varphi_{xx}\check v_m^\eps\big)
\big]dt
}
$$
$$
+\big(\mathcal{N}^0_m,\varphi\big)dB_t +d\zeta^\eps(t),
$$
where $\mathbf{E}\big(\sup\limits_{0\leq t\leq T}|\zeta^\eps(t)|\big)\longrightarrow 0$ as $\eps\to0$. Since
$\langle\mathcal{N}^0_m\rangle=0$, 
$$
\sup\limits_{0\leq t\leq T}\Big|\int_0^t (\mathcal{N}^0_m, \varphi)dB_s\Big| \to 0
$$
in probability.
Therefore, $\check v^\eps_m$ converges in law in $C_w(0,T;L^2(\mathbb R))\cap L^2(0,T;H_w^1(\mathbb R))$ to a solution of the problem
$$
\partial_t \check v^0_m=a^{\rm eff}\check v^0_{m, xx},    \qquad \check v^0_{m}(x,0)=0.
$$
Since the only solution of this problem is equal to zero, the required statement follows.
}
%\red{Note that the expectation of $\bL^2(\R \times (0,T))$ norm of $\check{v}^\eps$ is bounded. For $\al < 2$, it tends to zero since $\varpi > 0$. For $\al > 2$, let us split $\check{v}^{\eps} =\eps X^\eps + Y^\eps$ with
%\begin{eqnarray*}
%&& d X^{\eps} =  (\cA^\eps X^\eps)    dt - \sum_{k=1}^{K_0} \eps^{k\delta-\al/2} \check{w}^k(x,t) dt\\
%&& -  \eps^{\varpi-1} \sum_{k=0}^{K_0} \eps^{k\delta}   \widetilde \Upsilon^k\left( \frac{x}{\eps},\xi_{t/\eps^\al}\right) u^{k}_{xx}(x,t)\sigma(\xi_\frac t{\eps^\alpha}) dB_t
%\end{eqnarray*}
%and
%\begin{eqnarray*}
%&& d Y^{\eps} =  (\cA^\eps Y^\eps)    dt + \eps \sum_{k=1}^{K_0} \eps^{k\delta-\al/2} \check{w}^k(x,t) dt.
%\end{eqnarray*}
%Since $k\delta - \al/2 + 1 > 0$ for any $k\geq 1$, $Y^\eps$ strongly converges to zero. Observe that $X^\eps$ satisfies an SPDE  similar to that
%satisfied by $r^\eps$. Thus we can iterate the arguments used above and define $V^\eps$ such that $V^\eps_x = X^\eps +  \check{X}^\eps$, where the expectation of  $\bL^2(\R \times (0,T))$ norm of $\check{X}^\eps$ is bounded. Thereby if $v^\eps$ is bounded in $H^1$ (see Proposition \ref{prop:behaviour_v_eps}), by similarity $V^\eps_x$ is bounded in $\bL^2$ (with a suitable definition of the functions $\check{w}^k$) and
%$$\check{v}^{\eps} =\eps (V^\eps_x -\check{X}^\eps)  + Y^\eps$$
%tends to zero in probability in $\bL^2(\R \times (0,T))$.
%}
%
%If $r^\eps$ converges in law in $\bL^2(\R \times (0,T))$, then Slutsky's theorem gives the convergence for $\check{v}^\eps$
%to zero in probability.
\end{proof}

\subsection{Construction of correctors}
%------------------

The correctors $P^k$ and $Q^k$ are given by the equations
\begin{equation} \label{eq:def_P_0_Q_0}
(\bar a(z) P^0(z) )_{zz} = 0, \quad \langle P^0 \rangle = 1,\qquad \cL Q^0 = ((\bar a - a) P^0)_{zz}
\end{equation}
and for $k\geq 1$
\begin{align} \label{eq:def_P_k}
(\bar a(z) P^k(z) )_{zz} & =  -  \overline{(Q^{k-1} a)_{zz}}, \quad \langle P^k \rangle = 1,\\ \label{eq:def_Q_k}
\cL Q^k(z,y) & =  ((\bar a - a) P^k)_{zz} + \overline{(Q^{k-1} a)_{zz}} - (Q^{k-1} a)_{zz}.
\end{align}

\begin{lemma}
The functions $P^k$ are smooth periodic functions defined on $\mathbb{T}$ and $Q^k$ are smooth functions on $\R\times \R^n$, bounded in $z$ and of at most linear growth w.r.t. $y$.
\end{lemma}
\begin{proof}
Indeed let us begin with $k=0$ (Equation \eqref{eq:def_P_0_Q_0}). $P_0$ satisfies:
\begin{equation*}
(\bar a(z) P^0 )_{zz} = 0, \quad \langle P^0 \rangle = 1.
\end{equation*}
Hence $P^0 = 1 + \chi^1_z$ and classical computations for the dimension one show that
$$P^0(z) = \aeff \frac{1}{\bar a(z)}.$$
Next for $Q^0$ we have:
$$ \cL Q^0 (z,y)= ((\bar a - a) P^0)_{zz}.$$
Again here $z$ is a parameter of the equation. The right-hand side has zero mean value w.r.t. $y$ and is a smooth bounded function of the two variables $y$ and $z$. As was already shown there exists a unique solution $Q^0$ which is smooth w.r.t. $y$ and $z$,  bounded w.r.t. $z$ and of at most linear growth w.r.t. $y$. Then from \eqref{eq:def_P_k} and \eqref{eq:def_Q_k} and by recursion we obtain the desired result.
\end{proof}

Let us introduce the following notations: for $k\geq 0$ and $m\geq 0$
\begin{equation*} %\label{eq:def_Xi_k_m}
\Xi^{k,m} (x,t) = \overline{ \langle Q^k_y\left( .,. \right) \widetilde \cG^m \left(.,.,x,t\right)  \rangle \sigma (.) }.
\end{equation*}
The correctors $U^{k,\ell}$ are solution of the problem:
\begin{equation}\label{eq:def_U_k}
\cL U^{k,\ell}(x,t,y) +2 \langle Q^k_y\left( .,y \right) \widetilde \cG^\ell \left(.,y,x,t\right)  \rangle \sigma (y) - \Xi^{k,\ell} (x,t) = 0.
\end{equation}
Moreover
$$\widetilde P^k_z = P^k - \langle P^k\rangle = P^k -1\quad \mbox{and} \quad \widetilde Q^k_z = Q^k.$$
%Recall that $\widetilde \cG^k$ are defined by \eqref{eq:def_tilde_G_k}:
%\begin{equation*}
%\widetilde \cG^k\left( z,y,x,t\right) =  \sum_{j=0}^k \widetilde \Upsilon^j\left( z,y\right) u^{k-j}_x(x,t).
%\end{equation*}
%The functions $u^\ell$ satisfy the estimate \eqref{est_uuj}:
%\begin{equation*}
% |D^{\bf j}u^\ell | \leq C_{{\bf j},N}(1+|x|)^{-N},
%\end{equation*}
The correctors $\widetilde \Upsilon^{k}$ verify the inequality \eqref{eq:est_kappa_1}
\begin{equation*}
|\widetilde \Upsilon^{k}(z,y)| \leq C(1+|y|^p), \quad \forall (z,y) \in \mathbb T \times \R^n.
\end{equation*}
Thereby from the previous lemma, we deduce  the following result:
\begin{lemma}
The functions $U^{k,\ell}$ given by \eqref{eq:def_U_k} are well defined and smooth and also exhibit at most
polynomial growth in $y$.
%satisfy \eqref{eq:est_kappa_1}.
\end{lemma}

Finally we define the constant $\cC_{k,m}$, $0\leq k\leq m$ by
\begin{equation} \label{eq:cC_k_m}
 \cC_{k,m} =  \overline{ \langle Q^{\ell}_y\left( .,. \right) \widetilde \Upsilon^{k} \left(.,.\right) \rangle \sigma (.) } .
 \end{equation}
%Here
%$\langle Q^{\ell-j}_y\left( .,. \right) \widetilde \Upsilon^{j-k} \left(.,.\right) \rangle$ is the mean over $\mathbb{T}$ of the scalar product of the two vector-valued functions $ Q^{\ell-j}_y$ and $\widetilde \Upsilon^{j-k}$.

\subsection{Convergence of the antiderivative $v^\eps$}
%-------------

We first prove the boundedness of $v^\eps$ in $H^1(\R)$, then obtain a tightness result and finally we identify the weak limit.

\subsubsection{Bound in $H^1(\R)$ for $v^\eps$} \label{ssect:mild_sol}
%------------------

Let us consider the quantity
\begin{eqnarray*}
\cV^\eps_t & = & \sum_{k=0}^{K_0} \eps^{k\delta} \left[ \blla P^k\left( \frac{.}{\eps}\right) v^\eps(.,t) ,  v^{\eps}(.,t) \brra + \eps^\delta \blla Q^k\left( \frac{.}{\eps},\xi_{t/\eps^\alpha} \right)  v^\eps(.,t), v^{\eps}(.,t) \brra \right] \\
& + & \eps^\varpi \sum_{k=0}^{K_0} \sum_{\ell=0}^{K_0}\eps^{(k+\ell+1)\delta+\al/2}\blla U^{k,\ell}(.,t,\xi_{t/\eps^\al}),v^\eps(.,t)\brra
\end{eqnarray*}
where $P^k$, $Q^k$ and $U^{k,\ell}$ are the correctors defined respectively by \eqref{eq:def_P_0_Q_0}, \eqref{eq:def_P_k}, \eqref{eq:def_Q_k} and \eqref{eq:def_U_k}. The bracket $\lla .,.\rra$ stands for the scalar product in $L^2(\R)$.

Again by It\^o's formula and the very definition of all correctors, we deduce:
\begin{lemma}\label{lmm:dynamic_V_eps}
The quantity $\cV^\eps$ satisfies:
\begin{eqnarray}\label{eq:V_eps_final}
&&d\cV^\eps_t = \cB^\eps_t dt + M^\eps_t \sigma(\xi_\frac t{\eps^\alpha}) dB_t \\ \nonumber
&& \quad +  \eps^{(K_0+1)\delta - 2} \blla (Q^{K_0}\left( \frac{.}{\eps},\xi_{t/\eps^\alpha} \right) a^\eps)_{zz}   v^\eps(.,t), v^\eps(.,t) \brra dt\\ \nonumber
&& \quad -   2 \sum_{k=0}^{K_0} \eps^{k\delta}  \blla\left(P^k\left( \frac{.}{\eps}\right) + \eps^\delta Q^k\left( \frac{.}{\eps},\xi_{t/\eps^\alpha} \right) \right) v_x^\eps(.,t) , a^\eps v^\eps_{x}(.,t)   \brra  dt \\ \nonumber
&& \quad  +  2 \sum_{k=0}^{K_0} \sum_{m =0}^{K_0-1} \eps^{(k+m+1) \delta-\al/2}  \lla   \widetilde w^{m+1} \left( .,t\right)+ \eps^{\varpi} \Xi^{k,m} (.,t), v^\eps(.,t) \rra dt \\ \nonumber
&& \quad  +   \lla \mathcal{NT}^{1,\eps}(.,t) , v^\eps(.,t) \rra  dt - \lla\mathcal{NT}^{2,\eps}(.,t) , v^\eps_{x}(.,t) \rra dt,
\end{eqnarray}
where:
\begin{itemize}
\item $M^\eps$ %is defined by \eqref{eq:def_martin_term_H_1} and
stands for the integrand in the stochastic integral w.r.t. the Brownian motion $B$:
\begin{equation} \label{eq:def_martin_term_H_1}
M^\eps_t =  \sum_{k=0}^{K_0} \eps^{(k+1)\delta} \eps^{-\al/2}\blla  Q^k_y\left( \frac{.}{\eps},\xi_{t/\eps^\alpha} \right) v^\eps\left( .,t \right) ,v^\eps(.,t) \brra   + \eps^{\varpi} \widetilde M^\eps_t
\end{equation}
where in $\widetilde M^\eps$, all powers of $\eps$ are non-negative.
\item the term $\cB^\eps$ %(given by \eqref{eq:def_bounded_term_H_1})
does not depend on $v^\eps$ and is bounded uniformly w.r.t. $\eps$: for any $p \geq 1$, there exists a constant $\hat p \geq 1$ such that
$$\mE |\cB^\eps_t |^p \leq C\mE (1+|\xi_{t/\eps^\al}|^{\hat p})\, ;$$
\item there exists $\nu > 0$ such that for any $N>0$ the two terms $\mathcal{NT}^{1,\eps}$ and $ \mathcal{NT}^{2,\eps}$ (and their derivatives w.r.t. $x$) %, defined resp. by \eqref{eq:def_neglig_1} and \eqref{eq:def_neglig_2},
satisfy
\begin{equation} \label{eq:estim_neglig_terms}
\mE |\mathcal{NT}^{.,\eps}(x,t) |^p \leq \frac{C \eps^{\nu} }{(1+|x|)^{N}} \mE (1+|\xi_{t/\eps^\al}|^{\hat p}).
\end{equation}
\end{itemize}
The constant $C$ here does not depend on $\eps$.
\end{lemma}
%\begin{proof}
%The proof is based on the It\^o formula and the definitions of the correctors and is postponed in the appendix.
%\end{proof}

%Now %we define the constant
%%$$\mathcal{K}^\eps = \sum_{k=0}^{N_2-1} \eps^k \langle P^k \rangle = 1 + \eps \sum_{k=0}^{N_2-2} \eps^{k} \langle P^{k+1} \rangle $$%if $k+m \geq M$, $(k+m+1) \delta - \al/2 \geq (M+1)\delta - \al/2 > 0$, since
%%$$(M+1)\delta > (\delta/2) +1$$
%%and
%if $M=2(N_2-1)$, we define for $\ell=0,\ldots,M$
%\begin{equation}\label{eq:def_Z_ell}
%Z^\ell = \sum_{n=0}^{\ell \wedge (N_2-1)} \Xi^{\ell-n,n}.
%\end{equation}

The last double sum in \eqref{eq:V_eps_final} can be written as:
\begin{eqnarray*}
&& 2 \sum_{k=0}^{K_0} \sum_{m =0}^{K_0-1}  \eps^{(k+m+1)\delta - \al /2} \lla \eps^\varpi \Xi^{k,m}(.,t)+  \widetilde w^{m+1}(.,t), v^\eps(.,t)\rra\\
&& \quad = 2 \sum_{\ell=0}^{K_0-1}  \eps^{(\ell+1)\delta - \al /2} \sum_{m=0}^{\ell} \lla \eps^{\varpi} \Xi^{\ell-m,m}(.,t) + \widetilde w^{m+1}(.,t), v^\eps(.,t)\rra \\
&& \quad + \lla \mathcal{NT}^{3,\eps}(.,t), v^\eps(.,t)\rra,
\end{eqnarray*}
with
\begin{equation} \label{eq:def_neglig_3}
 \mathcal{NT}^{3,\eps} (.,t) =   2 \sum_{k+m\geq K_0}  \eps^{(k+m+1)\delta - \al /2} \left( \eps^\varpi \Xi^{k,m}(.,t)+  \widetilde w^{m+1}(.,t)\right).
 \end{equation}
Again since $(K_0+1)\delta > \al /2 $, all powers of $\eps$ in \eqref{eq:def_neglig_3} are positive and thus $ \mathcal{NT}^{3,\eps} $ also verifies \eqref{eq:estim_neglig_terms}.

If for $\ell \geq 0$
\begin{equation*}%\label{eq:def_Z_ell}
Z^\ell (x,t)= \sum_{m=0}^{\ell} \Xi^{\ell-m,m}(x,t),
\end{equation*}
then
\begin{eqnarray*}
&& 2 \sum_{k=0}^{K_0} \sum_{m =0}^{K_0-1}  \eps^{(k+m+1)\delta - \al /2} \lla \eps^\varpi \Xi^{k,m}(.,t)+  \widetilde w^{m+1}(.,t), v^\eps(.,t)\rra\\
&& \quad = 2 \sum_{\ell=0}^{K_0-1}  \eps^{(\ell+1)\delta - \al /2}  \lla \eps^{\varpi} Z^{\ell}(.,t) + \sum_{m=0}^{\ell} \widetilde w^{m+1}(.,t), v^\eps(.,t)\rra \\
&& \quad + \lla \mathcal{NT}^{3,\eps}(.,t), v^\eps(.,t)\rra.
\end{eqnarray*}
Here we distinguish two cases.
\begin{description}
\item[Case $\al < 2$] Then $\varpi = \delta$ and choosing $\widetilde w^1=0$ yields
\begin{eqnarray*}
&&  \sum_{\ell=0}^{K_0-1}  \eps^{(\ell+1)\delta - \al /2}  \lla \eps^{\varpi} Z^{\ell}(.,t) + \sum_{m=0}^{\ell} \widetilde w^{m+1}(.,t), v^\eps(.,t)\rra \\
&& \quad = \sum_{\ell=0}^{K_0-1}  \eps^{(\ell+2)\delta - \al /2}  \lla  Z^{\ell}(.,t) + \sum_{m=0}^{\ell} \widetilde w^{m+2}(.,t), v^\eps(.,t)\rra .
\end{eqnarray*}
Let us remark that from the definition of the sequences $w^k$ and $\cC_{k,m}$ by \eqref{eq:def_w_k_al<2} and \eqref{eq:cC_k_m}, we have:
\begin{eqnarray*}
Z^0(x,t) & = &  \overline{ \langle Q^0_y\left( .,. \right) \widetilde \Upsilon^0\left(.,.\right)  \rangle \sigma (.) } \ u^0_x(x,t) \\
& = & - \cC_{0,0} u^0_x(x,t) = -\widetilde w^2(x,t).
\end{eqnarray*}

And for $\ell = 2,\ldots, K_0-1$
\begin{eqnarray*}
\widetilde w^{\ell+2} &= &   -  \sum_{m=0}^{\ell} \cC_{\ell,m} u^m_x(x,t) - \sum_{m=2}^{\ell+1}  \widetilde w^{m}\\
&= &   -  \sum_{m=0}^{\ell} \overline{ \langle Q^{\ell}_y\left( .,. \right) \widetilde \Upsilon^{m} \left(.,.\right) \rangle \sigma (.) } u^{m}_x(x,t) - \sum_{m=0}^{\ell-1}   \widetilde w^{m+2}\\
%&= &   -  \sum_{m=0}^{\ell} \sum_{j=0}^{m} \overline{ \langle Q^{\ell-m}_y\left( .,. \right) \widetilde \Upsilon^{j} \left(.,.\right) \rangle \sigma (.) } u^{m-j}_x(x,t) - \sum_{m=0}^{\ell-1}   \widetilde w^{m+2}\\
& = &- \sum_{m=0}^{\ell } \overline{ \langle Q^{\ell}_y\left( .,. \right) \widetilde \cG^m \left(.,.,x,t\right)  \rangle \sigma (.) }- \sum_{m=0}^{\ell-1}  \widetilde w^{m+2} = - Z^\ell - \sum_{m=0}^{\ell-1}  \widetilde w^{m+2}.%\\
%&= & - \sum_{n=0}^{\ell } \Xi^{\ell-n,n} - \sum_{n=0}^{\ell-1}  \widetilde w^{n+1}.
\end{eqnarray*}
Thereby we obtain immediately that for any $\ell = 0,\ldots,K_0-1$
$$Z^\ell(.,t)+  \sum_{k=0}^{\ell}  \widetilde w^{k+2}(.,t) =0. $$
\item[Case $\al > 2$] Then $\varpi = 0$ and the same arguments lead to
$$Z^\ell(.,t)+  \sum_{k=0}^{\ell}  \widetilde w^{k+1}(.,t) =0. $$
\end{description}
In both cases, the equation \eqref{eq:V_eps_final} can be written as:
\begin{align}\label{eq:V_eps_final_2}
&\quad d\cV^\eps_t = \cB^\eps_t dt + M^\eps_t \sigma(\xi_\frac t{\eps^\alpha})dB_t \\ \nonumber
& \quad +  \eps^{(K_0+1)\delta - 2} \blla (Q^{K_0}\left( \frac{.}{\eps},\xi_{t/\eps^\alpha} \right) a^\eps)_{zz}   v^\eps(.,t), v^\eps(.,t) \brra dt\\ \nonumber
& \quad -   2 \sum_{k=0}^{K_0} \eps^{k\delta}  \blla\left(P^k\left( \frac{.}{\eps}\right) + \eps^\delta Q^k\left( \frac{.}{\eps},\xi_{t/\eps^\alpha} \right) \right) v_x^\eps(.,t) , a^\eps v^\eps_{x}(.,t)   \brra  dt \\ \nonumber
& \quad  +   \lla \mathcal{NT}^{1,\eps}(.,t) + \mathcal{NT}^{3,\eps}(.,t), v^\eps(.,t) \rra  dt - \lla\mathcal{NT}^{2,\eps}(.,t) , v^\eps_{x}(.,t) \rra dt.
\end{align}

\begin{proposition} \label{prop:behaviour_v_eps}
The quantity $v^\eps$ is bounded in $L^2((0,T)\times \Omega;H^1(\R))$, uniformly w.r.t. $\eps$: there exists a constant $C_{H^1}$ independent of $\eps$ such that
\begin{equation} \label{eq:bound_H1}
\mE \int_0^T \|v^\eps(.,t) \|_{H^1(\R)}^2 dt \leq C_{H^1} .
\end{equation}
Moreover the expectation of $L^\infty(0,T;L^2(\R))$ norm of $v^\eps$  is also bounded: there exists a constant $C_{L^\infty}$ again independent of $\eps$ such that
\begin{equation} \label{eq:Linfty_estim_v_eps}
\mE \left[ \sup_{t\in [0,T]} \|v^\eps(.,t) \|_{L^2(\R)}^2  \right] \leq C_{L^\infty} .
\end{equation}
\end{proposition}
\begin{proof}
Using \eqref{eq:V_eps_final_2}, we have for any $s \in [0,T]$
\begin{eqnarray*}
&& \cV^\eps_s + 2 \sum_{k=0}^{K_0} \int_0^s \eps^{k\delta} \left[ \lla P^k a^\eps  v_x^\eps, v^\eps_{x} \rra + \eps^\delta \lla Q^k  v_x^\eps,  a^\eps v^\eps_{x} \rra \right] dt \\ \nonumber
&& \quad =\int_0^s \cB^\eps_t dt +  \int_0^s M^\eps_t \sigma(\xi_\frac t{\eps^\alpha})dB_t +  \eps^{(K_0+1)\delta - 2}  \int_0^s  \lla (Q^{K_0} a^\eps)_{zz} v^\eps , v^\eps \rra dt\\ \nonumber
&& \quad +  \int_0^s \lla \mathcal{NT}^{1,\eps}(.,t)+\mathcal{NT}^{3,\eps}(.,t), v^\eps(.,t)\rra dt - \int_0^s  \lla \mathcal{NT}^{2,\eps}(.,t) , v^\eps_{x}(.,t)\rra dt.
\end{eqnarray*}
Remember that $\cV^\eps_0 = 0$ since $ v^\eps(.,0)= 0$. Hence
\begin{eqnarray*}
&&  \lla P^0 v^\eps,v^{\eps} \rra  + 2 \int_0^s  \lla P^0 a^\eps  v_x^\eps, v^\eps_{x} \rra dt  \\
&& \qquad + \sum_{k=1}^{K_0} \eps^{k\delta}\left[ \lla P^k v^\eps, v^{\eps} \rra + 2 \int_0^s  \lla P^k a^\eps  v_x^\eps, v^\eps_{x} \rra dt \right]\\
&& \qquad + \sum_{k=0}^{K_0} \eps^{(k+1)\delta }\left[ \lla Q^k  v^\eps, v^{\eps} \rra +2\int_0^s \lla Q^k  v_x^\eps,  a^\eps v^\eps_{x} \rra dt   \right] \\
&& \qquad +\eps^{\varpi} \sum_{k=0}^{K_0} \sum_{\ell=0}^{K_0}\eps^{(k+\ell+1)\delta+\al/2} \lla U^{k,\ell},v^\eps\rra \\
&& \quad = \int_0^s \cB^\eps_t dt +  \int_0^s M^\eps_t \sigma(\xi_\frac t{\eps^\alpha})dB_t +  \eps^{(K_0+1)\delta - 2}  \int_0^s  \lla (Q^{K_0} a^\eps)_{zz} v^\eps , v^\eps \rra dt\\ \nonumber
&& \qquad +  \int_0^s \lla \mathcal{NT}^{1,\eps}(.,t)+\mathcal{NT}^{3,\eps}(.,t), v^\eps(.,t)\rra dt - \int_0^s  \lla \mathcal{NT}^{2,\eps}(.,t) , v^\eps_{x}(.,t)\rra dt.
\end{eqnarray*}
Taking the expectation of all the terms in the last relation we deduce that there exists a constant $K$ such that for any $t \in [0,T]$
$$\mE \left[   \blla P^0\left( \frac{.}{\eps}\right)  v^\eps(.,t), v^{\eps}(.,t) \brra  + \int_0^t  \blla P^0\left( \frac{.}{\eps}\right) a^\eps(.,s)  v_x^\eps(.,s), v^\eps_{x}(.,s) \brra ds \right] \leq K.$$
Since $a$ is bounded and uniformly elliptic and $P^0(z) = \frac{\aeff}{\bar a (z)}$, this proves that $v^\eps$ is bounded in $L^2((0,T)\times \Omega;H^1(\R))$.

To obtain \eqref{eq:Linfty_estim_v_eps}, note that in the martingale term $M^\eps$ given by \eqref{eq:def_martin_term_H_1}, all powers of $\eps$ are non-negative except for the first sum:
$$\sum_{k=0}^{K_0} \eps^{(k+1)\delta-\al/2}\lla Q^k_y v^\eps ,v^\eps \rra .$$
We define $\widehat Q^k_z = Q^k_y$ (recall that $\langle Q^k_y\rangle = 0$) and integrate by parts:
$$\sum_{k=0}^{K_0} \eps^{(k+1)\delta-\al/2} \lla Q^k_y v^{\eps} , v^{\eps}\rra = - 2\sum_{k=0}^{K_0} \eps^{k\delta+ \delta - \al/2+1}  \lla \widehat Q^k v^\eps, v^{\eps}_x\rra .$$
In any case $ \delta - \al/2+1 > 0$.
The conclusion follows from the Burkholder-Davis-Gundy inequality.
\end{proof}

\subsubsection{Weak convergence of $v^\eps$}

Here we prove that the sequence $v^\eps$ is tight in $$V_T= L^2_w(0,T; H^1(\R))\cap C(0,T;L^2_w(\R)).$$
Remenber that the index $w$ means that the corresponding space is equipped with the weak topology.
For any function $\phi \in C^\infty_0(\R)$ we define
\begin{align}\label{eq:def_cV_eps}
&\qquad \widehat \cV^\eps_t  = \sum_{k=0}^{K_0} \eps^{k\delta} \left[ \blla P^k\left( \frac{.}{\eps}\right)  \phi, v^{\eps}(.,t) \brra + \eps^{\delta}  \blla Q^k\left( \frac{.}{\eps},\xi_{t/\eps^\alpha} \right)  \phi, v^{\eps}(.,t) \brra \right] \\ \nonumber
&+ \sum_{k=0}^{K_0} \eps^{k\delta+1} \left[ \blla \widehat P^k\left( \frac{.}{\eps}\right)  \phi_x, v^{\eps}(.,t) \brra + \eps^{\delta}  \blla \widehat Q^k\left( \frac{.}{\eps},\xi_{t/\eps^\alpha} \right)  \phi_x, v^{\eps}(.,t) \brra \right] \\ \nonumber
&+\eps^{\varpi} \sum_{k=0}^{K_0} \sum_{\ell=0}^{J_0}\eps^{(k+\ell+1)\delta+\al/2} \blla U^{k,\ell}(.,t,\xi_{t/\eps^\al}),\phi\brra
\end{align}
where $P^k$, $Q^k$ and $U^{k,\ell}$ are defined again by \eqref{eq:def_P_k}, \eqref{eq:def_Q_k} and \eqref{eq:def_U_k}.
\begin{lemma} \label{lmm:dynamic_V_eps_weak_conv}
If $\widehat P^k$ and $\widehat Q^k$ are solutions of:
$$(\bar a(z) \widehat P^k )_{zz} = -2 (P^k \bar a)_z, $$
and
$$\cL \widehat Q^k = 2((\bar a - a) P^k)_{z} + ((\bar a -a)\widehat P^k)_{zz}$$
then the  dynamics of $\widehat \cV^\eps$ is described by
\begin{align}\label{eq:Ito_form_calV_eps}
d\widehat \cV^\eps_t &=   \sum_{k=0}^{K_0} \eps^{k\delta}\blla  \left( \eps^\delta Q^k\phi + \eps^{\delta+1} \widehat Q^k \phi_x \right), a^\eps v^\eps_{xx}  \brra dt \\ \nonumber
& + \sum_{k=0}^{K_0} \eps^{k\delta}   \blla (P^k a^\eps) \phi_{xx} + 2  (\widehat P^k a^\eps)_{z} \phi_x+\eps (\widehat P^k a^\eps) \phi_{xxx},  v^\eps \brra dt \\ \nonumber
&+  \lla \mathcal{NT}^{3,\eps} (.,t), \phi \rra dt +\widehat \cB^\eps_t dt + \widehat M^\eps_t \sigma(\xi_\frac t{\eps^\alpha}) dB_t .
\end{align}
where the terms $\mathcal{NT}^{3,\eps}$ and $ \widehat \cB^\eps$ verify \eqref{eq:estim_neglig_terms}. The stochastic integrand is given by:
\begin{align}\label{eq:mart_int_weak_conv}
& \quad \widehat M^\eps_t = \sum_{k=0}^{K_0} \eps^{(k+1)\delta-\al/2} \blla Q^k_y \phi + \eps \widehat Q^k_y \phi_x, v^{\eps}(.,t) \brra \\ \nonumber
&\ + \eps^{\varpi}\sum_{k=0}^{K_0}\sum_{m=0}^{K_0}  \eps^{(k+m)\delta} \blla P^k \phi + \eps^\delta Q^k\phi + \eps  \widehat P^k  \phi_x + \eps^{\delta+1}\widehat Q^k \phi_x,  \widetilde \cG^m  \brra  \\ \nonumber
& \ + \eps^{\varpi} \sum_{k=0}^{K_0} \sum_{\ell=0}^{K_0}\eps^{(k+\ell+1)\delta} \lla U_y^{k,\ell}(.,t,\xi_{t/\eps^\al}),\phi\rra .
\end{align}
\end{lemma}
\begin{proof}
The arguments are very similar to those in  the proof of the lemma \ref{lmm:dynamic_V_eps} and are again based on the It\^o formula. % and are postponed in the appendix.
\end{proof}

Now we get a tightness result.
\begin{proposition}
There exist two constants $\nu > 0$ and $C > 0$ such that for any $\eps$ and any $0\leq t \leq \tau \leq T$,
\begin{equation} \label{eq:cont_estim_unif_eps}
\mE\left[ \sup_{t \leq s \leq \tau} |\widehat \cV^\eps_s - \widehat\cV^\eps_t| \right] \leq C \sqrt{ |\tau-t| } + C\eps^{\nu} .
\end{equation}
\end{proposition}
\begin{proof}
Indeed the absolutely continuous terms of order $\eps^0$ in \eqref{eq:Ito_form_calV_eps} are
$$\lla (P^0 a^\eps) \phi_{xx},  v^\eps \rra dt +  2 \lla (\widehat P^0 a^\eps)_{z} \phi_{xx},  v^\eps \rra dt.$$
And from \eqref{eq:mart_int_weak_conv},
\begin{eqnarray*}
\widehat M^\eps_t & = &\sum_{k=0}^{K_0} \eps^{(k+1)\delta-\al/2}   \lla  Q^k_y \phi , v^{\eps}(.,t) \rra \\
&+& \eps\sum_{k=0}^{K_0} \eps^{(k+1)\delta-\al/2} \blla  \widehat Q^k_y \phi_x, v^{\eps}(.,t) \brra \\ \nonumber
&+&\eps^{\varpi}\sum_{k=0}^{K_0}\sum_{m=0}^{K_0}  \eps^{(k+m)\delta} \blla P^k \phi + \eps^\delta Q^k\phi + \eps  \widehat P^k  \phi_x + \eps^{\delta+1}\widehat Q^k \phi_x,  \widetilde \cG^m  \brra  \\ \nonumber
& + &\eps^{\varpi} \sum_{k=0}^{K_0} \sum_{\ell=0}^{K_0}\eps^{(k+\ell+1)\delta} \blla U_y^{k,\ell}(.,t,\xi_{t/\eps^\al}),\phi\brra .
\end{eqnarray*}
The last three sums are multiplied by a positive power of $\eps$, since $\delta -\al/2 +1 >0$. Note that later, for $\al > 2$, we  have to keep the first term $\lla P^0 \phi ,  \widetilde \Upsilon^0 u^0_x  \rra$.
For the first sum, define $\mathfrak Q^k_z = Q^k_y$ (recall that $\langle Q^k_y\rangle = 0$) and make an integration by parts:
$$ \lla Q^k_y \phi , v^{\eps}\rra = - \eps \lla \mathfrak Q^k , (\phi v^{\eps})_x\rra .$$
Then
\begin{eqnarray*}
\widehat M^\eps_t & = &- \sum_{k=0}^{K_0} \eps^{(k+1)\delta+1-\al/2}   \lla  \mathfrak Q^k,( \phi  v^{\eps}(.,t))_x \rra \\
&+& \sum_{k=0}^{K_0} \eps^{(k+1)\delta+1 -\al/2} \blla  \widehat Q^k_y \phi_x, v^{\eps}(.,t) \brra \\ \nonumber
& +& \eps^{\varpi} \blla P^0 \phi ,\widetilde \Upsilon^0 u^0_x \brra +\eps^{\varpi+\delta }\sum_{0\leq k,m\leq K_0; \ k+m\geq 1}  \eps^{(k+m-1)\delta} \blla P^k \phi ,  \widetilde \cG^m  \brra \\ \nonumber
&+&\eps^{\varpi}\sum_{k=0}^{K_0}\sum_{m=0}^{K_0}  \eps^{(k+m)\delta} \blla  \eps^\delta Q^k\phi + \eps  \widehat P^k  \phi_x + \eps^{\delta+1}\widehat Q^k \phi_x,  \widetilde \cG^m  \brra  \\ \nonumber
& + &\eps^{\varpi} \sum_{k=0}^{K_0} \sum_{\ell=0}^{K_0}\eps^{(k+\ell+1)\delta} \lla U_y^{k,\ell}(.,t,\xi_{t/\eps^\al}),\phi\rra  \\
& =& \eps^{\nu}  \widetilde M^\eps_t +\eps^{\varpi} \blla P^0 \phi ,\widetilde \Upsilon^0 u^0_x \brra
\end{eqnarray*}
where $\nu = \min(\delta+1-\al/2,\varpi + \delta,\varpi+1) > 0$. In other words for any $0\leq t \leq s \leq T$:
\begin{align} \label{eq:dyn_cV}
&\qquad  \widehat \cV^\eps_s - \widehat \cV^\eps_t = \int_t^s \left[ \lla (P^0 a^\eps) \phi_{xx},  v^\eps \rra + 2 \lla (\widehat P^0 a^\eps)_{z} \phi_{xx},  v^\eps \rra \right] dr \\ \nonumber
 &\quad +\eps^{\nu}  \int_t^s \widetilde M^\eps_r \sigma(\xi_\frac t{\eps^\alpha}) dB_r +  \int_t^s \widehat \cB^\eps_r dr + \eps^\varpi \int_t^s  \blla P^0 \phi ,\widetilde \Upsilon^0 u^0_x \brra \sigma(\xi_\frac t{\eps^\alpha})dB_r.
\end{align}
Therefore,
\begin{eqnarray*}
\mE\left[ \sup_{t \leq s \leq \tau} |\cV^\eps_s - \cV^\eps_t| \right] & \leq &C  \| v^\eps\|_{L^2((0,T)\times \Omega;L^2(\R))} \times\sqrt{ |\tau-t| } \\
&+& \mE\left[ \sup_{t \leq s \leq \tau }\left| \int_t^s \widehat \cB^\eps_u du \right| \right]\\
%& + &\mE\left[ \sup_{s \leq t \leq r}\left| \int_s^t  \left( P^0 \phi,\widetilde \Upsilon^0 u^0_x \right) dW_u \right| \right]\\
& + &\eps^{\nu} \mE\left[ \sup_{t \leq s \leq \tau}\left|  \int_t^s \widetilde M^\eps_u \sigma(\xi_\frac t{\eps^\alpha})dB_u  \right| \right] \\
& + & \eps^\varpi  \mE\left[ \sup_{t \leq s \leq \tau}\left| \int_t^s  \blla P^0 \phi ,\widetilde \Upsilon^0 u^0_x \brra\sigma(\xi_\frac t{\eps^\alpha}) dB_u  \right| \right] .
\end{eqnarray*}
Note that for $\al > 2$, $\varpi = 0$. From BDG inequality
\begin{eqnarray*}
\mE\left[ \sup_{t \leq s \leq \tau}\left| \int_t^s  \blla P^0 \phi,\widetilde \Upsilon^0 u^0_x \brra \sigma(\xi_\frac t{\eps^\alpha}) dB_u \right| \right] & \leq & C \mE\left[ \left( \int_t^\tau \blla P^0 \phi,\widetilde \Upsilon^0 u^0_x \brra^2 du\right)^{1/2} \right]\\
& \leq & C \sqrt{|\tau-t|}
\end{eqnarray*}
From BDG and Young's inequalities we have
\begin{eqnarray*}
&& \mE\left[ \sup_{t \leq s \leq \tau}\left| \int_t^s  \widetilde M^\eps_u  \sigma(\xi_\frac t{\eps^\alpha})  dB_u \right| \right] \\
&& \qquad  \leq C \mE \left[ \left( \int_t^\tau \left(\widetilde M^\eps_u \right)^2  du\right)^{1/2} \right]\\
&& \qquad \leq C \mE \left[  \sup_{t\leq u\leq \tau}\left( 1 +|\xi_{u/\eps^\al}|^p \right)\right] + C\mE \left[ \int_t^\tau \| v^{\eps}\|^2_{H^1(\R)} du \right]
\end{eqnarray*}
for some $p\geq 1$. We know that for any $\beta > 0$
$$\lim_{\eps \to 0} \eps^{\beta} \mE \left[  \sup_{t\leq u\leq \tau}\left( 1 +|\xi_{u/\eps^\al}|^p \right)\right] = 0$$
(see Proposition 2.6 in \cite{CKP_2001}). Thereby since $\widehat \cB^\eps$ satisfies \eqref{eq:estim_neglig_terms}, we deduce the estimate \eqref{eq:cont_estim_unif_eps}.
\end{proof}

Therefore from \eqref{eq:bound_H1}, \eqref{eq:Linfty_estim_v_eps} and \eqref{eq:cont_estim_unif_eps}, together with Theorem 8.3 in \cite{bill:68} and Prokhorov criterium, the sequence $v^\eps$ is tight in $V_T$. Now we identify its limit as the law of the solution of a SPDE. Here we distinguish the two cases $\al < 2$ and $\al > 2$.
\begin{proposition} \label{prop:limit_ident_al<2}
For $\al < 2$, the sequence $v^\eps$ weakly converges in $V_T$ to zero.
\end{proposition}
\begin{proof}
Again let $\phi$ be a $C^\infty_0(\R)$ test function. From the definition \eqref{eq:def_cV_eps} of $\widehat \cV^\eps$, we deduce that
\begin{equation*}
\widehat \cV^\eps_t  = \blla P^0\left( \frac{.}{\eps}\right)  \phi, v^{\eps}(.,t) \brra + \eps^\nu \widehat  \cV^{\eps,b}_t
\end{equation*}
where $\widehat \cV^{\eps,b}$ is bounded in $\bL^2((0,T)\times \Omega)$. Hence since $\langle P^0\rangle =1$, if $\mathfrak{P}^0$ is such that $\mathfrak{P}^0_z= P^0 - 1$, then
\begin{equation*}
\widehat \cV^\eps_t  = \lla \phi, v^{\eps}(.,t) \rra + \eps \blla \mathfrak{P}^0_x \left( \frac{.}{\eps}\right)  \phi, v^{\eps}(.,t) \brra + \eps^\nu \widehat \cV^{\eps,b}_t
\end{equation*}
After  integration by parts, we deduce that:
\begin{equation} \label{eq:mart_pb_1}
\widehat \cV^\eps_t  = \lla \phi, v^{\eps}(.,t) \rra + \eps \blla \mathfrak{P}^0 \left( \frac{.}{\eps}\right) , (\phi v^{\eps}(.,t))_x \brra + \eps^\nu\widehat  \cV^{\eps,b}_t
\end{equation}
Since $v^\eps$ is bounded in $L^2((0,T)\times \Omega,H^1(\R))$, the middle term on the right-hand side converges to zero.

Now from \eqref{eq:cont_estim_unif_eps} the sequence $\cV^\eps$ is also tight in $C(0,T;\R)$. Recall that for $\al < 2$, $\varpi > 0$. Using \eqref{eq:dyn_cV} we have for some $\nu > 0$ and for any $0\leq t \leq s \leq T$
\begin{align}\label{eq:mart_pb_2}
\widehat \cV^\eps_s -\widehat  \cV^\eps_t & =  \int_t^s \left[ \lla (P^0 a^\eps) \phi_{xx},  v^\eps \rra + 2 \lla (\widehat P^0 a^\eps)_{z} \phi_{xx},  v^\eps \rra\right] dr \\ \nonumber
& + \eps^{\nu}  \int_t^s \widetilde M^\eps_r \sigma(\xi_\frac t{\eps^\alpha}) dB_r + \eps^{\nu}  \int_t^s \widetilde \cB^\eps_r dr.
\end{align}
For the first integral we define
$$\widetilde P^0(z,y) = P^0(z) a(z,y) - \langle P^0 a\rangle(y)+ 2(\widehat P^0 a)_{z}(z,y)$$
which has zero mean value in $z$. We can define again $\widetilde{\mathfrak{P}}^0$ such that $\widetilde{\mathfrak{P}}^0_z= \widetilde P^0$ and thus
\begin{eqnarray} \label{eq:mart_pb_3}
& &\int_t^s \blla (P^0 a^\eps + 2(\widehat P^0 a^\eps)_{z} )\phi_{xx},  v^\eps \brra  dr\\ \nonumber
&& \quad = \int_t^s \blla \langle P^0 a\rangle(\xi_{r/\eps^\al}) \phi_{xx},  v^\eps \brra  dr +\eps \int_t^s \blla \widetilde{\mathfrak{P}}^0_x \phi_{xx},  v^\eps\brra  dr \\ \nonumber
&& \quad = \int_t^s \blla \langle \overline{P^0 a}\rangle \phi_{xx},  v^\eps \brra  dr + \int_t^s \blla (\langle P^0 a\rangle(\xi_{r/\eps^\al}) - \langle \overline{P^0 a}\rangle) \phi_{xx},  v^\eps \brra  dr \\ \nonumber
&& \qquad  +\eps \int_t^s \blla\widetilde{\mathfrak{P}}^0_x \phi_{xx},  v^\eps \brra  dr.
\end{eqnarray}
For the term
$$\mathfrak{E}^\eps(t)= \int_0^t \blla (\langle P^0 a\rangle(\xi_{r/\eps^\al}) - \langle \overline{P^0 a}\rangle) \phi_{xx},  v^\eps \brra  dr,$$
the uniform bound \eqref{eq:Linfty_estim_v_eps}, together with the mixing property implied by assumption {\bf (A)}, lead to the convergence to zero of this term, a.s. and in $\bL^2(\Omega)$ by the dominated convergence theorem, uniformly w.r.t. $t \in [0,T]$.

Combining \eqref{eq:mart_pb_1}, \eqref{eq:mart_pb_2} and \eqref{eq:mart_pb_3}, we obtain for some $\nu > 0$
\begin{eqnarray*}
F_\phi(t,  v^{\eps})& = &\lla \phi, v^{\eps}(.,t) \rra  - \int_0^t \blla \langle \overline{P^0 a}\rangle \phi_{xx},  v^\eps \brra dr \\
& =  & \eps^{\nu}  \int_0^t \widetilde M^\eps_r \sigma(\xi_\frac t{\eps^\alpha})  dB_r + \eps^{\nu} \widehat \cV^{\eps,b}_t + \mathfrak{E}^\eps(t).
\end{eqnarray*}
Let $\Theta^\eps_s$ be any continuous (in the sense of the topology of $V_T$) and bounded functional of $\{v^{\eps}_\tau , \ 0 \leq  \tau \leq s\}$.
We have proved that for $0\leq s \leq t$
$$\lim_{\eps \downarrow 0} \E  \left| (F_\phi(t,  v^{\eps}) - F_\phi(s,  v^{\eps})) \Theta^\eps_s \right| = 0.$$
If we compute the quadratic variation\footnote{Denoted by $\left[  \left[ . \right]  \right]$ to be distinguishable from the mean over a period or the scalar product in $L^2$.} of the process $\widehat \cV^\eps$ we have
\begin{eqnarray*}
\left[ \! \left[ \widehat \cV^\eps \right]  \!  \right]_s -\left[  \!  \left[ \widehat \cV^\eps \right]  \!  \right]_t & = & \eps^{2\nu}  \int_t^s \left| \widetilde M^\eps_r\right|^2 \| \sigma(\xi_\frac t{\eps^\alpha}) \|^2 dr.
\end{eqnarray*}
Recall that for a vector $v \in \R^n$, $\| v \|^2 = \mbox{Trace}(v v^*)$ is the Euclidean norm.
We deduce that
\begin{equation*}
\lim_{\eps \downarrow 0} \mE \left[ (F_\phi(t,  v^{\eps}) - F_\phi(s,  v^{\eps}))^2 \Theta^\eps_s\right]  = \lim_{\eps \downarrow 0} \mE \left[ \left| \left[ \! \left[ \widehat  \cV^\eps \right]  \!  \right]_s - \left[ \! \left[ \widehat  \cV^\eps \right]  \!  \right]_t \right| \Theta^\eps_s\right]
\end{equation*}
is equal to zero. Passing to the limit, we deduce that $\{ F_\phi(t,v^0), \ 0\leq t \leq T\}$ is a square integrable martingale with respect to the natural filtration of $v^0$, with a null quadratic variation process. In other words we proved that the sequence $v^\eps$ weakly converges in $V_T$ to the unique solution $v^0$ of the PDE:
$$dv^0 = \langle \overline{P^0 a} \rangle v^0_{xx} dt$$
with initial condition zero. Hence $v^0=0$ and this achieves the proof of the Proposition.
\end{proof}

For $\al >2$, the preceding result has to be modified since $\varpi = 0$, which implies that there is a zero order term in the martingale part $\widehat M^\eps$ in \eqref{eq:Ito_form_calV_eps}.
\begin{proposition}\label{prop:limit_ident_al>2}
If $\al > 2$, the sequence $v^\eps$ weakly converges in $V_T$ to the unique solution $\widetilde r^0$ of the SPDE:
$$d\widetilde r^0 = \langle \overline{P^0 a} \rangle \widetilde r^0_{xx} dt + \left( \overline{\| \langle P^0 \widetilde \Upsilon^0 \rangle \sigma\|^2 }\right)^{1/2} u^0_x d W_t.$$
\end{proposition}
\begin{proof}
We argue almost as in the proof of Proposition \ref{prop:limit_ident_al<2}. In particular the beginning of the proof is the same. But now \eqref{eq:mart_pb_2} becomes:
%
%
%Again let $\phi$ be a $C^\infty_0(\R)$ test function. From the definition \eqref{eq:def_cV_eps} of $\cV^\eps$, we deduce that
%\begin{equation*}
%\cV^\eps_t  = \left( P^0\left( \frac{.}{\eps}\right)  \phi, \widetilde r^{\eps}(.,t) \right) + \eps^\nu \cV^{\eps,b}_t
%\end{equation*}
%where $\cV^{\eps,b}$ is bounded in $\bL^2((0,T)\times \Omega)$. Hence since $\langle P^0\rangle =1$, if $\mathfrak{P}^0$ is such that $\mathfrak{P}^0_z= P^0 - 1$, then
%\begin{equation*}
%\cV^\eps_t  = \left( \phi, \widetilde r^{\eps}(.,t) \right) + \eps \left( \mathfrak{P}^0_x \left( \frac{.}{\eps}\right)  \phi, \widetilde r^{\eps}(.,t) \right) + \eps^\nu \cV^{\eps,b}_t
%\end{equation*}
%With an integration by parts, we deduce that:
%\begin{equation} \label{eq:mart_pb_1}
%\cV^\eps_t  = \left( \phi, \widetilde r^{\eps}(.,t) \right) + \eps \left( \mathfrak{P}^0 \left( \frac{.}{\eps}\right) , (\phi \widetilde r^{\eps}(.,t))_x \right) + \eps^\nu \cV^{\eps,b}_t
%\end{equation}
%Since $\widetilde r^\eps$ is bounded in $L^2((0,T)\times \Omega,H^1(\R))$, the middle term converges to zero.
%Now from \eqref{eq:cont_estim_unif_eps} the sequence $\cV^\eps$ is also tight in $C(0,T;\R)$. And using \eqref{eq:dyn_cV} we have for some $\nu > 0$ and for any $0\leq t \leq s \leq T$
\begin{eqnarray*}%\label{eq:mart_pb_2_al>2}
&&\widehat \cV^\eps_s - \widehat \cV^\eps_t = \int_t^s \left[ \lla (P^0 a^\eps) \phi_{xx}, v^\eps \rra + 2 \lla (\widehat P^0 a^\eps)_{z} \phi_{xx},  v^\eps \rra\right] dr \\ \nonumber
&& \quad+ \int_t^s  \lla P^0 \phi,\widetilde \Upsilon^0 u^0_x \rra \sigma(\xi_\frac t{\eps^\alpha}) dB_r +\eps^{\nu}  \int_t^s \widetilde M^\eps_r \sigma(\xi_\frac t{\eps^\alpha}) dB_r +  \int_t^s \widetilde \cB^\eps_r dr.
\end{eqnarray*}
Now we obtain for some $\nu > 0$
\begin{eqnarray*}
F_\phi(t, v^{\eps})& = &\lla \phi, v^{\eps}(.,t) \rra  - \int_0^t \blla \langle \overline{P^0 a}\rangle \phi_{xx},  v^\eps \brra dr \\
& =  & \int_0^t  \lla P^0 \phi,\widetilde \Upsilon^0 u^0_x \rra \sigma(\xi_\frac t{\eps^\alpha}) dB_r + \eps^{\nu}  \int_0^t \widetilde M^\eps_r \sigma(\xi_\frac t{\eps^\alpha}) dB_r + \eps^{\nu}\widehat  \cV^{\eps,b}_t + \mathfrak{E}^\eps(t).
\end{eqnarray*}
The term
$$\mathfrak{E}^\eps(t)= \int_0^t \blla (\langle P^0 a\rangle(\xi_{r/\eps^\al}) - \langle \overline{P^0 a}\rangle) \phi_{xx},  v^\eps \brra dr,$$
can be handled as before and we have proved that for any continuous (in the sense of the topology of $V_T$) and bounded functional $\Theta^\eps_s$ of $\{v^{\eps}_\tau , \ 0 \leq  \tau \leq s\}$ and any $0\leq s \leq t$
$$\lim_{\eps \downarrow 0} \mE  \left| (F_\phi(t,  \widetilde r^{\eps}) - F_\phi(s,  \widetilde r^{\eps})) \Theta^\eps_s \right| = 0.$$
Concerning the quadratic variation of the process $\widehat \cV^\eps$, we have
\begin{eqnarray*}
\left[ \! \left[ \widehat \cV^\eps\right] \!\right]_s -\left[ \! \left[ \widehat \cV^\eps\right] \!\right]_t & = &  \int_t^s \left(  \left| \lla P^0 ,\widetilde \Upsilon^0 \phi u^0_x\rra \right|^2  + \eps^{2\nu}  \left|\widetilde M^\eps_r\right|^2 \right)  \| \sigma(\xi_\frac t{\eps^\alpha}) \|^2dr.
\end{eqnarray*}
Recall that for a vector $v \in \R^n$, $\| v \|^2 = \mbox{Trace}(v v^*)$ is the Euclidean norm. Again if we denote
$$\mathfrak{G}(y) = \langle P^0 \widetilde \Upsilon^0 \rangle (y), \qquad \mathfrak{Q}^0_z(z,y) = P^0(z)  \widetilde \Upsilon^0(z,y) - \mathfrak{G}(y)$$
the mean of $P^0 \widetilde \Upsilon^0$ w.r.t. $z$ and the periodic antiderivative of $P^0  \widetilde \Upsilon^0-\mathfrak{G}$,
then
\begin{eqnarray*}
&&\int_t^s \left(  \left| \lla P^0 ,\widetilde \Upsilon^0 \phi u^0_x\rra \right|^2  \right)  \| \sigma(\xi_\frac t{\eps^\alpha}) \|^2dr= \int_t^s  \left\| \lla P^0 \widetilde \Upsilon^0, \phi u^0_x\rra \sigma(\xi_\frac t{\eps^\alpha}) \right\|^2dr \\
&&\quad =\int_t^s  \left\| \lla P^0 \widetilde \Upsilon^0-\mathfrak{G}(\xi_\frac t{\eps^\alpha}), \phi u^0_x\rra \sigma(\xi_\frac t{\eps^\alpha}) \right\|^2dr\\
&& \qquad +\int_t^s  \left\|  \lla  \phi, u^0_x \rra  \mathfrak{G}(\xi_{r/\eps^\al})   \sigma(\xi_\frac t{\eps^\alpha})  \right\|^2 dr.
\end{eqnarray*}
And we have
$$\int_t^s  \left\| \lla P^0 \widetilde \Upsilon^0-\mathfrak{G}(\xi_\frac t{\eps^\alpha}), \phi u^0_x\rra \sigma(\xi_\frac t{\eps^\alpha}) \right\|^2dr \leq \eps^2 \int_t^s\left\|   \lla \mathfrak{Q}^0_x  \phi, u^0_x \rra   \sigma(\xi_\frac t{\eps^\alpha})\right\|^2dr.$$
Moreover again using assumption {\bf (A)}, we obtain that
$$\int_t^s  \left\|  \lla  \phi, u^0_x \rra  \mathfrak{G}(\xi_{r/\eps^\al})   \sigma(\xi_\frac t{\eps^\alpha})  \right\|^2 dr $$
converges a.s. and in $\bL^2(\Omega)$ to
$$ \int_t^s \overline{ \left\| \mathfrak{G}(\xi_{r/\eps^\al})  \sigma(\xi_\frac t{\eps^\alpha}) \lla \phi, u^0_x \rra \right\|^2} dr = \int_t^s \overline{ \left\| \langle P^0 \widetilde \Upsilon^0 \rangle \sigma \right\|^2} \lla \phi, u^0_x \rra^2dr .$$
We deduce that
\begin{equation*}
\lim_{\eps \downarrow 0} \mE \left[ (F_\phi(t,  v^{\eps}) - F_\phi(s, v^{\eps}))^2 \Theta^\eps_s\right]  = \lim_{\eps \downarrow 0} \mE \left( \left| \left[ \! \left[ \widehat  \cV^\eps\right] \!\right]_s -\left[ \! \left[ \widehat  \cV^\eps\right] \!\right]_t \right| \Theta^\eps_s\right)
\end{equation*}
is equal to
$$ \int_t^s \overline{ \left\| \langle P^0 \widetilde \Upsilon^0 \rangle \sigma \right\|^2} \lla \phi, u^0_x \rra^2dr.$$
Passing to the limit, we deduce that $\{ F_\phi(t,\widetilde r^0), \ 0\leq t \leq T\}$ is a square integrable martingale with respect to the natural filtration of $\widetilde r^0$, with the associated quadratic variation process given by
$$ \overline{\| \langle P^0 \widetilde \Upsilon^0 \rangle\sigma \|^2}   \lla \phi, u^0_x \rra^2 t.$$
This achieves the proof of the Proposition.
\end{proof}

Let us remark to conclude this part that $P^0 = 1 + \chi^0_z$, thus $\langle \overline{P^0 a} \rangle  = \aeff$. Moreover
$$\langle P^0 \widetilde\Upsilon^0 \rangle = -\langle \chi^0 \Upsilon^0 \rangle.$$
Hence
$$d\widetilde r^0 = \aeff  \widetilde r^0_{xx} dt + \left( \overline{\|\langle \chi^0 \Upsilon^0 \rangle \sigma \|^2} \right)^{1/2} u^0_x dW_t.$$

\subsubsection{Conclusion}
%--------------

Now we know that there exists a constant $C$ independent of $\eps$ such that
$$\mE \left( \|v^\eps_x\|^2_{L^2((0,T)\times \R)} \right) \leq C.$$
By Tchebychev's inequality for any $\delta > 0$, there exists a constant $K$ such that
$$\Pro( \|v^\eps_x\|^2_{L^2((0,T)\times \R)} \geq K) \leq C^2 / K^2 \leq \delta$$
provided $K$ is large enough. In other words, $v^\eps_x$ is tight for the weak topology on $L^2((0,T)\times \R)$. Using the dense set of $C^\infty_0$ functions and Propositions \ref{prop:limit_ident_al<2} and \ref{prop:limit_ident_al>2}, we deduce that %for $\al < 2$,
$r^\eps$ weakly converges to the solution of:
\begin{itemize}
\item For $\al < 2$:
$$dr^0 = \aeff r^0_{xx} dt,$$
with initial value 0, that is, $r^0=0$.
\item For $\al > 2$:
$$dr^0 = \aeff r^0_{xx} dt + \left( \overline{\| \langle \chi^0 \Upsilon^0 \rangle \sigma \|^2} \right)^{1/2} u^0_{xx} dW_t$$
again with initial value zero.
\end{itemize}

The proof of Theorem \ref{thm:main_result} is now complete in the case $\al < 2$, using Propositions \ref{prop:formal_exp_al<2} and \ref{prop:weak_conv_al<2} and the preceding results on the convergence of $r^\eps$. For $\al > 2$, using Proposition \ref{prop:dev_al}, the proof will be complete after the study of $\rho^\eps$, which is the aim of the next section \ref{sect:rest_init_cond}. Before, let us consider the case $\al < 1$, for which an easier proof can be done.

\subsection{Case $\alpha< 1$} \label{ssect:conv_al<1}
%-----------------

Here the assumption that $d=1$ is unnecessary for our arguments. In the problem \eqref{eq:mart_problem}, we now have $\varpi-1 =1-\alpha > 0$. Let us take $w^k \equiv 0$ for any $k$ :
\begin{eqnarray*}
dr^\eps &=& (\cA^\eps r^\eps) dt-\sum\limits_{k=1}^{K_0} \eps^{k\delta-\al/2} w^k(x,t) \ dt  \\ \nonumber
&-&\eps^{\varpi-1} \sum\limits_{k=0}^{K_0}\eps^{k\delta}\Upsilon^k\Big(\frac x\eps,\xi_\frac t{\eps^\alpha}\Big) u_x^{k}(x,t)\sigma(\xi_\frac t{\eps^\alpha}) \,dB_t \\
& =&  (\cA^\eps r^\eps) dt + \eps^{1-\alpha} \Theta^\eps\Big(\frac x\eps,\xi_\frac t{\eps^\alpha} ,x, t \Big) \,dB_t.
\end{eqnarray*}
Let us define
$$v^\eps_t = \int_{\R^d} r^\eps(x,t)^2 dx = \|r^\eps(\cdot,t)\|^2_{L^2(\R^d)}.$$
It\^o's formula leads to
\begin{eqnarray*}
v^\eps_t  & = & \int_0^t \int_{\R^d} r^\eps(x,s) \mathrm{div} \Big[{\rm a}\Big(\frac x\eps,\xi_\frac s{\eps^\alpha}\Big)\nabla r^\eps(x,s)\Big] dx ds \\
& +&  \eps^{(1-\alpha)} \int_0^t \int_{\R^d} r^\eps(x,s) \Theta^\eps\Big(\frac x\eps,\xi_\frac t{\eps^\alpha} ,x, t \Big)dx \,dB_s \\
& + & \eps^{2(1-\alpha)} \int_0^t \int_{\R^d} \left\| \Theta^\eps\Big(\frac x\eps,\xi_\frac t{\eps^\alpha} ,x, t \Big)\right\|^2 dx \,ds.
\end{eqnarray*}
An integration by part shows that
\begin{eqnarray*}
&& v^\eps_t + \int_0^t \int_{\R^d} \nabla r^\eps(x,s)  \Big[{\rm a}\Big(\frac x\eps,\xi_\frac s{\eps^\alpha}\Big)\nabla r^\eps(x,s)\Big] dx ds \\
&&\quad  =  \eps^{(1-\alpha)} \int_0^t \int_{\R^d} r^\eps(x,s) \Theta^\eps\Big(\frac x\eps,\xi_\frac t{\eps^\alpha} ,x, t \Big)dx \,dB_s \\
&& \quad + \eps^{2(1-\alpha)} \int_0^t \int_{\R^d} \left\| \Theta^\eps\Big(\frac x\eps,\xi_\frac t{\eps^\alpha} ,x, t \Big)\right\|^2 dx \,ds.
\end{eqnarray*}
From Condition \ref{a3}, taking the expectation, there exists a constant $C$ independent of $\eps$ such that
\begin{equation*} %\label{eq:H1_boundedness_V3}
\mE \int_0^T \left\| \nabla r^\eps(\cdot,s) \right\|^2_{L^2(\R^d)} ds \leq C \eps^{2(1-\alpha)}.
\end{equation*}
Moreover by Burkholder-Davis-Gundy inequality, we have
\begin{equation*}% \label{eq:Linfty_boundedness_V3}
\mE \left[ \sup_{t\in[0,T]}  v^\eps_t \right] = \mE \left[ \sup_{t\in[0,T]} \|r^\eps(\cdot,t)\|^2_{L^2(\R^d)} \right] \leq C \eps^{2(1-\alpha)}.
\end{equation*}
Hence if $\alpha < 1$, the convergence of $r^\eps$ to zero holds in $\bL^2([0,T]\times \Omega;H^1(\R^d))$ (and in $\bL^\infty([0,T];L^2(\R^d))$ in mean w.r.t. $\omega$).

\section{Role of the initial condition in the discrepancy }

%\section{Role of the initial condition \eqref{eq:init_cond_R_eps} of $R^\eps$}
\label{sect:rest_init_cond}
%----------------

Let us note that this part\footnote{Let us emphasize that all results of this section hold in $d >1$, that is for $z\in \bT^d$.} only concerns the case $\al > 2$ and the behavior of $\rho^\eps$. Recall the setting concerning $\rho^\eps$. It satisfies:
$$d \rho^\eps = (\cA^\eps \rho^\eps) dt$$
with initial condition \eqref{eq:init_cond_R_eps}:
$$\rho^\eps(x,0) = -\sum_{k=1}^{J_1}  \eps^{k-\al/2} \left[ \cI_k  + \sum_{\ell=1}^k \cI_{k-\ell}  \chi^{\ell-1}\left( \frac{x}{\eps}\right) \right] \partial_x^k u^0(x,0).$$
By linearity we can write:
\begin{equation*}%\label{eq:lin_decomp_rho}
\rho^\eps (x,t)= \sum_{k=1}^{J_1} \rho^{k,\eps}(x,t)
\end{equation*}
where the functions $\rho^{k,\eps}$ have the same dynamics \eqref{eq:SPDE_eps_1}, $d \rho^{k,\eps} = (\cA^\eps \rho^{k,\eps}) dt,$
but with initial condition
\begin{equation*}
\rho^{k,\eps}(x,0) = - \eps^{k-\al/2} \left[ \cI_k  + \sum_{\ell =1}^k \cI_{k-\ell } \chi^{\ell-1} \left( \frac{x}{\eps}\right)  \right] \partial^{k}_x u^{0} \left(x,0\right)  .
\end{equation*}
Recall that from \eqref{eq:u_0}, $u^0$ is a smooth function such that
%\begin{equation}\label{eq:u_0}
$u^0_t = \aeff u^0_{xx}$, with initial condition $u^0(x,0) = \imath(x).$
%\end{equation}

\vspace{0.5cm}
To lighten the notation, let us fix $k=1,\ldots,J_1$ and define $\varrho^\eps = \varrho^{k,\eps}$ as the solution of \eqref{eq:SPDE_eps_1} with initial condition
$$\varrho^{\eps}(x,0)  = \left[ \cI_k  + \sum_{\ell =1}^k \cI_{k-\ell }  \chi^{\ell-1} \left( \frac{x}{\eps}\right)   \right] \partial^{k}_x u^{0} \left(x,0\right)=  A_k\left( \frac{x}{\eps}\right) \partial^{k}_x u^{0} \left(x,0\right).$$
Thus $\rho^{k,\eps} = - \eps^{k-\al/2} \varrho^{\eps} = - \eps^{k-\al/2} \varrho^{k,\eps}$.

\begin{lemma}
The function $\varrho^\eps$ admits the following expansion:
\begin{align} \label{eq:expansion_varrho}
\varrho^{\eps}(x,t) & = \beta^{0,\eps} \left(  \frac{x}{\eps},\frac{t}{\eps^2} \right) \partial^{k}_x u^{0} \left(x,t\right)  \\ \nonumber
& +  \sum_{\ell = 1}^{J_1 - k } \eps^\ell \left[ \widehat m^{\ell-1,\eps}\left(  \frac{t}{\eps^2} \right) + \beta^{\ell,\eps} \left(  \frac{x}{\eps},\frac{t}{\eps^2} \right) \right] \partial^{k+\ell}_x u^{0} \left(x,t\right) \\ \nonumber
& +  \Gamma^{J_1-k,\eps}(x,t).
\end{align}
%The initial functions $\beta^{0,\eps}$, $\beta^{1,\eps}$, $\widehat m^{0,\eps}$ and $\mu^{0,\eps}$ are defined \eqref{eq:beta_0} but with $\beta^{0,\eps}(z,0) = \chi^k \left( z\right)$, \eqref{eq:beta_1}, \eqref{eq:def_m_0} and \eqref{eq:def_widehat_m_0}.
The functions $\beta^{0,\eps}$, $\widehat m^{0,\eps}$ and $\mu^{0,\eps}$ are defined by the following equations:
 \begin{equation} \label{eq:order_0}%\label{eq:beta_0}
 \left\{
 \begin{array}{rl}
\partial_t \beta^{0,\eps}(z,t) & =  (\mathfrak{a}^\eps\beta^{0,\eps}_z)_z , \qquad \beta^{0,\eps}(z,0) = A^k \left( z\right) \\ %\label{eq:def_m_0}
m^{0,\eps}(t) & =  \langle  a \left(  .,\xi_{t/\eps^\delta} \right) \beta^{0,\eps}_z \left(  .,t \right) \rangle\\% \label{eq:def_widehat_m_0}
\partial_t \widehat m^{0,\eps}(t) & =  m^{0,\eps}(t), \quad  \widehat m^{0,\eps}(0)=0%\\ \label{eq:def_mu_0}
%\mu^{0,\eps} & = & a^\eps \beta^{0,\eps}_z - m^{0,\eps}.
%\end{eqnarray}
%The next correctors $\beta^{1,\eps}$, $\widehat m^{1,\eps}$ and $\mu^{1,\eps}$ satisfy
%\begin{eqnarray}
\\% \label{eq:def_mu_0_2}
\mu^{0,\eps}(z,t) & = \mathfrak{a}^\eps \beta^{0,\eps}_z(z,t) - m^{0,\eps}(t) \\ %\label{eq:beta_1}
\partial_t \beta^{1,\eps}(z,t)  & =  (\mathfrak{a}^\eps\beta^{1,\eps}_z)_z + (\mu^{0,\eps}+ (\mathfrak{a}^\eps\beta^{0,\eps})_z ).
\end{array}\right. \end{equation}
with $\beta^{1,\eps}(z,0)=0$. The other quantities are given by:
 \begin{equation} \label{eq:order_1}%\label{eq:beta_0}
 \left\{
 \begin{array}{rl} %\label{eq:def_m_1}
m^{1,\eps}(t) & =  \langle  a \left(  .,\xi_{t/\eps^\delta} \right) \beta^{1,\eps}_z \left(  .,t \right) \rangle + \langle (\mathfrak{a}^\eps-\aeff) \beta^{0,\eps} \rangle, \\ %\label{eq:def_widehat_m_1}
\partial_t \widehat m^{1,\eps}(t) & =  m^{1,\eps}(t),\\  %\label{eq:def_mu_1}
\mu^{1,\eps}(z,t) & =  \mathfrak{a}^\eps \beta^{1,\eps}_z - m^{1,\eps}(t) + (\mathfrak{a}^\eps-\aeff) \beta^{0,\eps}.
\end{array}\right. \end{equation}
And for $\ell \geq 2$, the relations are defined recursively by:
 \begin{equation} \label{eq:order_ell}%\label{eq:beta_0}
 \left\{
 \begin{array}{rl} % \label{eq:def_beta_ell}
\partial_t \beta^{\ell,\eps} &  =  (\mathfrak{a}^\eps\beta^{\ell,\eps}_z)_z  +
(\widehat m^{\ell-2,\eps}(t)\mathfrak{a}^\eps_z + \mu^{\ell-1,\eps} + (\mathfrak{a}^\eps \beta^{\ell-1,\eps})_z ), \\ %\nonumber
\beta^{\ell,\eps}(z,0) &=  0, \\ %\label{eq:def_m_ell}
m^{\ell,\eps}(t) & =  \langle  \mathfrak{a}^\eps \beta^{\ell,\eps}_z \rangle+ \langle (\mathfrak{a}^\eps-\aeff) (\widehat m^{\ell-2,\eps}+\beta^{\ell-1,\eps}) \rangle, \\  %\label{eq:def_widehat_m_ell}
\partial_t \widehat m^{\ell,\eps}(t) & =  m^{\ell,\eps}(t),\\ %\label{eq:def_mu_ell}
\mu^{\ell,\eps}(z,t) & =  \mathfrak{a}^\eps \beta^{\ell,\eps}_z - m^{\ell,\eps}(t) + (\mathfrak{a}^\eps-\aeff) (\widehat m^{\ell-2,\eps}+\beta^{\ell-1,\eps}) .
\end{array}\right. \end{equation}
The last term in expansion \eqref{eq:expansion_varrho} is of order $\eps^{\nu}$ with $\nu > \al/2-k$.
\end{lemma}
Let us emphasize here that all terms defined in this lemma depend on $k$.

\vspace{0.3cm}
\noindent \begin{proof}
%The beginning of the proof is the same as the proof of Lemma \ref{lmm:decomposition_init_cond}.
Let us define on $\mathbb{T}\times (0,\infty)$, $\beta^{0,\eps}$ as in \eqref{eq:order_0}. Since $A_k$ is periodic, $\beta^{0,\eps}$ is well-defined.
%\begin{equation}  \label{eq:theta_n_0}
%\partial_t \beta^{0,\eps}  = (a(z,\xi_{t/\eps^\delta})\beta^{0,\eps}_z)_z , \qquad
%\beta^{0,\eps}(z,0) = - \chi^k \left( z\right).
%\end{equation}
Let us assume that
$$\varrho^{\eps}(x,t)= \beta^{0,\eps} \left(  \frac{x}{\eps},\frac{t}{\eps^2} \right) \partial^{1}_x u^{0} \left(x,t\right) + \Gamma^{0,\eps}(x,t).$$
Then $\Gamma^{0,\eps}$ satisfies: $\Gamma^{0,\eps}(x,0)=0$ and
\begin{eqnarray*}
\partial_t \Gamma^{0,\eps}  =  (\cA^\eps \Gamma^{0,\eps}) & + & \frac{1}{\eps} \left( a^\eps \beta^{0,\eps}_z + (a^\eps \beta^{0,\eps})_z \right)\partial^{2}_x u^{0} \left(x,t\right) \\
& + &  (a^\eps- \aeff) \beta^{0,\eps} \partial^{3}_x u^{0} \left(x,t\right).
\end{eqnarray*}
In \eqref{eq:order_0}, we define $m^{0,\eps}(t)$ %by \eqref{eq:def_m_0}
as the mean value w.r.t. $z$ of the function $\mathfrak{a}^\eps \beta^{0,\eps}_z$, $\widehat m^{0,\eps}(t)$ %by \eqref{eq:def_widehat_m_0}
and $\mu^{0,\eps}$ %by \eqref{eq:def_mu_0_2}
such that the mean value of $\mu^{0,\eps}$ w.r.t. $z$ is zero. Hence we can define on $\mathbb{T} \times (0,\infty)$ the function $\beta^{1,\eps}$. % by \eqref{eq:beta_1}.
Now we assume that
\begin{eqnarray*}
\Gamma^{0,\eps}(x,t) & = & \eps \left[ \widehat m^{0,\eps}\left(  \frac{t}{\eps^2} \right)+ \beta^{1,\eps} \left(  \frac{x}{\eps},\frac{t}{\eps^2} \right) \right] \partial^{2}_x u^{0} \left(x,t\right)+ \Gamma^{1,\eps}(x,t).
\end{eqnarray*}
%And \eqref{eq:expansion_varrho} is obtained.
To study the behaviour of $\Gamma^{1,\eps}$, let us remark first that
\begin{eqnarray*}
\partial_t \Gamma^{1,\eps}  =(\cA^\eps \Gamma^{1,\eps}) & + &  \left[  \widehat m^{0,\eps}(t/\eps^2)a^\eps_z  + \left( a^\eps \beta^{1,\eps}_z + (a^\eps \beta^{1,\eps})_z \right)  \right]  \partial^{3}_x u^{0} \left(x,t\right) \\
& + & (a^\eps- \aeff) \beta^{0,\eps} \partial^{3}_x u^{0} \left(x,t\right) \\
& +& \eps   \left[ \widehat m^{0,\eps}(t/\eps^2)  +  \beta^{1,\eps}\right](a^\eps-\aeff) \partial^{4}_x u^{0} \left(x,t\right).
\end{eqnarray*}
Let us do the same trick again. Using  \eqref{eq:order_1} % \eqref{eq:def_widehat_m_1} and \eqref{eq:def_mu_1}
yields:
%We define
%\begin{eqnarray*}
%m^{1,\eps}(t) & = & \langle  \mathfrak{a}^\eps \beta^{1,\eps}_z \rangle + \langle (\mathfrak{a}^\eps-\aeff) \beta^{0,\eps} \rangle, \\
%\partial_t \widehat m^{1,\eps}(t) & = & m^{1,\eps}(t),\\
%\mu^{1,\eps} (z,t)& = & \mathfrak{a}^\eps \beta^{1,\eps}_z - m^{1,\eps}(t) + (\mathfrak{a}^\eps-\aeff) \beta^{0,\eps}.
%\end{eqnarray*}
%such that
\begin{eqnarray*}
\partial_t \Gamma^{1,\eps}  & = & (\cA^\eps \Gamma^{1,\eps}) + m^{1,\eps}(t/\eps^2)\partial^{3}_x u^{0} \left(x,t\right)\\
&  +&  \left[  \mu^{1,\eps}  \left(  \frac{x}{\eps},\frac{t}{\eps^2} \right) + \widehat m^{0,\eps}(t/\eps^2)a^\eps_z  +  (a^\eps \beta^{1,\eps})_z  \right] \partial^{3}_x u^{0} \left(x,t\right)\\
& +& \eps  \left[ \widehat m^{0,\eps}(t/\eps^2)  +  \beta^{1,\eps}\right](a^\eps-\aeff) \partial^{4}_x u^{0} \left(x,t\right).
\end{eqnarray*}
If $\beta^{2,\eps}$ is the solution on $\mathbb{T} \times (0,\infty)$ in \eqref{eq:order_ell}
% \eqref{eq:def_beta_ell} for $\ell=2$,
%\begin{equation} \label{eq:theta_n_2}
%\partial_t \beta^{k,2}  = (a(z,\xi_{t/\eps^\delta})\beta^{k,2}_z)_z + (\mu^{k,1,\eps} +\widehat m^{k,0,\eps}(t)a^\eps_z +  (a^\eps \beta^{k,1})_z )
%\end{equation}
%\begin{eqnarray*}
%\partial_t \beta^{2,\eps} &  = & (\mathfrak{a}^\eps\beta^{2,\eps}_z)_z  +
%(\widehat m^{0,\eps}(t)\mathfrak{a}^\eps_z + \mu^{1,\eps} + (\mathfrak{a}^\eps \beta^{1,\eps})_z ), \\ \nonumber
%\beta^{2,\eps}(z,0) &= & 0,
%\end{eqnarray*}
and if
\begin{eqnarray*}
\Gamma^{1,\eps}(x,t) &=&\eps^2  \left[ \widehat m^{1,\eps}\left(  \frac{t}{\eps^2} \right) + \beta^{2,\eps}\left(  \frac{x}{\eps},\frac{t}{\eps^2} \right) \right] \partial^{3}_x u^{0} \left(x,t\right) + \Gamma^{2,\eps}(x,t),
\end{eqnarray*}
then
\begin{eqnarray*}
\partial_t \Gamma^{2,\eps}  & = & (\cA^\eps \Gamma^{2,\eps}) +\eps
 \left[ \widehat m^{1,\eps}(t/\eps^2)a^\eps_z + \left( a^\eps \beta^{2,\eps}_z + (a^\eps \beta^{2,\eps})_z \right)  \right] \partial^{4}_x u^{0} \left(x,t\right)\\
& +& \eps  \left[  \widehat m^{0,\eps}(t/\eps^2)  +  \beta^{1,\eps}\right](a^\eps-\aeff) \partial^{4}_x u^{0} \left(x,t\right)\\
& + & \eps^2  \left[  \widehat m^{1,\eps}(t/\eps^2)  +  \beta^{2,\eps} \right] (a^\eps-\aeff) \partial^{5}_x u^{0} \left(x,t\right)  .
\end{eqnarray*}
%All powers of $\eps$ are greater than $\al/2-1$. Thus the proof of the Lemma is achieved.
And %Namely we write
\begin{eqnarray*}
\varrho^{\eps}(x,t) & = & \beta^{0,\eps} \left(  \frac{x}{\eps},\frac{t}{\eps^2} \right) \partial^{k}_x u^{0} \left(x,t\right) \\
& + & \eps \left[ \widehat m^{0,\eps}\left(  \frac{t}{\eps^2} \right)+ \beta^{1,\eps} \left(  \frac{x}{\eps},\frac{t}{\eps^2} \right) \right] \partial^{k+1}_x u^{0} \left(x,t\right) \\
& + & \eps^2  \left[ \widehat m^{1,\eps}\left(  \frac{t}{\eps^2} \right) + \beta^{2,\eps}\left(  \frac{x}{\eps},\frac{t}{\eps^2} \right) \right] \partial^{k+2}_x u^{0} \left(x,t\right) + \Gamma^{2,\eps}(x,t).
\end{eqnarray*}
%with
%\begin{eqnarray*}
%\partial_t \Gamma^{2,\eps}  & = & (\cA^\eps \Gamma^{2,\eps}) +\eps
% \left[ \widehat m^{1,\eps}(t)a^\eps_z + \left( a^\eps \beta^{2,\eps}_z + (a^\eps \beta^{2,\eps})_z \right)  \right] \partial^{k+3}_x u^{0} \left(x,t\right)\\
%& +& \eps  \left[  \widehat m^{0,\eps}(t)  +  \beta^{1,\eps}\right](a^\eps-\aeff) \partial^{k+3}_x u^{0} \left(x,t\right)\\
%& + & \eps^2  \left[  \widehat m^{1,\eps}(t)  +  \beta^{2,\eps} \right] (a^\eps-\aeff) \partial^{k+4}_x u^{0} \left(x,t\right)  .
%\end{eqnarray*}
And then we iterate the arguments.
%We define $m^{2,\eps}$, $\widehat m^{2,\eps}$ and $\mu^{2,\eps}$ by \eqref{eq:def_m_ell},  \eqref{eq:def_widehat_m_ell} and \eqref{eq:def_mu_ell} with $\ell=2$ and $\beta^{3,\eps}$ by \eqref{eq:def_beta_ell} for $\ell=3$
%and we put
%\begin{eqnarray*}
%\Gamma^{2,\eps}(x,t) &=& \eps^3 \left[ \widehat m^{2,\eps}\left(  \frac{t}{\eps^2} \right) + \beta^{3,\eps}\left(  \frac{x}{\eps},\frac{t}{\eps^2} \right) \right] \partial^{k+3}_x u^{0} \left(x,t\right) + \Gamma^{3,\eps}(x,t).
%\end{eqnarray*}
For $\ell= 2,\ldots,J_1 - k $, we can iterate this procedure with $\beta^{\ell}$, $m^{\ell,\eps}$, $\widehat m^{\ell,\eps}$, $\mu^{\ell,\eps}$ given by \eqref{eq:order_ell}, %\eqref{eq:def_m_ell},  \eqref{eq:def_widehat_m_ell}, \eqref{eq:def_mu_ell}
and
$$\Gamma^{\ell,\eps} = \eps^{\ell+1} \left[ \widehat m^{\ell,\eps}\left(  \frac{t}{\eps^2} \right) + \beta^{\ell+1,\eps}\left(  \frac{x}{\eps},\frac{t}{\eps^2} \right) \right] \partial^{k+\ell+1}_x u^{0} \left(x,t\right) + \Gamma^{\ell+1,\eps}(x,t).$$
The last term will be of the form
\begin{eqnarray*}
&& \Gamma^{J_1 - k ,\eps}(x,t) = \Gamma^{J_1 - k + 1,\eps}(x,t) \\
&&\quad + \eps^{J_1 - k + 1}   \left[ \widehat m^{J_1 - k,\eps}\left(  \frac{t}{\eps^2} \right) + \beta^{J_1 - k + 1,\eps}\left(  \frac{x}{\eps},\frac{t}{\eps^2} \right) \right] \partial^{J_1  + 1}_x u^{0} \left(x,t\right)
\end{eqnarray*}
and
\begin{eqnarray*}
&&\partial_t \Gamma^{J_1 - k + 1,\eps}  = (\cA^\eps \Gamma^{J_1 - k + 1,\eps}) \\
&&+ \eps^{J_1 - k }
 \left[ \widehat m^{J_1  - k ,\eps}(t/\eps^2)a^\eps_z + \left( a^\eps \beta^{J_1  - k + 1,\eps}_z + (a^\eps \beta^{k,J_1  - k + 1})_z \right)  \right] \partial^{J_1   + 2}_x u^{0} \left(x,t\right)\\
&& + \eps^{J_1 - k }  \left[ \widehat m^{J_1 - k-1 ,\eps}(t/\eps^2)  +  \beta^{J_1 - k ,\eps}\right](a^\eps-\aeff) \partial^{n+J_1 - k + 2}_x u^{0} \left(x,t\right)\\
&& +\eps^{J_1 - k + 1}  \left[ \widehat m^{J_1 - k ,\eps}(t/\eps^2)  +  \beta^{J_1 - k + 1,\eps} \right] (a^\eps-\aeff) \partial^{J_1  + 2}_x u^{0} \left(x,t\right).
\end{eqnarray*}
All powers of $\eps$ are greater than $\al/2-k$. Thus the proof of the Lemma is achieved.
\end{proof}

Now let us precise the behaviour of the correctors $\beta^{\ell,\eps}$. For $\ell=0$, since
$$A_k(z) = \cI_k +  \sum_{m=1}^k \cI_{k-m} \chi^{m-1}\left( z\right),$$
one can easily deduce that
$$\beta^{0,\eps}(z) = \cI_k + \widetilde{\beta}^{0,\eps}(z)$$
where $\widetilde{\beta}^{0,\eps}$ satisfies the same equation (see Eq. \eqref{eq:order_0}), but with initial condition a periodic function with zero mean value. The key point in the sequel is that: $\beta^{0,\eps}_z = \widetilde{\beta}^{0,\eps}_z.$ And in the definition of $m^{0,\eps}$ and $\mu^{0,\eps}$, %\eqref{eq:def_m_0} and \eqref{eq:def_mu_0_2},
only the derivative is implied. We also denote
$$\mathfrak{K}_k =  \left\| \sum_{m=1}^k \cI_{k-m}    \chi^{m-1}  \right\|^2_{L^2(\bT)}.$$
The next result is an immediate consequence of Poincar\'e's inequality.
\begin{lemma} \label{lmm:behaviour_beta_0}
There exists a constant $\mathfrak{k}$ depending only on the uniform ellipticity constant of the matrix $a$, such that
$$\forall s \geq 0,\quad  \left\| \widetilde{\beta}^{0,\eps} \left(  .,s \right) \right\|^2_{L^2(\bT)}  \leq \mathfrak{K}_k e^{-\mathfrak{k} s}.$$
\end{lemma}
%\begin{proof}
%From \eqref{eq:beta_0}, using Poincar\'e's inequality, we deduce the result.
%%there exists a constant $c$ depending only on $\lambda$ such that:
%%$$\left\| \beta^{0,\eps} \left(  .,s \right) \right\|_{\bL^2(\bT^d)}^2 \leq \left\| \chi^k \right\|_{\bL^2(\bT^d)}^2 e^{-cs} = C_1  e^{-cs} .$$
%\end{proof}
For simplicity for $\ell \geq 1$, let us rewrite the definition of $\beta^{\ell,\eps}$ (Eq. \eqref{eq:order_0} and \eqref{eq:order_ell}) as:
\begin{align*} %\label{eq:def_beta_generic}
\partial_t \beta^{\ell,\eps} &  =  (a(z,\xi_{t/\eps^\delta})\beta^{\ell,\eps}_z)_z +
(\widehat m^{\ell-2,\eps}(t)\mathfrak{a}^\eps_z + \mu^{\ell-1,\eps} + (\mathfrak{a}^\eps \beta^{\ell-1,\eps})_z )\\ \nonumber
& =  (a(z,\xi_{t/\eps^\delta})\beta^{\ell,\eps}_z)_z + \varphi^{\ell,\eps}.
\end{align*}
\begin{lemma} \label{lmm:behaviour_beta_ell}
For $\ell=1,\ldots,J_1-1$  we have:
$$\forall s \geq 0,\quad  \left\| \beta^{\ell,\eps} \left(  .,s \right) \right\|^2_{L^2(\bT)}  \leq \mathfrak{K}_k  e^{-\mathfrak{k}  s}.$$
%The constant $c$ depends only on $\lambda$ (uniform ellipticity constant of $a$). % and $C_1$ does not depend on $\eps$, but only on $\chi^1$.
\end{lemma}
\begin{proof}
%Again for simplicity let us rewrite Equations \eqref{eq:beta_0}, \eqref{eq:beta_1} as:
%\begin{eqnarray*}
%\partial_t \beta^{\ell,\eps} &  = & (a(z,\xi_{t/\eps^\delta})\beta^{\ell,\eps}_z)_z +
%(\widehat m^{\ell-2,\eps}(t)a^\eps_z + \mu^{\ell-1,\eps} + (a^\eps \beta^{\ell-1,\eps})_z )\\ \nonumber
%& = & (a(z,\xi_{t/\eps^\delta})\beta^{\ell,\eps}_z)_z + \varphi^{\ell,\eps}.
%\end{eqnarray*}
%From \eqref{eq:beta_0}, using Poincar\'e's inequality, there exists a constant $c$ depending only on $\lambda$ such that:
%$$\left\| \beta^{0,\eps} \left(  .,s \right) \right\|_{\bL^2(\bT^d)}^2 \leq \left\| \chi^k \right\|_{\bL^2(\bT^d)}^2 e^{-cs} = C_1  e^{-cs} .$$
Recall that $\lambda$ is the ellipticity constant of $a$ (Condition \ref{a4}). Again by Poincar\'e's inequality, we deduce that
\begin{equation*}% \label{eq:estim_g_n_0}
|m^{0,\eps}(t)| = \left| \langle  a^\eps \left(  .,\xi_{\frac{t}{\eps^\delta}} \right) \beta^{0,\eps}_z \left(  .,t \right) \rangle \right| \leq  \frac{\mathfrak{K}_k}{\lambda} e^{-\mathfrak{k} t}.
\end{equation*}
And
$$\varphi^{1,\eps} = \mu^{0,\eps} + (\mathfrak{a}^\eps \beta^{0,\eps})_z  = \mathfrak{a}^\eps \beta^{0,\eps}_z - m^{0,\eps} + (\mathfrak{a}^\eps\beta^{0,\eps})_z $$
satisfies a similar inequality: $\left\| \varphi^{1,\eps}(.,t)\right\|^2_{L^2(\bT)} \leq \mathfrak{K}_k  e^{-\mathfrak{k} t}.$
From its very definition, we deduce that $\left\| \beta^{1,\eps} \left(  .,s \right) \right\|_{L^2(\bT)}^2 \leq \mathfrak{K}_k   e^{-\mathfrak{k} s} .$
By recursion, this achieves the proof of the Lemma.
\end{proof}

We also have to control the terms $\widehat m^{\ell,\eps}$ for $\ell=0,1,\ldots,N_0$.

%
%Again Arguing as in the proof of Lemmas \ref{lmm:behaviour_beta_ell} and \ref{lmm:behaviour_m_ell}, by recursion we obtain that for any $\ell=1,\ldots,N_0$ we have:
%$$\forall s \geq 0,\quad  \left\| \beta^{\ell,\eps} \left(  .,s \right) \right\|^2_{\bL^2(\bT^d)}  \leq \left\| \chi^k \right\|_{\bL^2(\bT^d)}^2 e^{-c s},$$
%and
%$$\forall t \geq 0,\quad \left| \widehat m^{\ell,\eps}(t) - \int_0^\infty \langle  \bar a \left(  . \right) \widehat \beta^{\ell-1}_z \left(  .,s \right) \rangle ds  \right| = \cO(\eps^{\delta/2})$$
%

\begin{lemma} \label{lmm:behaviour_m_ell}
For any $\widetilde \delta < \delta/2$, the quantity
$$\eps^{-\widetilde \delta} \left| \widehat m^{\ell, \eps}(t/\eps^2) - \int_0^\infty \langle  \bar a \left(  . \right) \widehat \beta^{\ell}_z \left(  .,s \right) \rangle ds  \right|$$
converges in probability to zero, uniformly in time, where
\begin{equation*} % \label{eq:hat_beta_0}
\partial_t \widehat \beta^{0}  = (\bar a(z)\widehat \beta^{0}_z)_z ,\quad \widehat \beta^{0}(z,0)=\sum_{j=1}^k \cI_{k-j} \chi^{j-1}\left( z\right),
\end{equation*}
and  for any $\ell \geq 1$
\begin{equation*}  %\label{eq:hat_beta_ell}
\partial_t \widehat \beta^{\ell}  = (\bar a(z)\widehat \beta^{\ell}_z)_z +\overline{\varphi^{\ell}}(z,t), \quad \widehat \beta^{\ell}(z,0)=0.
\end{equation*}
\end{lemma}
\begin{proof}
The function $\widehat \beta^{0}$ is well defined and do not depend on $\eps$. Moreover it also satisfies
$$\left\| \widehat \beta^{0} \left(  .,s \right) \right\|_{L^2(\bT)}^2 \leq \mathfrak{K}_k  e^{-\mathfrak{k} s}  .$$
We assume that
$$\widetilde \beta^{0,\eps}(z,s) =\widehat \beta^{0}(z,s)  + \eps^{\delta} \Psi^{0,\eps}(z,s,\xi_{s/\eps^\delta}) +\mathfrak{R}^{0,\eps}.$$
Then from the definition of $\beta^{0,\eps}$ in \eqref{eq:order_0} and \eqref{eq:order_ell} we obtain
\begin{eqnarray*}
d\widetilde \beta^{0,\eps}(z,s) & = & (\bar a(z)\widehat \beta^{0}_z)_z ds + \eps^{\delta} \left[\eps^{-\delta} \cL \Psi^{0,\eps} ds + \eps^{-\delta/2} \Psi^{0,\eps}_y q(\xi_{s/\eps^\delta})  dW_s \right. \\
&& \hspace{4cm} \left.+ \Psi^{0,\eps}_s ds  \right] + d\mathfrak{R}^{0,\eps} \\
 & =& (a(z,\xi_{s/\eps^\delta})\widehat \beta^{0}_z)_z ds + \eps^\delta (a(z,\xi_{s/\eps^\delta})\Psi^{0,\eps}_z)_z ds + (a(z,\xi_{s/\eps^\delta})\mathfrak{R}^{0,\eps}_z)_z ds.
 \end{eqnarray*}
If we define $\Psi^{0,\eps}$ by:
$$\cL \Psi^{0,\eps} = ((a(z,y)-\bar a(z))\widehat \beta^{0}_z)_z,$$
the residual $\mathfrak{R}^{0,\eps}$ satisfies the equation:
$$d\mathfrak{R}^{0,\eps} = (a(z,\xi_{s/\eps^\delta})\mathfrak{R}^{0,\eps}_z)_z ds + \eps^{\delta/2}\Psi^{0,\eps}_y q(\xi_{s/\eps^\delta})  dW_s + \eps^{\delta} \mathfrak{B}_s^{0,\eps}ds$$
where $\mathfrak{B}^{0,\eps}$ is bounded. The initial condition is:
$$\mathfrak{R}^{0,\eps}(z,0) = -\eps^{\delta} \Psi^{0,\eps}.$$
Coming back to $m^{0,\eps}$ in \eqref{eq:order_0} we have
\begin{eqnarray*}
&& m^{0,\eps}(t) - \langle  \mathfrak{a}^\eps \widehat \beta^{0}_z \left(  .,t \right) \rangle = \langle  \mathfrak{a}^\eps \left[ \beta^{0}_z \left(  .,t \right) - \widehat \beta^{0}_z \left(  .,t \right)  \right] \rangle \\
&& \quad =\eps^{\delta} \langle  \mathfrak{a}^\eps \Psi_z^{0,\eps}(.,t,\xi_{t/\eps^\delta})\rangle  + \langle  \mathfrak{a}^\eps \mathfrak{R}_z^{0,\eps}(.,t,\xi_{t/\eps^\delta})\rangle
\end{eqnarray*}
Note that $\Psi_z^{0,\eps}$ is bounded in $L^2(\bT^d)$ by $\mathfrak{K}_k e^{-\mathfrak{k} t}$ and the quantity $\mathfrak{R}^{0,\eps}_z$ is bounded in any space $\bL^p(\Omega)$ by  $\eps^{\delta/2} \mathfrak{K}_k e^{-\mathfrak{k} t}$. Hence we deduce that
\begin{eqnarray*}
&& \left| \widehat m^{0,\eps}(t) - \int_0^t \langle  a\left(  .,\xi_{\frac{s}{\eps^\delta}} \right) \widehat \beta^{0}_z \left(  .,s\right) \rangle ds \right| \\
&& \quad \leq \int_0^t \left| m^{0,\eps}(s) - \langle  a \left(  .,\xi_{\frac{s}{\eps^\delta}} \right) \widehat \beta^{0}_z \left(  .,s \right) \rangle \right| ds \\
&& \quad \leq \lambda\eps^{\delta}  \int_0^t \mathfrak{K}_k e^{-\mathfrak{k} s} ds + \lambda \int_0^t  \|\mathfrak{R}_z^{0,\eps}(.,s,\xi_{s/\eps^\delta}) \|_{\bL^2(\bT)} ds .
\end{eqnarray*}
Therefore for any $p \geq 1$, there exists a constant C (independent of $\eps$) such that for any $t \geq 0$
$$\E \left( \left| \widehat m^{0,\eps}(t) -  \int_0^t  \langle  a \left(  .,\xi_{\frac{s}{\eps^\delta}} \right) \widehat \beta^{0}_z \left(  .,s\right) \rangle ds \right|^p \right) \leq C \eps^{\delta p/2}.$$
In particular the previous inequality holds when we replace $t$ by $t/\eps^2$. Moreover from the estimate of $\widehat \beta^0$, there exists a constant $C$ such that a.s. for any $\eps > 0$ and $t > 0$
%$$ \int_K^{\infty} \left| \langle  a^\eps \left(  .,\xi_{\frac{s}{\eps^\delta}} \right) \widehat \beta^{0}_z \left(  .,s \right) \rangle \right| ds \leq  Ce^{-\mathfrak{k} K} .$$
%Hence for a fixed $t > 0$,
 $$\left| \int_0^{t/\eps^2}  \langle  a \left(  .,\xi_{\frac{s}{\eps^\delta}} \right) \widehat \beta^{0}_z \left(  .,s\right) \rangle ds - \int_0^{+\infty}  \langle  a \left(  .,\xi_{\frac{s}{\eps^\delta}} \right) \widehat \beta^{0}_z \left(  .,s\right) \rangle ds \right| \leq Ce^{-\mathfrak{k} t/\eps^2}$$
Let us consider for a fixed $T > 0$
$$\int_0^T \langle  (a \left(  .,\xi_{\frac{s}{\eps^\delta}} \right) - \bar a(.)) \widehat \beta^{0}_z \left(  .,s\right) \rangle ds = \eps^\delta \int_0^{T/\eps^\delta}  \langle  (a \left(  .,\xi_{s} \right) - \bar a(.)) \widehat \beta^{0}_z \left(  .,\eps^\delta s\right) \rangle ds.$$
Our assumption {\bf (A)} implies that $\xi$ satisfies a strong mixing condition (see \cite{vere:97}). Thus from the ergodic theorem, this quantity converges a.s. to zero (see \cite{lipt:shir:89}, chapter 4 or \cite{chec:piat:sham:07}, chapter 1). Moreover the rate of convergence is of order $\eps^{\delta/2} = \eps^{\al/2-1}$, that is for any $\gamma > 0$ the following quantity
\begin{eqnarray*}
&& \eps^{-\delta/2 +\gamma} \int_0^T \langle  (a \left(  .,\xi_{\frac{s}{\eps^\delta}} \right) - \bar a(.)) \widehat \beta^{0}_z \left(  .,s\right) \rangle ds \\
&& \quad =\eps^{\gamma} \left[\eps^{\delta/2} \int_0^{T/\eps^\delta}  \langle  (a \left(  .,\xi_{s} \right) - \bar a(.)) \widehat \beta^{0}_z \left(  .,\eps^\delta s\right) \rangle ds \right]
\end{eqnarray*}
tends to zero in probability as $\eps$ goes to zero. Indeed it is a consequence of the central limit theorem (implied by our assumption {\bf (A)} and the mixing property, see \cite{lipt:shir:89}, chapter 9) together with Slutsky's theorem. To finish the proof we have:
\begin{eqnarray*}
&& \left|\int_0^\infty \langle a \left(  .,\xi_{\frac{s}{\eps^\delta}} \right) \widehat \beta^{0}_z \left(  .,s\right) \rangle ds -\int_0^\infty \langle \bar a(.) \widehat \beta^{0}_z \left(  .,s\right) \rangle ds \right|\\
&& \quad \leq \left|\int_0^T \langle  (a \left(  .,\xi_{\frac{s}{\eps^\delta}} \right) - \bar a(.)) \widehat \beta^{0}_z \left(  .,s\right) \rangle ds \right| \\
&& \qquad +  \left|\int_T^\infty \langle  (a \left(  .,\xi_{\frac{s}{\eps^\delta}} \right) - \bar a(.)) \widehat \beta^{0}_z \left(  .,s\right) \rangle ds \right|.
\end{eqnarray*}
The first part converges a.s. to zero when $\eps$ tends to zero (with a rate of convergence of order $\eps^{\delta/2}$ in probability) to a fixed $T$, whereas the second part converges to zero when $T$ tends to $+\infty$ in any $\bL^p(\Omega)$.

Then by recursion we can complete the proof of the lemma.
\end{proof}

\vspace{0.5cm}

We introduce again the constant $k$ in all functions. Since $\cI_1= 0$, gathering all previous Lemmata, we deduce that the expansion \eqref{eq:expansion_varrho} of $\varrho^{1,\eps}$ can be written:
\begin{eqnarray*}
\varrho^{1,\eps}(x,t) & = &  \beta^{1,0,\eps} \left(  \frac{x}{\eps},\frac{t}{\eps^2} \right) \partial^{1}_x u^{0} \left(x,t\right)  \\ \nonumber
& + & \eps \left[ \widehat m^{1,0,\eps} \left( \frac{t}{\eps^2} \right) + \beta^{1,1,\eps} \left(  \frac{x}{\eps},\frac{t}{\eps^2} \right) \right] \partial^{2}_x u^{0} \left(x,t\right) + \Gamma^{1,1,\eps}(x,t) ,
%\\
%& = & \eps \left[ \int_0^\infty \langle  \bar a \left(  . \right) \widehat \beta^{0}_z \left(  .,s \right) \rangle ds   \right] \partial^{2}_x u^{0} \left(x,t\right) + \cO(\eps^\nu)
\end{eqnarray*}
where $\Gamma^{1,1,\eps}= \cO(\eps^\nu)$ (which means $\eps^\nu$ times some bounded term) with $\nu > \al/2-1$ and $\beta^{1,0,\eps}$ and $\beta^{1,1,\eps}$ converge exponentially fast to zero.
Now let us come again to
\begin{eqnarray*}
\rho^{1,\eps} (x,t) & = & -\eps^{1-\al/2} \varrho^{1,\eps} \\
& = & -\eps^{2-\al/2}  \widehat m^{1,0,\eps} \left( \frac{t}{\eps^2} \right) \partial^{2}_x u^{0} \left(x,t\right) + \cO(\eps^{\nu+1-\al/2}).
\end{eqnarray*}
We have proved that $L^2(\R \times (0,T))$ norm of the remainder $\rho^{1,\eps}$ converges  to zero in probability.

\vspace{0.5cm}
Now the preceding results imply that the expansion \eqref{eq:expansion_varrho} of $\varrho^{k,\eps}$
\begin{eqnarray*} %\label{eq:expansion_varrho}
&& \varrho^{k,\eps}(x,t) =  \beta^{k,0,\eps} \left(  \frac{x}{\eps},\frac{t}{\eps^2} \right) \partial^{k}_x u^{0} \left(x,t\right)  \\ \nonumber
&& \quad + \sum_{\ell = 1}^{J_1 - k } \eps^\ell \left[ \widehat m^{k,\ell-1,\eps} \left(  \frac{t}{\eps^2} \right)  + \beta^{k,\ell,\eps} \left(  \frac{x}{\eps},\frac{t}{\eps^2} \right) \right] \partial^{k+\ell}_x u^{0} \left(x,t\right) + \Gamma^{J_1-k,\eps}(x,t).
\end{eqnarray*}
can be written:
\begin{eqnarray*} %\label{eq:expansion_varrho}
&& \varrho^{\eps}(x,t) =\cI_k \partial^{k}_x u^{0} \left(x,t\right)  \\ \nonumber
&& \quad + \sum_{\ell = 1}^{J_1 - k } \eps^\ell \left[ \int_0^\infty \langle  \bar a \left(  . \right) \widehat \beta^{k,\ell-1}_z \left(  .,s \right) \rangle ds   \right] \partial^{k+\ell}_x u^{0} \left(x,t\right) + \mathfrak{r}^\eps(x,t).
\end{eqnarray*}
Again the $L^2(\R \times (0,T))$ norm of the remainder $\mathfrak{r}^\eps$ converges  in probability to zero. %Indicating the dependence w.r.t. $k$ again in the previous expansion, we can remark that $\widehat \beta^{\ell-1} = \widehat \beta^{k,\ell-1}$.
Denote by
$$\mathfrak{C}_{k,\ell-1}  = \int_0^\infty \langle  \bar a \left(  . \right) \widehat \beta^{k,\ell-1}_z \left(  .,s \right) \rangle ds $$
and note that $\mathfrak{C}_{k,\ell-1}$ depends only on $\cI_0=1,\cI_1=0,\ldots,\cI_{k-1}$.
Therefore we obtain, up to some negligible term of order $\cO(\eps^\nu)$:
\begin{eqnarray*}
\rho^\eps (x,t) & = &- \sum_{k=1}^{J_1}  \eps^{k-\al/2} \left[ \cI_k \partial^{k}_x u^{0} \left(x,t\right)  + \sum_{\ell = 1}^{J_1 - k} \eps^{\ell} \mathfrak{C}_{k,\ell-1} \partial^{k+\ell}_x u^{0} \left(x,t\right)  \right] \\
& = & - \sum_{m=2}^{J_1}  \eps^{m-\al/2} \left[\cI_m+\sum_{\ell = 1}^{m-1} \mathfrak{C}_{m-\ell,\ell-1} \right]  \partial^{m}_x u^{0} \left(x,t\right) .
\end{eqnarray*}
Then for $k\geq 2$, if we choose
\begin{equation*} %\label{def:I_k_al>2}
\cI_k = -  \sum_{\ell = 1}^{k-1} \mathfrak{C}_{k-\ell,\ell-1}  =- \int_0^\infty \langle  \bar a \left(  . \right) ( \sum_{\ell = 1}^{k-1} \widehat \beta^{k-\ell,\ell-1}_z ) \left(  .,s \right) \rangle ds
\end{equation*}
provided this sequence is well defined, we deduce that $\rho^\eps (x,t) = \cO(\eps^\nu)$.

To complete the proof we need to show that the sequence $\cI_k$ for $k\geq 2$ is well-defined. We have
\begin{equation*} %\label{eq:def_C_k}
\cI_k =- \int_0^\infty \langle  \bar a \left(  . \right) \mathfrak{B}^k_z \left(  .,s \right) \rangle ds
\end{equation*}
with
$$\mathfrak{B}^k =  \sum_{\ell = 1}^{k-1} \widehat \beta^{k-\ell,\ell-1}.$$
Thus $\cI_k$ is wellposed if $\mathfrak{B}^k$ only depends on $\cI_0,\cI_1,\ldots,\cI_{k-1}$. But for $k=2$
$$\cC_2 = -  \mathfrak{C}_{1,0} =- \int_0^\infty \langle  \bar a \left(  . \right) \widehat \beta^{1,0}_z \left(  .,s \right) \rangle ds$$
and $\widehat \beta^{1,0}$ depends only on $\chi^0$.
Then the function $\mathfrak{B}^k$ satisfies the equation
\begin{equation*}  %\label{eq:B_k}
\partial_t \mathfrak{B}^k  = (\bar a(z)\mathfrak{B}^k_z)_z + \cH^{k}(z,t)
\end{equation*}
with initial value
$$\mathfrak{B}^k(z,0) = \widehat \beta^{k-1,0}(z,0) = \sum_{n=1}^{k-1} \cI_{k-1-n} \chi^{n-1}\left( z\right)$$
and with
$$\cH^{k}(z,t)  = \sum_{\ell=1}^{k-1}\overline{\phi^{k-\ell,\ell-1}}(z,t).$$
We can prove by recursion that $\phi^{k-\ell,\ell-1}$ only depends on $\cI_2,\ldots,\cI_{k-\ell-1}$, which leads to the well-posedness on $\cI_k$. Finally we obtain:
\begin{proposition} \label{prop:rest_init_cond}
There exists a sequence $(\cI_k, \ k\geq 2)$ such that $\bL^2(\R \times (0,T))$ norm of the residual $\rho^\eps$ converges to zero  in probability.
\end{proposition}

\bigskip\noindent
{\bf Acknowledgements.}  The work of the second author was partially supported by Russian Science Foundation, project number 14-50-00150.

\end{document}